\documentclass[a4paper,11pt,reqno]{amsart}
\usepackage{amssymb,amsthm}
\usepackage{ifthen}
\usepackage{amsmath}
\usepackage{color}
\usepackage{stmaryrd}
\usepackage{bm}
\usepackage{cleveref}
\usepackage{pgfplots}
\usepackage{graphicx}

\usepackage{algorithm}
\usepackage{algpseudocode}

\usepackage{listings}
\definecolor{mygreen}{rgb}{0,0.6,0}
\definecolor{mygray}{rgb}{0.5,0.5,0.5}
\definecolor{mymauve}{rgb}{0.58,0,0.82}
\newtheorem{lemma}{Lemma}[section]

\newcommand{\eps}{\varepsilon}
\newcommand{\W}{\ensuremath{\mathcal{W}}}
\newcommand{\T}{\ensuremath{\mathcal{T}}}
\newcommand{\E}{\ensuremath{\mathcal{E}}}
\newcommand{\tr}[1]{\ensuremath{\,\mathrm{tr}(#1)}}

\newcommand{\Div}[2][]{\ensuremath{\,\mathrm{div}^{#1}(#2)}}

\newcommand{\mat}[1]{\ensuremath{\begin{pmatrix}#1\end{pmatrix}}}
\newcommand{\Av}[1]{\ensuremath{\{#1\}}}
\newcommand{\bkappa}{\ensuremath{\bm{\kappa}}}
\newcommand{\bsigma}{\ensuremath{\bm{\sigma}}}
\newcommand{\bI}{\ensuremath{\bm{I}}}
\newcommand{\Surf}{\ensuremath{\mathcal{S}}}

\newcommand{\bA}{\ensuremath{\bm{A}}}

\newcommand{\bX}{\ensuremath{\bm{X}}}
\newcommand{\bnu}{\ensuremath{\bm{\nu}}}
\newcommand{\VT}{{\mathbf{T}}}
\newcommand{\Tt}{\ensuremath{\VT_t}}
\newcommand{\pTt}{\ensuremath{\partial\bm{T}_t}}
\newcommand{\bP}{\ensuremath{\bm{P}}}
\newcommand{\bPsi}{\ensuremath{\bm{\Psi}}}
\newcommand{\bu}{\ensuremath{\bm{u}}}

\newcommand{\CC}{\ensuremath{C_c}}

\newcommand{\Pol}{\ensuremath{\mathcal{P}}}
\newcommand{\VR}{{\mathbf{R}}}
\newcommand{\VX}{{\mathbf{X}}}
\newcommand{\taubf}{\boldsymbol{\tau}}
\newcommand{\nubf}{\boldsymbol{\nu}}
\newcommand{\mubf}{\boldsymbol{\mu}}
\providecommand{\Id}{\Op{Id}}
\newcommand*{\Op}[1]{\mathsf{#1}} % Operators
\newcommand{\Vu}{{\mathbf{u}}}
\newcommand{\VA}{{\mathbf{A}}}
\newcommand{\VV}{{\mathbf{V}}}
\newcommand{\VW}{{\mathbf{W}}}
\newcommand{\Vx}{{\mathbf{x}}}
\newcommand{\Va}{{\mathbf{a}}}
\newcommand{\Vb}{{\mathbf{b}}}
\newcommand{\Vc}{{\mathbf{c}}}

\numberwithin{equation}{section}
\counterwithin{figure}{section}
\title{Numerical shape optimization of the Canham-Helfrich-Evans bending energy}

\author{Michael Neunteufel} 
\address[Michael Neunteufel]{Institute of Analysis and Scientific Computing, TU Wien, Wiedner Hauptstrasse 8-10, 1040 Wien, Austria.}
\email{michael.neunteufel@tuwien.ac.at}
\urladdr{https://www.asc.tuwien.ac.at/~schoeberl/wiki/index.php/Michael\_Neunteufel}

\author{Joachim Sch\"oberl} 
\address[Joachim Sch\"oberl]{Institute of Analysis and Scientific Computing, TU Wien, Wiedner Hauptstrasse 8-10, 1040 Wien, Austria.}
\email{joachim.schoeberl@tuwien.ac.at}
\urladdr{https://www.asc.tuwien.ac.at/~schoeberl/wiki/index.php/Joachim\_Sch{\"o}berl}

\author{Kevin Sturm} 
\address[Kevin Sturm]{Institute of Analysis and Scientific Computing, TU Wien, Wiedner Hauptstrasse 8-10, 1040 Wien, Austria.}
\email{kevin.sturm@tuwien.ac.at}
\urladdr{https://www.asc.tuwien.ac.at/sturm/}

\date{\today}
\begin{document}
\maketitle

\begin{abstract}
In this paper we propose a novel numerical scheme for the Canham--Helfrich--Evans bending energy based on a three-field lifting procedure of the distributional shape operator to an auxiliary mean curvature field. Together with its energetic conjugate scalar stress field as Lagrange multiplier the resulting fourth order problem is circumvented and reduced to a mixed saddle point problem involving only second order differential operators. Further, we derive its analytical first variation (also called first shape derivative), which is valid for arbitrary polynomial order, and discuss how the arising shape derivatives can be computed automatically in the finite element software NGSolve. We finish the paper with several numerical simulations showing the pertinence of the proposed scheme and method.\\

%a procedure implemented in NGSolve to automatically compute the arising shape derivatives. We finish the paper with several numerical simulations showing the pertinence of the proposed scheme and method.\\
%which is due the special jump term of the lifting method fairly involved

\noindent
\textbf{\textit{Keywords:}} Canham-Helfrich-Evans bending energy, shape optimization, mixed finite element method, distributional curvature, biomembrane.\\

\noindent
\textbf{\textit{MSC2020:}} 65N30, 65K10, 53E40, 49M05, 74K15
%65N30 (FEM boundary value problem involving PDE), 65K10 (numerical optimization and variational techniques), 53E40 (high-order geometric flows), 49M05 (Numerical methods in optimal control), 74K15 (Mechanics of deformable solids; Thin bodies, structures; Membranes)
\end{abstract}
\section{Introduction}
\label{sec:introduction}
In this paper we study the numerical minimization of the Canham-Helfrich-Evans \cite{CANHAM70,Evans74,a_HE_1973a} bending energy
\begin{align}
\W(\partial \Omega)= 2\kappa_b\int_{\partial\Omega}(H-H_0)^2\,ds, \quad \Omega\subset \VR^3  \text{ bounded domain},\label{eq:canham_helfrich_energy}
\end{align}
subject to the following volume and area constraints
\begin{align}
|\Omega| = V_0, \qquad |\partial \Omega| = A_0, \label{eq:constraints}
\end{align}
where the positive constants $V_0$, $A_0>0$  obey the isoperimetric inequality
\begin{align}
V_0 \leq \frac{A_0^{\frac{3}{2}}}{6\sqrt{\pi}}.
\end{align}
Here $H:=\frac{1}{2}(\kappa_1+\kappa_2)$ denotes the mean curvature of $\partial\Omega$, $\kappa_1$ and $\kappa_2$ its principal curvatures, $2H_0$ the so-called spontaneous curvature ($H_0$ is half the spontaneous curvature), and $\kappa_b$ a bending elastic constant. Henceforth we will use the abbreviation $\Surf:=\partial\Omega$ keeping in mind that $\Surf$ is the surface enclosing the volume $\Omega$. The Energy \eqref{eq:canham_helfrich_energy} was proposed to model membranes such as vesicles and red-blood cells \cite{a_HE_1973a,ME94}. The numerical treatment of this problem is not straight-forward since the computation 
of the mean curvature typically involves the Laplace-Beltrami operator of the normal field, which would involve fourth order derivatives of the surface coordinates and thus requires a certain smoothness of the (discretized) surface. Typically shapes are approximated with continuous, non-smooth triangulations (mostly linear or quadratic ones) leading to the fundamental and non-trivial question of computing/approximating the appropriate curvature.
Nevertheless, several approaches to tackle this problem have been proposed. For a recent comprehensive review of the Canham-Helfrich-Evans energy including various numerical approaches we refer to \cite{GG17}. 

A variety of methods are based on the approximation of the 
Laplace-Beltrami operator by means of discrete differential geometry (DDG) \cite{MDSB03,GGRZ06,WBHZG07}. The minimization is then achieved by e.g., differentiating the discrete energy with respect to the 
nodes and follow a negative gradient, see \cite{a_BILIKO_2020a} for a comparison with several established numerical approximation schemes minimizing the Canham-Helfrich-Evans energy. A popular discretization scheme for the Laplace-Beltrami operator is the finite difference cotangent method on a Voronoi area entailing also a direct computation of a possibly involved Gaussian curvature in terms of the angle deficit, used e.g. in \cite{BFM11,BLJ11,SHKL16}.

The shape derivative of geometric quantities and the full Canham-Helfrich-Evans energy has been computed, e.g., in \cite{CMP17,DN12,LMS10,Walker15} involving fourth order derivatives and the Gauss curvature of the shape.
Beside boundary integral methods \cite{Poz92,VRBZ11,FBM14}, procedures based on surface finite element methods (SFEM) \cite{DE13,ESV12} approximate the surface of the shape with (possible high-order curved) isoperimetric elements. For linear triangulation the discrete normal vectors are sometimes averaged giving the possibility of computing the weak gradient globally \cite{BGN08}. For higher polynomial orders, however, the shape derivative yields complicated expressions due to the (nonlinear) averaging procedure. To avoid $C^1$-conforming elements the mean curvature $H$ or mean curvature vector $\bm{H}=H\nubf$ gets introduced as independent field and the equations are rewritten in such a way that no expressions in terms of the normal vector $\nubf$ in strict sense are left, \cite{Rusu05,Dziuk08,BGN08,BNS10}.

Using smooth approximations of the surface by, e.g.,  high-order B-splines or sub-division algorithms has been recently investigated in \cite{SDMS17,TMA19}, circumventing a non-continuous normal vector field.

Level set and phase field approaches \cite{DLW04,MMPR12,LSM14} discretize the full space and the surface gets represented implicitly by a level set function. On the one hand geometric quantities as the normal vector are therefore easier accessible and changes of the shape's topology are allowed, but on the other hand full-space computations have to be performed.

Mostly, instead of a quasi-static procedure a time-stepping algorithm with possible damping and stabilization techniques are used to find stationary solutions or for dynamic tests. For evolutionary geometries the famous time-stepping algorithm of Dziuk \cite{Dziuk90} is frequently considered and has been firstly extended to Willmore flows in \cite{Rusu05}. Recently Dziuk's algorithm has been further developed for mean curvature and Willmore flows in \cite{KLL19,KLL20}, where convergence has been rigorously proven.

In this work, we propose a novel discretization approach based on a lifting procedure of the distributional (mean) curvature to a more regular auxiliary curvature field (not to be confused with the lifting from the discrete to the exact surface in the sense of \cite{DE13}). Besides the classical element-wise shape operator, also the angle of the jump of the normal vector between two adjacent elements is considered as element-boundary integral to describe the full curvature. This has the advantage that we can directly apply the shape derivative to each of the individual terms. A derivation of the involved duality pairing is presented to build a bridge between (distributional based) surface finite elements and DDG, where several formulations rely also on the angle \cite{GGRZ06}. By introducing the scalar-valued mean curvature $\kappa$ as independent unknown, in combination with the corresponding Lagrange multiplier $\sigma$ the fourth order problem is avoided by introducing two second order problems. Further, the method also works for low-order polynomials on affine triangulation as well as for arbitrary polynomially curved elements without changing any term. The gradient based shape optimization algorithm is then applied to several well-established benchmark examples, where the stationary equilibrium shapes of the Canham-Helfrich-Evans energy, including possible spontaneous curvature, are computed.

The highlights of our paper are:

\begin{itemize}
	\item novel numerical scheme to discretize the Canham-Helfrich-Evans bending energy based on a lifting of the distributional shape operator
	\item derivation of distributional curvature in context of FEM
	\item rigorous computation of first variation of the discretized bending energy
	\item numerical minimization of the bending energy using gradient-type algorithm using the first variation
\end{itemize}

\section{Notation and Problem statement}
\label{sec:notation_prob_statement}
In what follows, we will denote  by $\Surf$ a smooth $d-1$-dimensional closed submanifold in $\VR^d$, $d=2,3$, and by $\Omega\subset\VR^d$ the enclosed volume, i.e., $\Surf = \partial \Omega$ is the topological boundary of $\Omega$. We say that a function $f:\Surf \to \VR^d$ is $k$-times differentiable if there exists a neighborhood $U\subset\VR^d$ of $\Surf$ and a $k$-times differentiable function $\tilde f:U \to \VR^d$, such that, $\tilde f = f$ on $\Surf$. Given a differentiable function 
$f:\Surf \to \VR^d$ and an extension $\tilde f$, we define the tangential Jacobian and gradient of a function $f:\Surf\to\VR^d$ by
\begin{align}
\partial^{\Surf}f:=\partial \tilde f \bP_{\Surf}, && \nabla^{\Surf} f:= \bP_{\Surf}\partial \tilde f^\top.
\end{align}
Here, $\bP_{\Surf}:= \bI -\nubf\otimes\nubf:\,\VR^d\to T\Surf:=\cup_{p\in \Surf} T_p\Surf$ denotes the projection onto the tangent bundle of $\Surf$, with $(\bm{a}\otimes \bm{b})\bm{c}:= (\bm{b}\cdot \bm{c})\bm{a}$ for $\bm{a},\bm{b},\bm{c}\in \VR^d$ being the outer product, and $\nubf$ denotes the outward pointing normal vector field along $\partial \Omega$. Further, we will neglect the subscript for the Euclidean norm $\|\cdot\|_2$ and denote in three dimensions the vector cross product by $\bm{a}\times \bm{b}$.

For the discretization, let $\T_h$ be a piecewise smooth and globally continuous surface approximating $\Surf$. More precisely, let $\T_h=\{T_i\}_{i=1}^N$ with $T_i$ smooth manifolds and piecewise smooth boundary $\partial T_i$ and the vertices of $\T_h$ lie on $\Surf$. Denote by $TT_i$ the tangent bundle of the smooth manifold $T_i$. We define $T\T_h:=\cup_i TT_i$ as the discrete tangent bundle of $\T_h$ and $\bm{P}_{\T_h}:\VR^d\to T\T_h$ the corresponding projection onto the discrete tangent bundle. In 3D, on the edges we can define (normalized) tangential vectors $\taubf_L$ and $\taubf_R$ such that the co-normal (element-normal) vectors $\mubf_L:=\nubf_L\times \taubf_L$ and $\mubf_R:=\nubf_R\times \taubf_R$ are pointing outward of $T_L$ and $T_R$, respectively, see Figure \ref{fig:nv_tv_env}. We will neglect the subscripts $L$ and $R$ if the corresponding element $T$ is obvious. Integrating over volume, boundary, or edges (vertices in 2D) is denoted by $dx$, $ds$, or $d\gamma$, respectively.

\begin{figure}[h]
	\centering
	\includegraphics[width=0.27\textwidth]{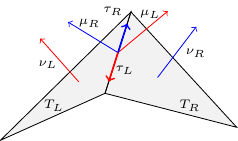}
	\caption{Normal, tangential, and co-normal (element-normal) vectors $\nubf$, $\taubf$, and $\mubf$ on two neighboured elements.}
	\label{fig:nv_tv_env}
\end{figure}

We incorporate the constraints \eqref{eq:constraints} in a weak sense using a penalty formulation, with $c_A, c_V>0$ denoting the penalty parameters:
\begin{equation}
\mathcal{J}(\Surf) = \W(\Surf) + c_A\underbrace{(|\Surf|-A_0)^2}_{=:J_{\mathrm{surf}}(\Surf)} +  c_V\underbrace{(|\Omega|-V_0)^2}_{=:J_{\mathrm{vol}}(\Omega)}.\label{eq:min_prob_constr}
\end{equation}
Other approaches such as  (augmented) Lagrangian \cite{NW06} are also possible.

Later, in Section~\ref{sec:solving_algo}, we present a procedure to improve surface area preservation if the initial shape already has the desired area.

\section{Curvature computation}
\label{sec:curvature_comp}
In this section we derive the discrete shape operator in terms of distributions, related to discrete differential geometry involving the angle of the normal vector jump between two adjacent elements. Then a variational formulation for computing the curvature is presented and further tailored to the problem of lifting only the mean curvature $H$ instead of the full shape operator. Finally, the corresponding perturbed problem is derived as preparation for the shape derivatives in Section~\ref{sec:shape_der}.

\subsection{Discrete shape operator}
\label{subsec:discrete_shape_op}
Given the shape operator $-\partial^{\Surf}\bnu:T\Surf\times T\Surf\to\VR$, also called the Weingarten tensor, on a smooth $d-1$-dimensional submanifold, the mean curvature $H=\frac{1}{d-1}\sum_{i=1}^{d-1}\kappa_i$, with $\kappa_i$ denoting principal curvatures, is computed by $\frac{1}{d-1}$  times the trace of $-\partial^{\Surf}\bnu$
\begin{align}
H = -\frac{1}{d-1}\tr{\partial^{\Surf}\bnu}.
\end{align}

Let now $\T_{h,k}$ be a triangulation of $\Surf$, where the elements $T\in \T_{h,k}$ are curved of polynomial degree $k\ge 1$ to fit the exact surface. For procedures curving the mesh appropriately for optimal isoparametric finite element we refer to \cite{Len86,DE13,Deml09}. In this work we use a projection-based interpolation procedure for curving geometries described in \cite{Demk04}. For ease of presentation we will also simply write $\T_{h}$. Given a triangulation, we define the skeleton $\E_{h,k}$ of $\T_{h,k}$ as the set of all edges or vertices of $\T_{h,k}$ in 3D or 2D, respectively. The set of all polynomials up to  order $\ell\ge 0$ on the triangulation $\T_{h,\ell}$ and - in three dimensions - skeleton $\E_{h,\ell}$ is denoted by $\Pol^\ell(\T_h)$ and $\Pol^\ell(\E_h)$, respectively.

For an affine triangulation $\T_{h,1}$ the discrete outer normal vector $\nubf$ is constant on each facet of $\T_{h,1}$ (i.e., piecewise constant) and thus, $\partial^{\Surf}\bnu|_T=0$ for all $T\in \T_{h,1}$. Moreover the normal vector may jump over the interfaces, see Figure \ref{fig:nv_jump}. Hence, the shape operator, which we refer to as discrete shape operator, can at best be a distribution and will be defined below. Our definition is also motivated by discrete differential geometry, e.g. \cite{GGRZ06}, where the angle is also used as part of the curvature computation.

\begin{figure}[h]
	\centering
	\includegraphics[width=0.17\textwidth]{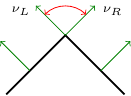}
	\caption{Jump of normal vector over two affine elements.}
	\label{fig:nv_jump}
\end{figure}

To illustrate the idea we start with a polygon curve in two dimensions and without loss of generality consider two line segments, denoted by $\hat{T}_L$ and $\hat{T}_R$, and one point $\bP=(0,0)$, where the normal vector jumps with angle $\alpha>0$ as depicted in Figure~\ref{fig:constr_approx} (a). Our goal is to derive an approximation of the curvature formula at the point $\bP$. To this end, we construct a family of $C^1$-smooth approximation of the curve parameterized by $\varepsilon>0$ sufficiently small depending on the triangulation. Starting with an $\varepsilon$-circle centered at $\bP$, we define the unique circle that goes through the same intersection points with the curve as the $\varepsilon$-circle and intersects it in a 90 degree angle, see Figure \ref{fig:constr_approx} (b). This circle with radius $r_{\varepsilon}=\varepsilon\frac{1+\cos(\alpha)}{\sin(\alpha)}$ and midpoint $\bm{M}_{\varepsilon}=(\varepsilon,-r_{\varepsilon})$ is then used as $C^1$-approximation of the junction. To be precise, the resulting curve $\T_{\varepsilon}$ consists of the remaining line segment parts $T_L$ and $T_R$ together with the connecting circle segment $T_{\varepsilon}$, Figure \ref{fig:constr_approx} (c). Thus, we can define the continuous and piecewise smooth approximated normal vector $\nubf_{\varepsilon}:\T_{\varepsilon}\to\mathbb{S}^1$ by
\begin{align}\label{eq:regu_nu_eps}
\nubf_{\varepsilon}(\bm{x})=\begin{cases}
\nubf_L &\text{ for }\bm{x}\in T_L,\\
\frac{\bm{x}-\bm{M}_{\varepsilon}}{\|\bm{x}-\bm{M}_{\varepsilon}\|} &\text{ for }\bm{x}\in T_{\varepsilon},\\
\nubf_R & \text{ for } \bm{x}\in T_R.
\end{cases} 
\end{align}
Let us now calculate the shape operator of the regularized surface. We first notice that $\|\nubf_{\varepsilon}\|=1$ and thus 
$\partial^{\Surf} \nubf_\varepsilon = \partial \nubf_\varepsilon$. Therefore fixing $\bm{x}$ near $T_\varepsilon$ we compute the $(i,j)$th entry of $\partial \nubf_\varepsilon(\bm{x})$, $\bm{x}\in T_{\varepsilon}$:
\begin{equation}
\partial_{x_i} (\nubf_\varepsilon)_j(\bm{x}) = \frac{1}{\|\bm{x}-\bm{M}_\varepsilon\|}\left( \delta_{ij} - \frac{1}{\|\bm{x}-\bm{M}_\varepsilon\|^2}( (\bm{x}-\bm{M}_\varepsilon)\cdot \bm{e}_i (\bm{x}-\bm{M}_\varepsilon)\cdot \bm{e}_j)  \right),
\end{equation}
where $\delta_{ij}$ denotes the Kronecker delta and $\bm{e}_i$ the $i$th unit-vector, $\bm{e}_i(j)=\delta_{ij}$. This can equivalently be written as 
\begin{equation}
\partial \nubf_\varepsilon(\bm{x}) = \frac{1}{r_{\varepsilon}}\mubf_{\varepsilon}\otimes \mubf_{\varepsilon}, \qquad \mubf_{\varepsilon}:=\frac{1}{r_{\varepsilon}}\mat{-(x_2-M_{\varepsilon,2})\\ x_1-M_{\varepsilon,1}}.\label{eq:reg_shape_op}
\end{equation}
%where $\bI_2$ denotes the identity matrix in two dimensions.

Note that $\mubf_{\varepsilon}=-\mubf_{L}$ and $\mubf_{\varepsilon}=\mubf_{R}$ on the interfaces $\overline{T}_{\varepsilon}\cap \overline{T}_L$ and $\overline{T}_{\varepsilon}\cap \overline{T}_R$, where $\mubf_L$ and $\mubf_R$ are the co-normal vectors, cf. Figure~\ref{fig:nv_tv_env}. Further, there exists a continuous and bijective mapping $\Phi_{\varepsilon}:\T_h\to\T_{\varepsilon}$   given by 
\begin{align}
\Phi_{\varepsilon}(\bm{x}):=\begin{cases}
\bm{x} &\text{ for } \bm{x}\in \T_h\backslash U_{\varepsilon}(\bP),\\
\bm{M}_{\varepsilon}+\frac{r_{\varepsilon}}{\|\bm{x}-\bm{M}_{\varepsilon}\|}(\bm{x}-\bm{M}_{\varepsilon})&\text{ for } \bm{x}\in \T_h \cap U_{\varepsilon}(\bP),
\end{cases}
\end{align}
with $\Phi_{\varepsilon}\overset{\varepsilon\to 0}{\longrightarrow}\Id$.

\begin{figure}[h]
	\centering
	\begin{tabular}{ccc}
		\includegraphics[width=0.17\textwidth]{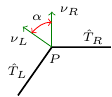}&
		\includegraphics[width=0.17\textwidth]{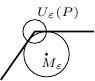}&
		\includegraphics[width=0.17\textwidth]{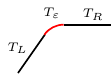}\\
		(a) & (b) & (c)
	\end{tabular}
	\caption{Construction of approximation of discrete jump. (a) The polygon curve with jump angle $\alpha$. (b) The construction of the circles. (c) The final approximated smooth curve.}
	\label{fig:constr_approx}
\end{figure}

To compute the limit $\varepsilon\to0$ we define the corresponding test function on the triangulation $\bPsi:T\T_h\times T\T_h\to\VR$ co-normal--co-normal continuous, i.e., $\bPsi_{\mubf_L\mubf_L} := \mubf_L^\top\bPsi|_{T_L}\mubf_L= \mubf_R^\top\bPsi|_{T_R}\mubf_R =: \bPsi_{\mubf_R\mubf_R} $ on the skeleton $\E_h$. Thus, the co-normal--co-normal component does not ``see'' the junction of the discretized geometry. Further it should be symmetric, as the shape operator is, and thus, is of the form $\bPsi = \Psi \mubf\otimes\mubf$ with $\Psi:\VR^2\to\VR$ a continuous function. The test function on the smoothed surface $\T_{\varepsilon}=\Phi_{\varepsilon}(\T_h)$ reads $\bPsi_{\varepsilon}=\Psi\circ\Phi_{\varepsilon}^{-1}\mubf_{\varepsilon}\otimes\mubf_{\varepsilon}$. Then, in view of $\lim_{\eps\searrow 0} \bm{M}_\eps =\bP$ and $\lim_{\eps\searrow 0} r_\eps =0$, a change of variables, and Lebesgue dominated convergence Theorem yield
\begin{align*}
\langle -\partial^{\Surf}\bnu,\bPsi\rangle_{\T_{h,1}} &:=\lim\limits_{\varepsilon\rightarrow 0} (-\partial^{\Surf}\bnu_{\varepsilon},\bPsi_{\varepsilon})_{L^2(\T_{\varepsilon})}\\
& \stackrel{\eqref{eq:reg_shape_op}}{=} \lim\limits_{\varepsilon\rightarrow 0}\int_{\Gamma_{\varepsilon}=\Phi_{\varepsilon}(\Gamma_1)}-\frac{1}{r_{\varepsilon}}\mubf_{\varepsilon}\otimes\mubf_{\varepsilon}:\Psi\circ\Phi_{\varepsilon}^{-1}\mubf_{\varepsilon}\otimes\mubf_{\varepsilon}\,ds\\
&= \lim\limits_{\varepsilon\rightarrow 0}\int_{\Gamma_{1}}-\Psi(\bm{M}_{\varepsilon}+r_{\varepsilon} \bm{x})\,ds = -\underbrace{|\Gamma_1|}_{=\alpha}\Psi(\bP)\\
& = -\int_{\bP} \sphericalangle(\nubf_L,\nubf_R)\Psi\,d\gamma,
\end{align*}

with the notation $\sphericalangle(\nubf_L,\nubf_R):=\arccos(\nubf_L\cdot\nubf_R)$ and
\begin{align}
\Gamma_{\varepsilon} := \left\{\bm{M}_{\varepsilon}+r_{\varepsilon}\mat{\cos(s)\\\sin(s)},\,s\in (\frac{\pi}{2},\frac{\pi}{2}+\alpha)\right\}, && \Gamma_{1} := \left\{\mat{\cos(s)\\\sin(s)},\,s\in (\frac{\pi}{2},\frac{\pi}{2}+\alpha)\right\}.
\end{align} 

Thus, on a general affine triangulation $\T_{h,1}$ the discrete shape operator reads
\begin{align}
\langle -\partial^{\Surf}\bnu,\bPsi\rangle_{\T_{h,1}}=-\sum_{E\in\E_h}\int_E\sphericalangle(\nubf_L,\nubf_R)\bPsi_{\mubf\mubf}\,d\gamma,\label{eq:discr_shape_op_1d_lo}
\end{align}
for all $\bPsi \in \bm{\Sigma}_h:=\{ \Psi\,\mubf\otimes\mubf\,:\,\Psi:\Surf\to\VR \text{ continuous}\}$.
In the curved case $\T_{h,k}$, $k>1$, the jump of the normal vector across elements might be smaller, but still be present. In terms of the two co-vectors $\mubf$ of these points we can apply the above procedure to obtain the distributional point curvatures, whereas away from the interfaces the element-wise classical shape operator can be applied leading to the formula
\begin{align}
\langle -\partial^{\Surf}\bnu,\bPsi\rangle_{\T_{h,k}}=-\sum_{T\in\T_h}\int_T\partial^{\Surf}\bnu|_T:\bPsi\,ds -\sum_{E\in\E_h}\int_E\sphericalangle(\nubf_L,\nubf_R)\bPsi_{\mubf\mubf}\,d\gamma,\label{eq:discr_shape_op_1d}
\end{align}
for all $\bPsi \in \bm{\Sigma}_h$.
%Note, that for a high order approximation of the surface the jump term becomes less important in terms of curvature information, however, it will be crucial for numerical stability. 
The derivation in the affine case can be related to Steiner's offset formula \cite{Stein40}, where the surface is shifted along the discrete normal vector. Then, depending on how the appearing gap between the elements is filled, different expressions in terms of the angle are gained \cite{BPW10}, which are equivalent in the limit of vanishing angle $\alpha\to0$.

\begin{figure}[h]
	\centering
	\begin{tabular}{cc}
		\includegraphics[width=0.18\textwidth]{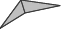}\hspace*{2cm}&
		\includegraphics[width=0.18\textwidth]{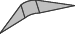}\\
		(a) \hspace*{2cm}& (b)
	\end{tabular}
	\caption{Construction of approximation of discrete jump in 3D. (a) The affine curve with junction. (b) The approximated smooth surface.}
	\label{fig:constr_approx_2d}
\end{figure}

The generalization to two-dimensional sub-manifolds in three dimensions is done in an analogous manner by smoothing the edges with an $\varepsilon$-tube, compare Figure~\ref{fig:constr_approx_2d}. As the discrete normal vector $\bnu$ is square integrable, $\bnu\in [L^2(\T_h)]^3$ its derivative components are in $H^{-1}(\T_h)$, the dual space of $H^1(\T_h):=\{u\in L^2(\T_h):\,\partial^{\Surf} u\in [L^2(\T_h)]^3\}$. Therefore the distributional parts are concentrated only on the edges, but not on the vertices. One could treat the vertex contributions by an $\varepsilon$-sphere, compute the shape operator, and take the limit $\varepsilon\to0$. However, as the regularized shape operator is of order $\mathcal{O}(\varepsilon^{-1})$ whereas the surface is of $\mathcal{O}(\varepsilon^2)$ the limit is zero. 

Therefore, analogously to the two-dimensional case, the test function  $\bPsi:\T_h\to\VR^{3\times 3}_{\mathrm{sym}}$, $\VR^{3\times 3}_{\mathrm{sym}}$ denoting the set of all $3\times3$ symmetric matrices, has to be co-normal--co-normal continuous. The jump over element interfaces is denoted by $\llbracket \bPsi_{\mubf\mubf}\rrbracket$ and the space of symmetric co-normal--co-normal continuous matrices acting on the tangent space is given by the Hellan--Herrmann--Johnson (HHJ) finite element space, see \cite{Com89} and therein references, mapped on the surface, which is a non-conforming subspace of $H(\mathrm{div\,div})$ \cite{PS11,NS19},
\begin{align}
M_h^k(\T_h):= \{\bsigma_h \in [\Pol^k(\T_h)]^{3\times 3}_{\mathrm{sym}}:\,\llbracket\bsigma_{h,\mubf\mubf}\rrbracket=0, \, \bsigma_h\nubf = \nubf^\top\bsigma_h=0 \}.\label{eq:fespace_hhj}
\end{align}
To construct such a finite element space one can start in the flat two-dimensional case and then map the resulting elements onto the surface by using the so-called Piola transformation to preserve the normal-normal continuity, we refer to \cite{PS11} for an explicit construction and additional properties. 
%We emphasize that the finite element space is related to the function space
%\begin{align}
%& \HDivDiv[\Surf]:= \{\bsigma\in [L^2(\Surf)]^{3\times 3}_{\mathrm{sym}}:\, \Div[S]{\Div[S]{\bsigma}}\in H^{-1}(\Surf),\, \bsigma\nubf = \nubf^\top\bsigma=0 \},
%\end{align}
%where $H^{-1}(\Surf)$ denotes for closed surfaces $\Surf$ the topological dual space of $H^1(\Surf):=\{u\in L^2(\Surf):\,\partial^{\Surf} u\in [L^2(\Surf)]^3\}$.

As will be discussed in the following subsections, the matrix valued curvature tensor is going to be reduced to a scalar quantity representing the mean curvature and thus, we do not get into further details with respect to matrix valued spaces.

\subsection{Variational formulation}
\label{subsec:variational_formulation}
The corresponding variational problem for \eqref{eq:discr_shape_op_1d} computing the lifted discrete Weingarten tensor $\bkappa$ of the distributional curvature reads: Find $\bkappa\in M^k_h(\T_h)$ such that for all $\delta \bkappa \in M^k_h(\T_h)$
\begin{align}
\int_{\T_h} \bkappa:\delta\bkappa \,ds = -\sum_{T\in\T_h}\int_T \partial^{\Surf}\bnu:\delta\bkappa\,ds - \sum_{E\in\E_h}\int_E\sphericalangle(\nubf_L,\nubf_R)\delta\bkappa_{\mubf\mubf}\,d\gamma,
\end{align}
where we use the notation $\int_{\T_h}:=\sum_{T\in \T_h}\int_T$ if the involved fields are in $L^2(\T_h)$. We introduce the averaged normal vector
\begin{align}
\label{eq:av_normal_vector}
\Av{\nubf}:=\frac{\nubf_L+\nubf_R}{\|\nubf_L+\nubf_R\|},
\end{align}
which is independent of the dimension, triangle size, or polynomial order of approximation of the surface.

As discussed in \cite{NS19,Neun21} the jump terms can be reordered yielding
\begin{align}\label{eq:weak_form_normal}
\int_{\T_h} \bkappa:\delta\bkappa \,ds &= -\sum_{T\in\T_h}\left(\int_T \partial^{\Surf}\bnu:\delta\bkappa\,ds + \int_{\partial T}\sphericalangle(\nubf,\Av{\nubf})\delta\bkappa_{\mubf\mubf}\,d\gamma\right)\nonumber\\
&=-\sum_{T\in\T_h}\left(\int_T \partial^{\Surf}\bnu:\delta\bkappa\,ds + \int_{\partial T}\left(\frac{\pi}{2}-\sphericalangle(\mubf,\Av{\nubf})\right)\delta\bkappa_{\mubf\mubf}\,d\gamma\right).
\end{align}
The latter (equivalent) formulation is numerically more stable as the derivative of $\arccos(x)$ has a singularity at $x=1$ and we expect the discrete shapes to be close of being smooth such that $\nubf\cdot\Av{\nubf}\approx 1$ but $\mubf\cdot\Av{\nubf}\approx 0$.

With this preliminary work, the minimization problem together with the constraints \eqref{eq:min_prob_constr} can be described by the Lagrange functional
\begin{equation}
\begin{split}
\label{eq:lag_func_def_curvature}
\mathcal{L}(\T_h,\bkappa,\bsigma) &:= \sum_{T\in\T_h}\bigg(\int_{T}2\kappa_b\left(\frac{1}{2}\tr{\bkappa}-H_0\right)^2 + (\bkappa + \partial^{\Surf}\bnu):\bsigma\,ds \\
& +\int_{\partial T}\left(\frac{\pi}{2}-\sphericalangle(\mubf,\Av{\nubf})\right)\bsigma_{\mubf\mubf}\,d\gamma\bigg)+c_A\,J_{\mathrm{surf}}(\T_h)+c_V\,J_{\mathrm{vol}}(\T_h),
\end{split}
\end{equation}
where $\bsigma\in M^k_h(\T_h)$ is the Lagrange multiplier forcing the ``lifting'' $\bkappa=-\nabla^{\Surf}\bnu$ having the physical meaning of a moment tensor. Lagrangian \eqref{eq:lag_func_def_curvature} can be seen as a three-field formulation involving the shape (or equivalently the displacement), the independent shape operator field $\bkappa$, and the moment tensor $\bsigma$. As discussed in the following section, the formulation can be reduced as only the trace of $\bkappa$ will enter the bending energy of the shape. Therefore there is no need to lift the whole shape operator $-\partial^{\Surf}\bnu$.

\subsection{Reduction for mean curvature}
\label{subsec:reduction_mean_curv}
In \eqref{eq:lag_func_def_curvature} the full shape operator $-\partial^{\Surf}\nubf$ is lifted to $\bkappa$ via the Lagrange parameter $\bsigma$. However, only the trace, $\tr{\bkappa}$, is involved in the final energy. Thus, we would ``waste'' computational effort lifting the deviatoric part of $-\partial^{\Surf}\nubf$. Further, using e.g. the lowest order HHJ elements $M_h^0(\T_h)$ \eqref{eq:fespace_hhj} its trace is only a constant per element. The degrees of freedom of $M_h^0$, however, are constants placed at the edges $\E_h$ of the triangulation. As there are more edges than triangles on a closed surface, $\#E=\frac{3}{2}\#T$, this yields a huge kernel of the trace operator. That means, deformations of the shape can occur producing mean curvature but lying in the kernel of the trace of $\bkappa$ yielding a non-robust formulation.

With this motivation, we adapt \eqref{eq:lag_func_def_curvature} by inserting only test-functions $\bsigma$, which are deviatoric-free, i.e. $\bsigma=\sigma \bP_{\T_h}$, where $\sigma:\T_h\to\VR$. From the co-normal--co-normal continuity of $\bsigma$ we deduce that with $\bsigma_{\mubf\mubf} = \mubf^\top (\sigma\,\bP_{\T_h})\mubf=\sigma$ the reduced field $\sigma$ is continuous and thus in $H^1(\T_h)$. The volume term of \eqref{eq:lag_func_def_curvature} changes to
\begin{align*}
\int_{T}2\kappa_b\left(\frac{1}{2}\tr{\bkappa}-H_0\right)^2 + (\tr{\bkappa} + \tr{\partial^{\Surf}\bnu})\sigma\,ds,
\end{align*}
such that only the trace part of $\bkappa$ needs to be considered, $\bkappa=\kappa\,\bP_{\T_h}$ with $\kappa\in H^1(\T_h)$.

With the $H^1$-conforming Lagrangian finite element space
\begin{align}
V_h^k(\T_h):= \{u\in \Pol^k(\T_h)\,:\, u \text{ continuous on }\T_h \}\subset H^1(\T_h).\label{eq:lagrangian_fe}
\end{align}
the reduced curvature problem reads: Find $\kappa\in V_h^k(\T_h)$ such that for all $\delta\kappa\in V_h^k(\T_h)$
\begin{align}\label{eq:weak_form_mean_curv}
\int_{\T_h} \kappa\,\delta\kappa \,ds =-\sum_{T\in\T_h}\Big(\int_T \tr{\partial^{\Surf}\bnu}\delta\kappa\,ds + \int_{\partial T}\left(\frac{\pi}{2}-\sphericalangle(\mubf,\Av{\nubf})\right)\delta\kappa\,d\gamma\Big).
\end{align}
With the same arguments for $\bsigma$ the reduced final Lagrangian reads for $\kappa,\sigma\in V_h^k(\T_h)$
\begin{align}
\begin{split}
\label{eq:red_lag_func_def_curvature}
\mathcal{L}(\T_h,\kappa,\sigma) &:= \sum_{T\in\T_h}\bigg(\int_{T}2\kappa_b\left(\frac{1}{2}\kappa-H_0\right)^2 + (\kappa + \tr{\partial^{\Surf}\bnu})\sigma\,ds \\
& +\int_{\partial T}\left(\frac{\pi}{2}-\sphericalangle(\mubf,\Av{\nubf})\right)\sigma\,d\gamma\bigg)+c_A\,J_{\mathrm{surf}}(\T_h)+c_V\,J_{\mathrm{vol}}(\T_h).
\end{split}
\end{align}
We note, that the mean curvature (vector) has been added as additional unknown in several works \cite{Rusu05,Dziuk08,BGN08,BNS10}, which can be interpreted as mixed formulation. Therein, however, the mean curvature (vector) has been considered to eliminate all terms involving the normal vector $\bnu$. In this work we still use the normal vector and the additionally involved quantities $\kappa$ and $\sigma$ are only scalar-valued.

In Section~\ref{sec:shape_der} we will discuss the shape derivative of \eqref{eq:red_lag_func_def_curvature}. These computations are not straight forward due to the non-standard jump term.

\subsection{Perturbed problem}
\label{subsec:perturbed_problem}
As preparation for the shape derivatives, we will now introduce perturbations of the triangulation $\T_h$ using vector fields $\VX\in [W^{1,\infty}(\VR^d)]^d$, where $[W^{1,\infty}(\VR^d)]^d$ denotes the space of Lipschitz continuous functions, which will be discretized with $H^1$-conforming finite elements. Further, we will focus on the three-dimensional case, the two-dimensional one follows the same lines. In the following let $\VX\in [V_h^k(\T_h)]^3$, $k\ge 1$, be a vector field. Then we consider \eqref{eq:weak_form_mean_curv} on the family of perturbed domains: 
\[
\T_h^t:= \{\Tt(T):\; T\in \T_h\}, \qquad \Tt(\bm{x})= \bm{x} + t\,\VX(\bm{x}),\quad \bm{x}\in\T_h, \quad\text{ for } t\ge 0 \text{ small. } 
\]
Find $\kappa_{t}\in V_h^k(\T^t_h)$, such that 
\begin{align}\label{eq:weak_form_normal_per}
\int_{\T^t_h} \kappa_{t}\,\delta\kappa_{t} \,ds = -\sum_{T\in\T_h^t}\left(\int_T \tr{\partial^{S_t}\bnu_{t}}\delta\kappa_{t}\,ds + \int_{\partial T}\left(\frac{\pi}{2}-\sphericalangle(\mubf_{t},\Av{\nubf}_t)\right)\delta\kappa_{t}\,d\gamma\right)
\end{align}
for all $\delta \kappa_{t} \in V_h^k(\T^t_h)$. We emphasize, that the polynomial order of the initial and perturbed triangulation fit with the polynomial degree used for the vector field $\VX\in [V_h^k(\T_h)]^3$, i.e., $\T_h = \T_{h,k}$ and $\T^t_h = \T^t_{h,k}$.

To compute the perturbed averaged normal vector $\Av{\nubf}_t$ information of two neighboured elements are required at once. Under the assumption that the perturbation is ``small enough'' instead of measuring the angle with the perturbed averaged normal vector $\Av{\nubf}_t$, we can use the unperturbed one. Starting in 2D for the derivation we consider the situation demonstrated in Figure \ref{fig:angle_comp_av_nv} where the unperturbed averaged normal vector is used to compute the angle in (c). Only, if the perturbation is too large, such that the unperturbed averaged normal vector $\Av{\bnu}$ does not remain between the perturbed co-normal vectors $\mubf_{R,t}$ and $\mubf_{L,t}$, a wrong angle is measured, which needs to be avoided. In the three dimensional setting, however, one has additionally to project $\Av{\nubf}\circ \Tt^{-1}$ to the plane orthogonal to the perturbed tangent vector $\taubf_t$ and renormalize it to measure the correct angle
\begin{align}
\bP_{\taubf_{t}}^{\perp}(\Av{\nubf}):=\frac{\Av{\nubf}\circ\Tt^{-1}-(\Av{\nubf}\circ\Tt^{-1}\cdot \taubf_{t}) \taubf_{t}}{\|\Av{\nubf}\circ\Tt^{-1}-(\Av{\nubf}\circ\Tt^{-1}\cdot \taubf_{t}) \taubf_{t}\|}.\label{eq:proj_edge_ortho}
\end{align}
A simple example for demonstrating the necessity of projection \eqref{eq:proj_edge_ortho} is given as follows: if the two elements rotate by $\Tt$ around the axis $\Av{\nubf}\times \taubf_L$ no change of angle occurs. However, with $\sphericalangle(\mubf_{t},\Av{\nubf})$ a too small angle is now measured from both sides indicating a wrong change of curvature.

Therefore, \eqref{eq:weak_form_normal_per} changes to
\begin{align}\label{eq:weak_form_normal_per_angle}
\int_{\T^t_h} \kappa_{t}\,\delta\kappa_{t} \,ds = -\sum_{T\in\T_h^t}\left(\int_T \tr{\partial^{S_t}\bnu_{t}}\,\delta\kappa_{t}\,ds +  \int_{\partial T}\left(\frac{\pi}{2}-\sphericalangle(\mubf_{t},\bP_{\taubf_{t}}^{\perp}(\Av{\nubf}))\right)\delta\kappa_{t}\,d\gamma\right)
\end{align}
for all $\delta \kappa_t \in V_h^k(\T^t_h)$.

\begin{figure}[h]
	\centering
	\begin{tabular}{ccc}
		\includegraphics[width=0.2\textwidth]{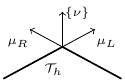}&
		\includegraphics[width=0.2\textwidth]{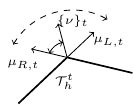}&
		\includegraphics[width=0.2\textwidth]{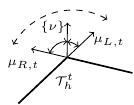}\\
		(a)&(b)&(c)
	\end{tabular}
	\caption{ Angle computation in two-dimensional setting. (a) Unperturbed surface. (b) Perturbed averaged normal vector on perturbed surface. (c) Unperturbed averaged normal vector on perturbed surface.}
	\label{fig:angle_comp_av_nv}
\end{figure}

A rigorous proof of the equivalence of \eqref{eq:weak_form_normal_per} and \eqref{eq:weak_form_normal_per_angle} is provided by the following lemma.

\begin{lemma}
	There holds for all sufficiently small $|t|$:
	\begin{equation}\label{eq:equivalence}
	\sum_{T\in\T_h^t}\int_{\partial T}\sphericalangle(\mubf_t,\Av{\nubf}_t)\,d\gamma =\sum_{T\in\T_h^t}\int_{\partial T}\sphericalangle(\mubf_t,\bP_{\taubf_{t}}^\perp(\Av{\nubf}))\,d\gamma.
	\end{equation}
\end{lemma}
\begin{proof}
	First, we rewrite the left and right hand side in \eqref{eq:equivalence} as sum over edges
	\begin{align*}
	\sum_{T\in\T_h^t}\int_{\partial T}\sphericalangle(\mubf_t,\Av{\nubf}_t)\,d\gamma &= \sum_{E\in\E^t_h}\int_E \sphericalangle(\mubf_{L,t},\Av{\nubf}_t)+\sphericalangle(\mubf_{R,t},\Av{\nubf}_t)\,d\gamma,\\
	\sum_{T\in\T_h^t}\int_{\partial T}\sphericalangle(\mubf_t,\bP_{\taubf_{t}}^\perp(\Av{\nubf}))\,d\gamma &= \sum_{E\in\E^t_h}\int_E \sphericalangle(\mubf_{L,t},\bP_{\taubf_{t}}^\perp(\Av{\nubf}))+\sphericalangle(\mubf_{R,t},\bP_{\taubf_{t}}^\perp(\Av{\nubf}))\,d\gamma.
	\end{align*}
	Note that $\Av{\nubf}_t$ and $\bP_{\taubf_{t}}^\perp(\Av{\nubf})$ are single valued on $E$. Now, it is sufficient to prove that for all edges $E\in\E_h^t$
	\begin{align}
	\label{eq:edge_wise_angles}
	\int_E \sphericalangle(\mubf_{L,t},\Av{\nubf}_t)+\sphericalangle(\mubf_{R,t},\Av{\nubf}_t)\,d\gamma=\int_E \sphericalangle(\mubf_{L,t},\bP_{\taubf_{t}}^\perp(\Av{\nubf}))+\sphericalangle(\mubf_{R,t},\bP_{\taubf_{t}}^\perp(\Av{\nubf}))\,d\gamma.
	\end{align}
	We start with the left side. By noting that (cf. Figure~\ref{fig:nv_tv_env}) the perturbed tangent vector $\taubf_{R,t}=-\taubf_{L,t}$ are orthogonal to the perturbed co-normal vectors. Further, they are perpendicular to the averaged normal vector, $\Av{\nubf}_t \perp \taubf_{L,t},\taubf_{R,t}$. In combination with $\dim(\text{span}\{\mubf_{L,t},\mubf_{R,t}\})=2$ (for sufficiently small perturbations) we obtain that the perturbed averaged normal vector is in the plane spanned by the two co-normal vectors 
	\begin{align*}
	\Av{\nubf}_t \in \text{span}\{\mubf_{L,t},\mubf_{R,t}\}.
	\end{align*}
	Further $\Av{\nubf}_t$ is normalized by construction and (for fine enough grids and small enough perturbations, see Figure~\ref{fig:nv_tv_env} and Figure~\ref{fig:angle_comp_av_nv}) there exists an $\alpha\in [0,1]$ such that $\Av{\nubf}_t \in \text{span}\{\alpha\,\mubf_{L,t}+(1-\alpha)\mubf_{R,t}\}$. Thus, the requirements of Lemma~\ref{lem:additive_angle} in Appendix~\ref{sec:angle_equivalence} are fulfilled and we have
	\begin{align*}
	\sphericalangle(\mubf_{L,t},\Av{\nubf}_t)+\sphericalangle(\mubf_{R,t},\Av{\nubf}_t) = \sphericalangle(\mubf_{L,t},\mubf_{R,t}).
	\end{align*}
	For the right side of \eqref{eq:edge_wise_angles} we already  showed $\mubf_{R,t},\mubf_{L,t}\perp\taubf_t$ and $\dim(\text{span}\{\mubf_{L,t},\mubf_{R,t}\})=2$. By construction of the projected averaged normal vector we further have
	\begin{align*}
	\|\bP_{\taubf_{t}}^\perp(\Av{\nubf})\|=1,\qquad \bP_{\taubf_{t}}^\perp(\Av{\nubf}) \perp \taubf_t
	\end{align*}
	and thus, under the assumption that the perturbation is small enough, we get with the same argument as before
	\begin{align*}
	\sphericalangle(\mubf_{L,t},\bP_{\taubf_{t}}^\perp(\Av{\nubf}))+\sphericalangle(\mubf_{R,t},\bP_{\taubf_{t}}^\perp(\Av{\nubf})) = \sphericalangle(\mubf_{L,t},\mubf_{R,t})
	\end{align*}
	and by transitivity
	\begin{align*}
	\sphericalangle(\mubf_{L,t},\bP_{\taubf_{t}}^\perp(\Av{\nubf}))+\sphericalangle(\mubf_{R,t},\bP_{\taubf_{t}}^\perp(\Av{\nubf})) = \sphericalangle(\mubf_{L,t},\Av{\nubf}_t)+\sphericalangle(\mubf_{R,t},\Av{\nubf}_t).
	\end{align*}
\end{proof}

We will see that the projection $\bP_{\taubf_{t}}^{\perp}(\cdot)$ does not induce a term in the first shape derivative, compare Lemma~\ref{lem:shape_der_wg_jump}. Therefore one could neglect it if a gradient based algorithm is applied to solve the problem numerically. For a shape Newton algorithm, where additionally the second shape derivative is involved, however, the projection induces additional terms and cannot be omitted. Thus, for sake of completeness, it is kept in the following.

Before transforming the perturbed geometric quantities back to the initial shape $\T_h$ we define the following basic properties.
Let $\Phi: \hat{\Surf}\rightarrow \Surf$ be a mapping between two manifolds. For scalar functions $f: \Surf\to \VR$ and the identity matrix $\bI$ we have the chain rule
\begin{equation}
\nabla^{\Surf} f \circ \Phi = \bm{A} \nabla^{\hat S} (f\circ \Phi), \quad \bm{A} := \left(\bI - \frac{\partial \bm{\Phi}^{-\top} \nubf}{|\partial \bm{\Phi}^{-\top} \nubf|}\otimes \frac{\partial \bm{\Phi}^{-\top} \nubf}{|\partial \bm{\Phi}^{-\top} \nubf|}\right)\partial \bm{\Phi}^{-\top} 
\end{equation}
and for a vector valued function $f:\Surf\to\VR^3$
\begin{equation}
\partial^{\Surf} \bm{f} \circ \Phi = \partial^{\hat S} (\bm{f}\circ \Phi)\bm{A}^\top.\label{eq:surfgradtrafo_vv}
\end{equation}
Further, with $\Surf_t=\Tt(\Surf)$, $\Tt(\bm{x}):= \bm{x}+t\VX$, $\VX\in [\CC^1(\VR^3)]^3$, there holds
\begin{equation}
\bm{A}^\top(t) = \pTt^{-1}\left(\bI - \frac{\pTt^{-\top} \nubf}{\|\pTt^{-\top} \nubf\|}\otimes \frac{\pTt^{-\top} \nubf}{\|\pTt^{-\top} \nubf\|}\right),\qquad \bm{A}^\top(0) = \bI - \nubf\otimes \nubf\label{eq:gradtrafo}
\end{equation}
and
\begin{equation}
\label{eq:shape_der_A}
\begin{split}
(\bm{A}^\top)'(0) &= - \partial \bX (\bI - \nubf\otimes \nubf) + \partial^{\Surf} \bX^\top\nubf\otimes \nubf + \nubf\otimes \partial^{\Surf} \bX^\top\nubf\\
&=- \partial^{\Surf} \bX+ 2\,\mathrm{Sym}(\nubf\otimes \nubf\partial^{\Surf} \bX).
\end{split}
\end{equation}
Here, $C^1(\VR^d)$ denotes the set of continuous differentiable functions $f:\VR^d\to\VR$ and $\CC^1(\VR^d)$ the set of continuous differentiable functions with compact support.

Next, we define element-wise the transformation determinants
\begin{align}
\label{eq:wt_wte}
& w_t := \det(\pTt)\|\pTt^{-\top} \nubf\| \quad \text{ and } \quad  w_t^E := \|\pTt \taubf\|.
\end{align}
It is readily checked that we have on the initial triangulation $\T_h$
\begin{equation}\label{eq:tauSt_nuSt}
\taubf^t:=\taubf_t\circ\Tt = \frac{\pTt \taubf}{\|\pTt \taubf\|} \quad \text{ and  } \quad \nubf^t:=\nubf_t\circ \Tt = \frac{\pTt^{-\top} \nubf}{\|\pTt^{-\top} \nubf\|}. 
\end{equation}
We define $\mubf^t := \nubf^t\times \taubf^t$ and could use \eqref{eq:tauSt_nuSt} for the pull-back of $\mubf^t$. However, there exists a more compact form:
\begin{lemma}\label{lem:mubft}
	With $\bA(t):= (\bI-\nubf^t\otimes\nubf^t)\pTt^{-\top}$ we have for all $|t|$ sufficiently small,
	\begin{equation}\label{eq:mu_St}
	\mubf^t:=\mubf_t\circ\Tt = \frac{\bA(t)\mubf}{\|\bA(t) \mubf\|}.
	\end{equation}
\end{lemma}
\begin{proof}
	Since $|t|$ is sufficiently small $T_t$ and $\partial T_t$ are invertible. We compute using the formulas \eqref{eq:tauSt_nuSt} for $\nubf^t$ and $\taubf^t$:
	\begin{align*}
	&\mubf^t \cdot\nubf^t=\frac{\pTt^{-\top} \mubf}{\|\bA(t) \mubf\|}\cdot(\bI-\nubf^t\circ\Tt\otimes\nubf^t\circ\Tt)\nubf^t = 0,\\
	&\mubf^t \cdot\taubf^t=\frac{\pTt^{-\top} \mubf}{\|\bA(t) \mubf\|}\cdot(\bI-\nubf^t\otimes\nubf^t)\taubf^t = \frac{1}{\|\bA(t)\mubf\|\|\pTt\taubf\|}\mubf\cdot\taubf = 0.
	\end{align*}
	Therefore $\mubf^t$ lies in $\mathrm{span}\{\nubf^t,\taubf^t\}^\bot = \mathrm{span}\{\mubf^t\}$ and thus $\det(\nubf^t,\taubf^t,\mubf^t)=\pm1$. Indeed the determinant is positive since for small $t$ we have
	\begin{align*}
	\det(\nubf^t,\taubf^t,\mubf^t) &=\det\left(\frac{\pTt^{-\top}\nubf}{\|\pTt\taubf\|},\frac{\pTt\taubf}{\|\pTt^{-\top}\nubf\|},\left(\bI-\frac{\pTt^{-\top}\nubf}{\|\pTt^{-\top}\nubf\|}\otimes \frac{\pTt^{-\top}\nubf}{\|\pTt^{-\top}\nubf\|}\right)\pTt^{-\top}\mubf\right)\\
	&=\det\left(\frac{\pTt^{-\top}\nubf}{\|\pTt\taubf\|},\frac{\pTt\taubf}{\|\pTt^{-\top}\nubf\|},\pTt^{-\top}\mubf\right)\\
	&=\det(\pTt^{-\top})\det\left(\frac{\nubf}{\|\pTt\taubf\|},\frac{\pTt^{\top}\pTt\taubf}{\|\pTt^{-\top}\nubf\|},\mubf\right)\\
	&=\frac{\det(\pTt^{-\top})}{\|\pTt\taubf\|\|\pTt^{-\top}\nubf\|}\underbrace{\mubf\times\nubf}_{=\taubf}\cdot (\pTt^{\top}\pTt\taubf) \\
	& =\frac{\det(\pTt^{-\top})}{\|\pTt\taubf\|\|\pTt^{-\top}\nubf\|}\|\pTt\taubf\|^2>0.
	\end{align*}
\end{proof}
Using Lemma~\ref{lem:mubft} the perturbation of \eqref{eq:red_lag_func_def_curvature} reads by changing variables together with \eqref{eq:surfgradtrafo_vv} and $\partial^{\Surf}(\bnu_t\circ \Tt)=(\partial^{S_t}\bnu_t)\circ \Tt \bA^{-\top}(t)$
\begin{equation}
\label{eq:pert_lag_func_def_curvature}
\begin{split}
\mathcal{L}^t (\T^t_h,&\kappa_t,\sigma_t)  = \sum_{T\in \mathcal{T}_h}\bigg(\int_T w_t\,2\kappa_b\left(\frac{1}{2}\kappa-H_0\right)^2  + w_t\,(\kappa + \tr{\partial^{\Surf}\bnu^t\bA^\top(t)}\sigma\,ds \\
&  + \int_{\partial T}w^E_t\left(\frac{\pi}{2}-\sphericalangle(\mubf^t,\bP_{\taubf^{t}}^\perp(\Av{\nubf}))\right)\sigma\,d\gamma\bigg)+ c_A\,J_{\mathrm{surf}}(\T^t_h) + c_V\,J_{\mathrm{vol}}(\T^t_h).
\end{split}
\end{equation}

\section{Shape derivatives}
\label{sec:shape_der}
In this section we derive all shape derivatives \cite{DZ11,SZ92} involved for the perturbed Lagrangian \eqref{eq:pert_lag_func_def_curvature} using the notation from the previous section. For completeness we present all shape derivatives, but concentrate on the more involved shape operator and corresponding distributional jump term. Further, the state and adjoint state problems used in the solving algorithm in Section~\ref{sec:solving_algo} are presented. As before we assume that $\Surf\subset \VR^3$ is a smooth embedded submanifold of dimension two. We stress that the two-dimensional case with a one-dimensional sub-manifold directly follows.

\subsection{Shape derivative of normal/tangential vectors and constraints}
\label{subsec:shape_der_vectors_constr}
We start with the following well-known, but crucial, shape-derivatives of the geometric quantities.
\begin{lemma}\label{lem:wt_nut}
	Let $w_t$, $w_t^E$, $\taubf^{t}$, $\nubf^t$, and $\mubf^t$ defined as in \eqref{eq:wt_wte}, \eqref{eq:tauSt_nuSt}, and \eqref{eq:mu_St}. Then  for every $T\in \T_h$ and $E\in \E_h$:
	\begin{subequations}
		\begin{align}
		\frac{d}{dt}  \nubf^t |_{t=0} & = -\partial^{\Surf} \bX^\top\nubf   && \text{ in } [C(T)]^3,\label{eq:shape_der_nv_surface}\\
		\frac{d}{dt} \taubf^t |_{t=0} & = (\bI-\taubf\otimes\taubf)\partial^{\Surf} \bX\taubf &&  \text{ in } [C(E)]^3,\\
		\frac{d}{dt}\mubf^t |_{t=0} & = ((\bI-\taubf\otimes\taubf)\partial^{\Surf}\bX-\partial^{\Surf}\bX^\top)\mubf && \text{ in } [C(E)]^3,\label{eq:shape_der_mu_surface}\\
		\frac{d}{dt} w_t|_{t=0}                  & = \Div[S]\VX && \text{ in } C(T)\label{eq:shape_der_area},\\
		\frac{d}{dt}w_t^E |_{t=0} & = \partial^{\Surf} \bX_{\taubf\taubf}:= \partial^{\Surf} \bX \taubf\cdot \taubf && \text{ in } C(E),
		\end{align}
	\end{subequations}
	where convergence in $C(T)$ has to be understood with respect to the norm $\|f\|_{C(\T_h)}:= \max_{\bm{x}\in \T_h}\|f(\bm{x})\|$, and analogously for $C(E)$.
\end{lemma}
\begin{proof}
	Recalling $\frac{d}{dt} \partial \Tt^{-\top}|_{t=0} = - \partial \VX^\top$ and the formula \eqref{eq:tauSt_nuSt} for $\nubf^t$, we compute using the product rule: 
	\begin{align*}
	\frac{d}{dt}  (\nubf^t) |_{t=0} &= \frac{d}{dt}  \frac{\pTt^{-\top}\nubf}{\|\pTt^{-\top}\nubf\|}|_{t=0} \\
	& = -\partial \bX^{\top}\nubf + (\partial \bX\nubf\cdot\nubf)\nubf  = -(\bI-\nubf\otimes\nubf)\partial \bX^\top\nubf = -\partial^{\Surf} \bX^\top\nubf.
	\end{align*}
	The other identities follow analogously together with $\frac{d}{dt}\det(\pTt)|_{t=0}=\Div[S]{\VX}$.
\end{proof}
As an immediate consequence we obtain for the constraints:
\begin{lemma}
	\label{lem:dJ_surf_vol}
	The shape derivatives of the surface and volume constraint  in direction $\VX\in [\CC^1(\VR^d)]^d$  are given by
	\begin{align}
	DJ_{\mathrm{surf}}(\T_h)(\VX) & = 2(|\T_h|-A_0)\int_{\T_h}\Div[S]\VX\;ds, \label{eq:dJ_surf}\\
	DJ_{\mathrm{vol}}(\T_h)(\VX) & = 2(|\Omega_h|-V_0)\int_{\T_h} \VX\cdot \nubf\;ds.
	\end{align}
\end{lemma}
\begin{proof}
	The shape derivatives of $J_{\mathrm{surf}}$ follows directly from \eqref{eq:shape_der_area} and $J_{\mathrm{vol}}$ from $D\int_{\Omega_h}1\,dx(\VX)=\int_{\Omega_h}\Div{\VX}\,dx$ together with Gauss theorem, where $\Omega_h$ denotes the volume enclosed by $\T_h$.
\end{proof}

\subsection{Shape derivative of shape operator}
\label{subsec:shape_der_shape_op}

\begin{lemma}
	\label{lem:shape_der_jacobian}
	There holds for $\Vu\in [C^1(\VR^3)]^3$ and $\Vu_t:= \Vu\circ \Tt^{-1}$
	\begin{align}
	\frac{d}{dt}(\partial^{S_t} \Vu_t)\circ\Tt|_{t=0} = \partial^{\Surf} \bu\left(2\,\mathrm{Sym}(\nubf\otimes\nubf\partial^{\Surf}\bX)-\partial^{\Surf}\bX\right).
	\end{align}
\end{lemma}
\begin{proof}
	We have  for $t\ge 0$ by the chain rule
	\begin{equation}\label{eq:trafo_gradient}
	(\partial^{S_t} \Vu_t) \circ \Tt = \partial^{\Surf} \Vu\,\VA(t)^\top
	\end{equation}
	with $\VA(t)$ as in \eqref{eq:surfgradtrafo_vv}. Moreover, with \eqref{eq:shape_der_A} it follows by differentiating \eqref{eq:trafo_gradient} that
	\begin{align*}
	\frac{d}{dt}(\partial^{S_t} \bu_t)\circ\Tt |_{t=0} &= \partial^{\Surf} \bu\, (\bA^\top)'(0) = \partial^{\Surf} \bu\left(2\,\mathrm{Sym}(\nubf\otimes \nubf\partial^{\Surf} \bX)-\partial^{\Surf} \bX\right).
	\end{align*}
\end{proof}

\begin{lemma}
	\label{lem:shape_der_wg}
	For the shape operator $\partial^{\Surf}\bnu$ and its trace there hold
	\begin{align}
	\frac{d}{dt}(\partial^{S_t}\bnu_t)\circ \Tt|_{t=0}  &= \partial^{\Surf}\bnu\left(2\,\mathrm{Sym}(\nubf\otimes\nubf\partial^{\Surf}\bX)-\partial^{\Surf}\bX\right)- \mathrm{hess}(\VX)(\nubf) - \partial^{\Surf}\bX^\top\partial^{\Surf}\bnu,\\
	\begin{split}\frac{d}{dt}\tr{\partial^{S_t}\bnu_t}\circ \Tt|_{t=0}  &= - \Delta^{\Surf}\VX\cdot\nubf - 2\partial^{\Surf}\bX : \partial^{\Surf}\bnu,
	\end{split}
	\end{align}
	where $\mathrm{hess}(\VX)$ denotes the Riemannian Hessian on $\Surf$ of $\VX$ and $\Delta^{\Surf}\VX:=\Div[S]{\partial^{\Surf}\VX}$ the \emph{Laplace-Beltrami} operator.
	
	Further, with a continuous and piece-wise smooth function $\sigma$ there holds for the weak form
	\begin{align}
	\label{eq:shape_der_mean_weak}
	\int_T -(\Delta^{\Surf}\VX\cdot\nubf +2\partial^{\Surf}\bX : \partial^{\Surf}\bnu)\sigma\,ds = \int_T\partial^{\Surf}\VX\partial^{\Surf}\sigma\cdot\nubf -\partial^{\Surf}\bX : \partial^{\Surf}\bnu\sigma\,ds - \int_{\partial T}\partial^{\Surf}\VX\mubf\cdot \nubf\,\sigma\,d\gamma.
	\end{align}
\end{lemma}
\begin{proof}
	With the product rule and Lemma \ref{lem:shape_der_jacobian} there holds
	\begin{align*}
	\frac{d}{dt}(\partial^{S_t}\bnu_t)\circ \Tt|_{t=0} = \partial^{\Surf}\left(\frac{d}{dt}(\bnu_t)\circ \Tt|_{t=0}\right) + \partial^{\Surf} \bnu\left(2\,\mathrm{Sym}(\nubf\otimes \nubf\partial^{\Surf} \bX)-\partial^{\Surf} \bX\right)
	\end{align*}
	and further with \eqref{eq:shape_der_nv_surface}
	\begin{align*}
	\partial^{\Surf}\left(\frac{d}{dt}(\bnu_t)\circ \Tt|_{t=0}\right) = -\partial^{\Surf}\left(\partial^{\Surf} \bX^\top\bnu\right) = -\partial^{\Surf} \bX^\top\partial^{\Surf}\bnu - \text{hess}(\VX)(\nubf).
	\end{align*}
	For the trace of the shape operator we compute 
	\begin{align*}
	&\tr{\partial^{\Surf}\bnu\left(2\,\mathrm{Sym}(\nubf\otimes\nubf\partial^{\Surf}\bX)-\partial^{\Surf}\bX\right)- \mathrm{hess}(\VX)(\nubf) - \partial^{\Surf}\bX^\top\partial^{\Surf}\bnu} = \\
	& 2 \partial^{\Surf}\bnu : (\mathrm{Sym}(\nubf\otimes\nubf\partial^{\Surf}\bX)) - \partial^{\Surf}\bnu:\partial^{\Surf}\bX - \tr{\mathrm{hess}(\VX)(\nubf)} - \partial^{\Surf}\bX : \partial^{\Surf}\bnu = \\
	&2 \underbrace{\partial^{\Surf}\bnu : (\nubf\otimes\nubf\partial^{\Surf}\bX)}_{=0} - \tr{\mathrm{hess}(\VX)(\nubf)} - 2\partial^{\Surf}\bX : \partial^{\Surf}\bnu = - \Delta^{\Surf}\VX\cdot\nubf - 2\partial^{\Surf}\bX : \partial^{\Surf}\bnu,
	\end{align*}
	where we used that $\partial^{\Surf}\bnu^\top\bnu=0$. The weak form \eqref{eq:shape_der_mean_weak} follows directly with integration by parts.
\end{proof}
We will also need the shape derivative of the distributional part of the curvature.
\begin{lemma}
	\label{lem:shape_der_wg_jump}
	Let $T\in\T_h$ and $\hat{T}\in \T_h$ sharing a common edge $E$ on which $\Av{\bnu}$ denotes the averaged normal vector \eqref{eq:av_normal_vector}. There holds on $\partial T$ with its co-normal vector $\mubf$ on $E$
	\begin{align}
	\frac{d}{dt}\sphericalangle(\mubf^{t},\bP_{\taubf^{t}}^{\perp}(\Av{\nubf}))|_{t=0}= -\frac{(\partial^{\Surf}\bX-\partial^{\Surf}\bX^\top)\mubf\cdot \Av{\nubf}}{\sqrt{1-(\mubf\cdot\Av{\nubf})^2}}.
	\end{align}
\end{lemma}
\begin{proof}
	See Appendix \ref{app:comp_shape_der}.
\end{proof}

We emphasize that the same result holds if we would neglect the projection $\bP_{\taubf^{t}}^{\perp}(\cdot)$ and solely consider the term $\sphericalangle(\mubf^{t},\Av{\nubf})$. However, for the second shape derivatives, which is important when considering a shape Newton algorithm, the results would differ.

\subsection{Shape derivative, state, and adjoint state problem}
\label{subsec:shape_der_state_adjoint}

The parameterized Lagrangian is defined by
\begin{equation}
G (t,\kappa,\sigma) :=  \mathcal{L}^t(\T^t_h,\kappa\circ \Tt^{-1},\sigma\circ \Tt^{-1}),
\end{equation}
where $\mathcal{L}^t$ is given by \eqref{eq:pert_lag_func_def_curvature}.
Therefore the shape derivative can be computed by  (see \cite{a_LAST_2016a})
\begin{align}
D\mathcal{J}(\T_h)(\VX) = \partial_tG(0,\kappa,\sigma),
\end{align}
where $(\kappa,\sigma)$ solve
\begin{equation}
\label{eq:state_adjoint eq}
\begin{split}
\text{ find } \kappa, \text{ such that } \;  \partial_{\sigma} G(0,\kappa,\sigma)(\delta\sigma) =0 \quad \text{ for all } \delta\sigma\in V_h^k(\T_h), \\
\text{ find } \sigma, \text{ such that } \; \partial_{\kappa} G(0,\kappa,\sigma)(\delta\kappa) =0 \quad \text{ for all } \delta\kappa\in  V_h^k(\T_h), 
\end{split}
\end{equation}
with 
\begin{align}
&\partial_{\kappa} G(0,\kappa,\sigma)(\delta\kappa) = \int_{\T_h}2\kappa_b\left(\frac{1}{2}\kappa-H_0\right)\delta\kappa+\delta\kappa\,\sigma\,ds,\\
&\partial_{\sigma} G(0,\kappa,\sigma)(\delta\sigma) = \sum_{T\in \mathcal{T}_h}\Big(\int_{T}\kappa\,\delta\sigma+\tr{\partial^{\Surf}\bnu}\delta\sigma\,ds+\int_{\partial T}\left(\frac{\pi}{2}-\sphericalangle(\mubf,\bP_{\taubf}^\perp(\Av{\nubf}))\right)\delta\sigma\,d\gamma\Big).
\end{align}
Adding up all terms together with Lemma~\ref{lem:dJ_surf_vol} the shape derivative of Lagrangian \eqref{eq:red_lag_func_def_curvature} reads
\begin{align}
D\mathcal{J}(\T_h)(\VX) &= \sum_{T\in \T_h}\Big(\int_{T} \Div[S]{\VX}2\kappa_b(\frac{1}{2}\kappa-H_0)^2 + \Div[S]{\VX}\sigma\,\kappa +\Div[S]{ \VX} \tr{\partial^{\Surf}\bnu}\sigma\nonumber\\
&+\partial^{\Surf}\VX\partial^{\Surf}\sigma\cdot\nubf -\partial^{\Surf}\bX : \partial^{\Surf}\bnu\sigma\,ds - \int_{\partial T}\partial^{\Surf}\VX\mubf\cdot \nubf\,\sigma\,d\gamma\nonumber\\
&+2c_A(|\T_h|-A_0)\int_{T}\Div[S]{\VX}\,ds+2c_V(|\Omega_h|-V_0)\int_{T} \VX\cdot \nubf \,ds \nonumber\\
&+ \int_{\partial T}\big(\partial^{\Surf} \bX_{\taubf\taubf}\left(\frac{\pi}{2}-\sphericalangle(\mubf,\Av{\nubf})\right)+\frac{(\partial^{\Surf}\bX-\partial^{\Surf}\bX^\top)\mubf\cdot \Av{\nubf}}{\sqrt{1-(\mubf\cdot\Av{\nubf})^2}}\big)\sigma\,d\gamma\Big).\label{eq:lag_func_curv_shape_der}
\end{align}
In the lowest order case, $\T_{h,1}$, there holds $\VX \in [V_h^1(\T_h)]^3$ and $\bnu|_T=\text{const}$, and therefore \eqref{eq:lag_func_curv_shape_der} simplifies with \eqref{eq:shape_der_mean_weak} to
\begin{align}
D\mathcal{J}(\T_h)(\VX) &= \sum_{T\in \T_h}\Big(\int_{T} \Div[S]{\VX}2\kappa_b\left(\frac{1}{2}\kappa-H_0\right)^2 + \Div[S]{\VX}\sigma\,\kappa\,ds\nonumber\\
& +2c_A(|\T_h|-A_0)\int_{T}\Div[S]{\VX}\,ds+2c_V(|\Omega_h|-V_0)\int_{T} \VX\cdot \nubf \,ds \nonumber\\
&+ \int_{\partial T}\big(\partial^{\Surf} \bX_{\taubf\taubf}\left(\frac{\pi}{2}-\sphericalangle(\mubf,\Av{\nubf})\right)+\frac{(\partial^{\Surf}\bX-\partial^{\Surf}\bX^\top)\mubf\cdot \Av{\nubf}}{\sqrt{1-(\mubf\cdot\Av{\nubf})^2}}\big)\sigma\,d\gamma\Big).\label{eq:lag_func_curv_shape_der_lo}
\end{align}
We observe that in this case the lifting of the distributional shape operator $-\partial^{\Surf}\bnu$ is done only by the boundary jump terms.
For a one-dimensional sub-manifold in 2D the jump term in \eqref{eq:lag_func_curv_shape_der} simplifies as no deformation determinant is involved
\begin{align}
D\mathcal{J}^{2D}(\T_h)(\VX) &= \sum_{T\in \T_h}\Big(\int_{T} \Div[S]{\VX}2\kappa_b(\kappa-H_0)^2 + \Div[S]{\VX}\sigma\,\kappa +\Div[S]{ \VX} \tr{\partial^{\Surf}\bnu}\sigma\nonumber\\
&+\partial^{\Surf}\VX\partial^{\Surf}\sigma\cdot\nubf -\partial^{\Surf}\bX : \partial^{\Surf}\bnu\sigma\,ds - \int_{\partial T}\partial^{\Surf}\VX\mubf\cdot \nubf\,\sigma\,d\gamma\nonumber\\
&+2c_A(|\T_h|-A_0)\int_{T}\Div[S]{\VX}\,ds+2c_V(|\Omega_h|-V_0)\int_{T} \VX\cdot \nubf \,ds \nonumber\\
&+ \int_{\partial T}\frac{(\partial^{\Surf}\bX-\partial^{\Surf}\bX^\top)\mubf\cdot \Av{\nubf}}{\sqrt{1-(\mubf\cdot\Av{\nubf})^2}}\sigma\,d\gamma\Big).\label{eq:lag_func_curv_shape_der_2d}
\end{align}

\subsection{Stabilization of element-areas}
\label{subsec:stab_element_areas}
With the penalty term $c_A(|\T_h|-A_0)^2$ we control the total surface area to be close to a prescribed value $A_0$. As the solution is invariant under re-parameterization it may happen, however, that some elements shrink or increase their local area, leading to a deterioration of the shape regularity of the elements. To mitigate this possible mesh-degeneration we use a local area preservation constraint (see \cite{a_BILIKO_2020a}):
\begin{align}
\sum_{T\in\T_h}c_{A_{\mathrm{loc},T}}(|T|-|T_0|)^2,
\end{align}
where $T_0$ denotes the element area on the initial shape and $c_{A_{\mathrm{loc},T}}>0$ the penalty parameter, which can be chosen for each element $T$ individually. Further, its shape-derivative is of the same form as the global area constraint.

Other approaches to prevent ill-shaped meshes are regularity techniques as viscous regularization \cite{MK08} or remeshing, including local refinement, coarsening, and smoothing \cite{Paul08,BNP10}. Note that for free-boundary problems the question of the placement of additionally inserted nodes to obtain a consistent mesh is not straight-forward, especially for high-order curved shapes.

\section{Solving algorithm}
\label{sec:solving_algo}
\subsection{Basic algorithm}
\label{subsec:basic_algo}
Let $\T_h$ be a fixed initial surface and let $\bm{H}=[H^1(\T_h)]^3$ be equipped with the scalar product
\begin{align}
(\VV,\VW)_{\bm{H}}:= \int_{\T_h}\partial^{\Surf}\bm{V}:\partial^{\Surf}\bm{W} + \varepsilon\,\VV\cdot \VW\,ds,\qquad \varepsilon>0.\label{eq:h1_eps_scalar_product}
\end{align} 
Here, $\varepsilon>0$, which will be fixed to $\varepsilon:=1\times 10^{-10}$ throughout this work, is needed to guarantee positive definiteness, as we will consider closed surfaces without a possible boundary. We emphasize that the full gradient $\partial^{\Surf}$ leads to more regular displacement updates than considering only its symmetric part $\partial^{\Surf}\VV+\partial^{\Surf}\VV^\top$. Further, the constant $\varepsilon$ in \eqref{eq:h1_eps_scalar_product} should be chosen to be small as otherwise the mass matrix gains more weight also leading to worse mesh-quality updates. Then, the gradient $\nabla^{\bm{H}} g(\VV)$ is defined by
\begin{align}
\partial g(\VV)(\VW)=(\nabla^{\bm{H}} g(\VV),\VW)_{\bm{H}}\qquad \forall\, \VW\in \bm{H},\label{eq:shape_der_via_grad}
\end{align}
where the mapping $g$ is defined by
\begin{align}
\VV\mapsto g(\VV):= \mathcal{J}((\Id+\VV)(\T_h))
\end{align}
and there holds for the derivative of $g$ at $\VV$ in direction $\VW$ (see \cite{KSNS20,a_IGSTWE_2018a}) 
\begin{align}
\partial g(\VV)(\VW)=D\mathcal{J}((\Id+\VV)(\T_h))(\VW\circ(\Id+\VV)^{-1}).
\end{align}
The shape optimization algorithm reads as follows:
\begin{algorithm}[H] 
	\label{alg:gradient_alorithm}
	\begin{algorithmic}[1]
		\State{{\bf Input:} surface $\T_h^0$, $n=0$, $N_{\mathrm{max}}>0$, $\delta >0$, $\alpha > 0$  }
		\State {\bf Output:} optimal shape $\T_h^*$ 
		\While{ $n\le N_{\mathrm{max}}$ and $|\nabla \mathcal{J}(\T_h^n)|>\delta$}
		\If{$\mathcal{J}((\Id - \alpha \nabla \mathcal{J}(\T_h^n))(\T_h^n)) \leq \mathcal{J}(\T_h^n)$  }
		\State $\T_h^{n+1} \gets (\Id - \alpha \nabla \mathcal{J}(\T_h^n))(\T_h^n)$
		\State $n\gets n+1$ 
		\State increase $\alpha$
		\Else
		\State reduce $\alpha$
		\EndIf
		\EndWhile
		\caption{gradient algorithm}
		\label{alg:gradient}
	\end{algorithmic}
\end{algorithm}
The input quantities are the initial shape $\T_{h}^0$, the maximal number of optimization steps $N_{\mathrm{max}}>0$, a threshold $\delta>0$ for the shape gradient residuum, and the initial step-size $\alpha$. A line-search is performed by testing if the goal functional, the mean curvature together with the volume and area constraints, is decreasing. Otherwise the step-size $\alpha$ will be reduced. If the step gets accepted it is possible to increase $\alpha$ to gain a faster convergence towards the minimum. Note, however, that a raise of $\alpha$ has to be done carefully as the shape may run e.g. into singularities.

One iteration step of Algorithm~\ref{alg:gradient_alorithm} involves:
\begin{enumerate}
	\item For a fixed surface $\T_h^n$ average the corresponding normal vector $\bnu$ by \eqref{eq:av_normal_vector} and solve for $\kappa$, $\sigma$ the state and adjoint state equation \eqref{eq:state_adjoint eq}.
	\item With the new $\kappa$ and $\sigma$ calculate the gradient by computing the shape derivative \eqref{eq:lag_func_curv_shape_der}.
\end{enumerate}

To reduce the possibility to get stuck in a local minimum a non-monotone gradient method is considered, where the next gradient step needs to result in a lower cost then the maximum of the last $M=5$ energies to be accepted. Therefore the right-hand side of Line 4 in Algorithm~\ref{alg:gradient_alorithm} changes to $\leq \max_{i=0}^{M-1} \mathcal{J}(\T_h^{n-i})$.

After convergence the parameters involving the area and volume constraints can be increased and the algorithm is repeated.

For simplicity we keep with the standard (non-monotone) gradient algorithm. Other methods as l-BFGS or nonlinear conjugate gradient algorithms (NCG) to speed up the convergence and relying on the first shape derivative can directly be adapted; see, e.g.,  \cite{a_BL_2021a,a_SCSIWE_2016a,a_IGSTWE_2018a}.

\subsection{Improved surface preservation}
\label{subsec:improved_area_cons}
In the shape optimization algorithm three different parameters regulating the volume and area constraints are involved: $c_V$, $c_A$, and $c_{A_{\mathrm{loc},T}}$. It is desirable having as less parameters as possible while preserving the convergence and performance of the algorithm. In the case where the initial shape already has the appropriate area it is possible to generate deformation updates such that the area gets close to being constant.

For this purpose, instead of computing the shape gradient in $H^1$ via inner product \eqref{eq:shape_der_via_grad} we incorporate a divergence free condition by solving a Stokes problem, where an additional pressure-like unknown $p$ is introduced to enforce that the displacement increment is divergence-free, i.e., the surface area should be linearly preserved. Given a vector field $\VV$, we seek $(\VX,p)\in \bm{H}\times H^1(\T_h)$, such that
\begin{subequations}
	\label{eq:shape_der_stokes}
	\begin{alignat}{3}
	&(\bm{X},\VW)_{\bm{H}} + (p,\Div{\VW})_{L^2} &&= \partial g(\VV)(\VW),\qquad &&\text{ for all } \VW\in \bm{H},\\
	& (q,\Div{\bm{X}})_{L^2} &&= 0,\qquad&& \text{ for all } q\in H^1(\T_h).
	\end{alignat}
\end{subequations}
The function $\bm{X}$ is the shape gradient $\nabla^{\bm{H}} g(\VV)$ with respect to $\bm{H}$ with the additional condition that 
$\Div{\bm{X}}=0$ in a weak sense. 
Note, that this requires a Stokes stable finite element pairing for $\VX$ and $p$. The famous Taylor-Hood pairing for example requires that $p$ is of one polynomial order lower than $\VX$ and thus, quadratic polynomials for the deformation fields have to be used in combination with a linear approximation of the pressure. Other choices, such as the MINI element, or a pressure-projection stabilized equal order pair \cite{DB04} are possible.

Using this procedure, any moderately choice of $c_A>0$ leads to shapes with stable surface area (we observed that the error evolves with less then 1\%). Setting $c_A=0$ is invalid as in this case we loose control over the global area constraint, having only linear divergence-free updates.
As a result only the volume constraint has to be adapted after the shape optimization algorithm determinates.

The obvious disadvantage of \eqref{eq:shape_der_stokes} is the increased computational effort, since now a saddle-point problem has to be solved in each step instead of a positive define one.

\subsection{Costs}
For a better comparison with other methods we describe the numerical costs for the solving algorithm in this section.

One iteration step of Algorithm~\ref{alg:gradient_alorithm} has the costs of averaging the normal vector $\bnu$ which consists of local problems involving two adjacent elements at each edge. Then two systems are solved to compute the state and adjoint state $\kappa$ and $\bsigma$ where the same mass matrix of a scalar Lagrange unknown (which is symmetric and positive definite (spd)) with two different right-hand sides is used and thus, the matrix has to be assembled, factorized, and inverted only once. Furthermore, one might consider lumped mass matrices, where only the diagonal entries are non-zero such that the matrix is trivial to invert \cite{CJRT01}. Finally, the shape gradient step updating the deformation involves assembling and solving the (regularized) stiffness matrix of a vector-valued Lagrangian finite element and is thus also spd. If the improved surface preservation algorithm from Section~\ref{subsec:improved_area_cons} is considered instead, it becomes a Stokes-like saddle point problem involving an additional Lagrangian pressure unknown.

\subsection{Automatic shape derivatives in NGSolve}
\label{subsec:automatic_shape_der}
In Section~\ref{sec:shape_der} all shape derivatives were computed analytically such that the shape optimization Algorithm~\ref{alg:gradient_alorithm} can directly be applied if the current (deformed) meshes are accessible. One possibility consists of (manually) deforming all vertices of the mesh during each optimization step and then computing the state and adjoint state problems as well as the next shape derivative on it. This, however, is not applicable for curved elements. Instead 
we use an ALE (arbitrary Lagrangian Eulerian) technique, where a mesh (mostly the initial shape) is fixed and all computations are performed on it. Therefore, the involved deformation gradients and determinants have to be incorporated, which is straight forward but error prone and can lead to complicated or confusing expressions. In this work we use the open source finite element software NGSolve\footnote{www.ngsolve.org} \cite{Sch14}, where the method \texttt{SetDeformation} can be used to avoid the manual computation of the transformations and chain rules. In the supplementary material (Appendix~\ref{sec:suppl_material}) a full code example including a detailed description can be found.

To demonstrate how the deformation of a mesh is realized we consider the following lines of code:
\begin{lstlisting}
mesh.SetDeformation(displacement)
A.Assemble()
mesh.UnsetDeformation()
\end{lstlisting}
where $A$ corresponds to the standard stiffness bilinear form
\begin{align*}
a:H^1(\Omega)\times H^1(\Omega)\to\VR,\qquad a(u,v)=\int_{\Omega}\nabla u \cdot\nabla v\,dx,
\end{align*}
\texttt{mesh} is the intial shape $\T_h^0$, and the object \texttt{displacement} knows how the mesh has to be deformed leading to the current shape $\T_h^n$.
Then, everything between \texttt{SetDeformation} and \texttt{UnsetDeformation} gets assembled as it was on the current configuration by using the appropriate transformation rules, namely
\begin{align*}
a(u,v)=\int_{\Omega}J\,\bm{F}^\top\nabla u \cdot\bm{F}^\top\nabla v\,dx
\end{align*}
with $\bm{F}=\bm{I}+\nabla$\texttt{displacement} and $J=\det(\bm{F})$.

The computation of shape derivatives can be challenging and also error prone due complicated expressions. Although we computed and presented the shape derivatives in this work for this specific problem, it is convenient and useful if they can be calculated automatically. If, e.g., the constraints or equations are slightly changed, (parts of) the shape derivatives would have to be recomputed by hand. In the recent publication  \cite{KSNS20} the fully automated and semi-automated 
computation of shape derivatives in NGSolve was presented. For instance, to compute the shape derivative \eqref{eq:lag_func_curv_shape_der}, excluding the area and volume constraint, one can consider for fixed $\kappa$ and $\sigma$ the linear form (compare \eqref{eq:pert_lag_func_def_curvature})
\begin{align*}
F(\T_h,\kappa,\sigma):=\sum_{T\in \mathcal{T}_h}&\Big(\int_T 2\kappa_b\big(\frac{1}{2}\kappa-H_0\big)^2  + (\kappa + \tr{\partial^{\Surf}\bnu})\sigma\,ds\\
&\quad+ \int_{\partial T}\left(\frac{\pi}{2}-\sphericalangle(\mubf,\bP_{\taubf}^\perp(\Av{\nubf}))\right)\sigma\,d\gamma\Big)
\end{align*}
which can be written symbolically in Python as
\begin{lstlisting}
def F(kappa, sigma):
	return (2*kb*(1/2*kappa-H0)**2 + (kappa + Trace(Grad(nsurf)))*sigma)*ds + (pi/2-acos(nel*nav))*sigma*ds(element_boundary=True)
\end{lstlisting}
where \texttt{nav} $:=\bP_{\taubf}^\perp(\Av{\nubf})$, and then call the function \texttt{DiffShape} to obtain $DF(\T_h,\kappa,\sigma)(\VX)$. We emphasize that in the final code we neglected the nonlinear projection operator $\bP_{\taubf}^\perp(\cdot)$ saving unnecessary computations as noted below Lemma~\ref{lem:shape_der_wg_jump}.
\begin{lstlisting}
fesH = VectorH1(mesh, order=order)
X = fesH.TestFunction()

f = LinearForm(fesH)
f += F(kappa,sigma).DiffShape(X)
f.Assemble()
\end{lstlisting}
This procedure can directly be combined with the \texttt{SetDeformation} method from before to compute the shape derivative automatically on the current configuration without actually changing the mesh.

Due to the integral form of the constraints the utilized form for the shape derivative is given by e.g.,
\begin{align*}
\frac{d}{dt}c_V\left(|\Tt(\Omega_h)|-|V_0|\right)^2|_{t=0} = 2\,c_V\left(|\Omega_h|-|V_0|\right)\frac{d}{dt}\,|\Tt(\Omega_h)|\,|_{t=0}
\end{align*}
and then the \texttt{DiffShape} function can be applied to $|\Omega_h| =\int_{\T_h}\frac{1}{d}\Vx\cdot\bnu\,ds$ to compute the full shape derivative.

This tool of automatic shape derivatives can also be extremely helpful in terms of cross-checking the manually computed shape derivatives or an efficient and utilised for fast testing of changes of the equations without the necessity of recomputing all derivatives by hand.

\section{Numerical examples}
\label{sec:numerical_examples}
In this section, we demonstrate the performance of the proposed method. First, we test the mean curvature computation of our method showing the pertinence of the non-standard boundary jump term measuring the angle between two neighbored triangles. Particularly in the lowest order case when approximating the curvature with linear polynomials on an affine triangulation the inner part of \eqref{eq:weak_form_mean_curv} vanishes as the normal vector is piece-wise constant in this case (compare Figure~\ref{fig:nv_jump}). Then, we present two benchmark examples for equilibrium shapes motivated by cell membranes including non-zero spontaneous curvature.

\subsection{Prescribed configurations of sphere and biconcave-oblate}
\label{subsec:prescribed_conf}
\begin{figure}[h]
	\centering
	\includegraphics[width=0.2\textwidth]{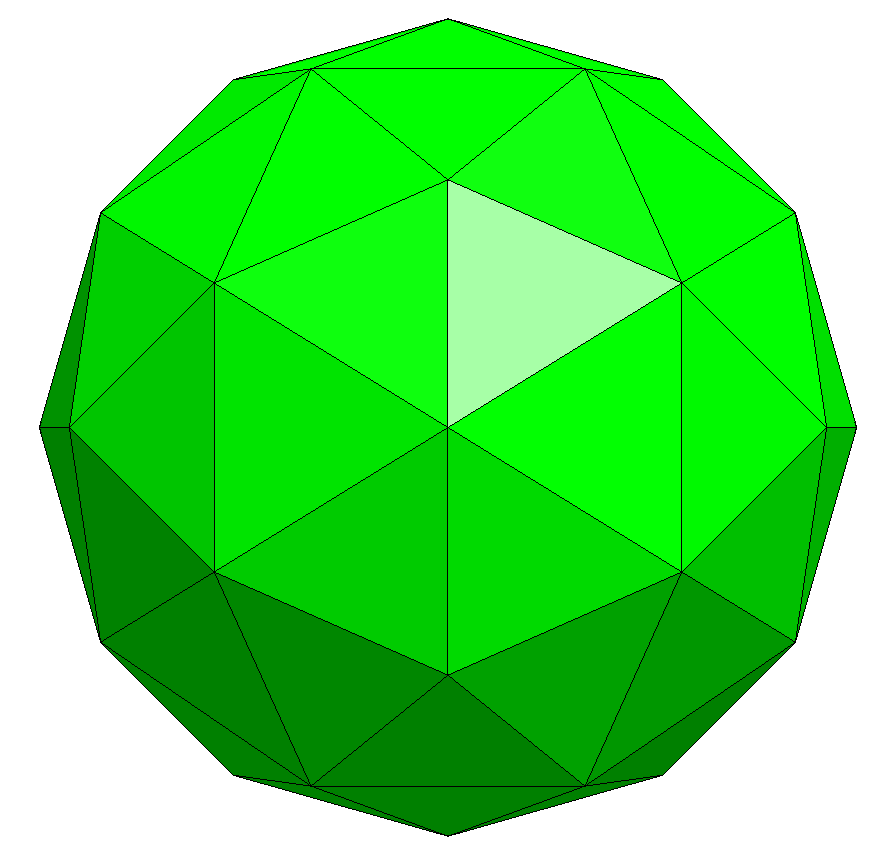}\hspace*{3cm}
	\includegraphics[width=0.28\textwidth]{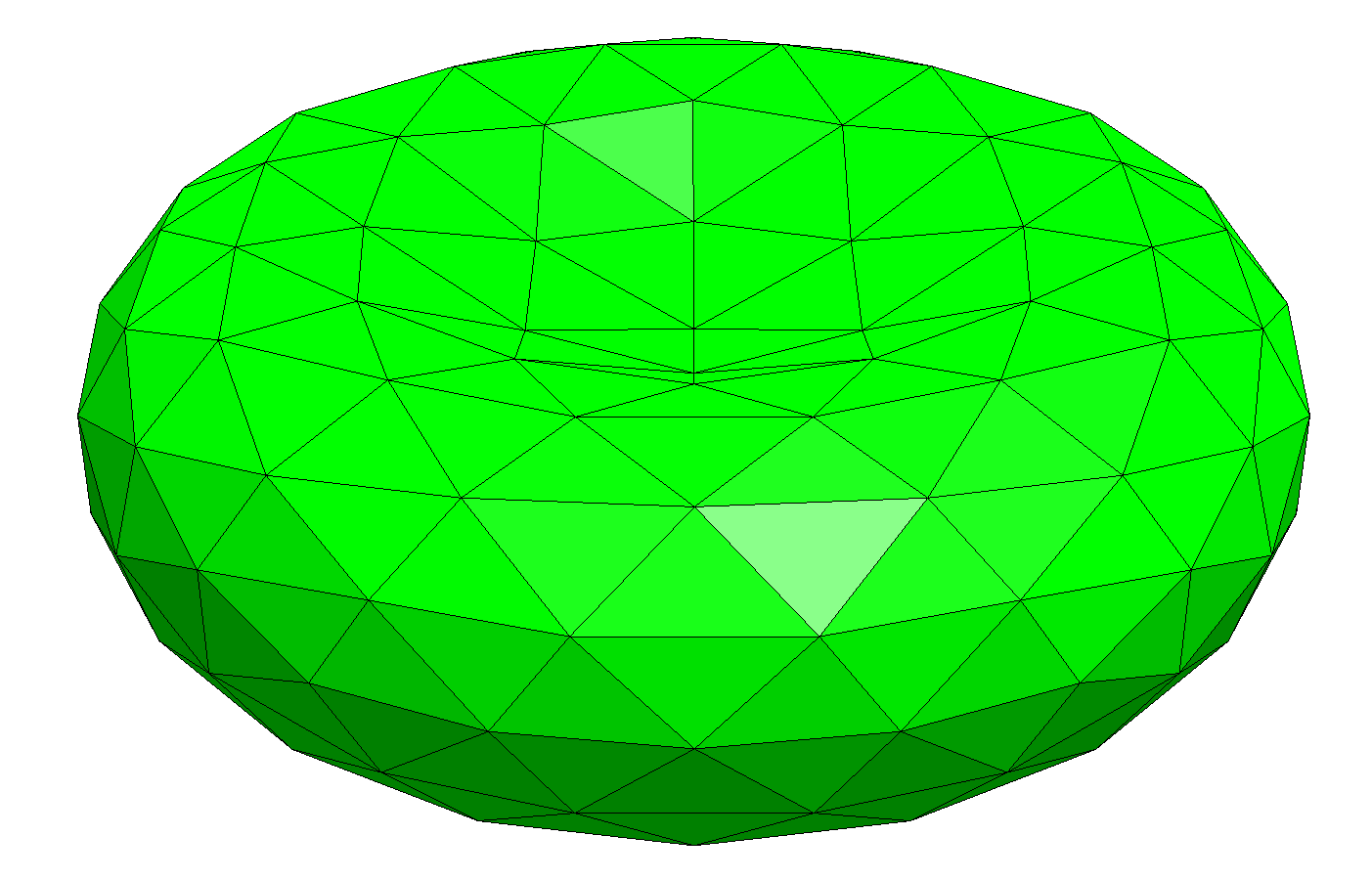}
	\caption{Icosahedron and biconcave-oblate.}
	\label{fig:icos_biconcave}
\end{figure}

We compute the Canham-Helfrich-Evans energy \eqref{eq:canham_helfrich_energy} as in \cite{a_BILIKO_2020a} with $\kappa_b=1$ of a given sphere of radius $R=1$ and a biconcave-oblate described by the embedding
\begin{align*}
x = \sin(u)\sin(v),\qquad y=\sin(u)\cos(v),\qquad z=F(\cos(u)),
\end{align*}
where $(u,v)\in [-\pi/2,\pi/2]\times[0,2\pi]$ are the parametric coordinates of a sphere and $F(p)=0.54353p+0.121435p^3-0.561365p^5$.

The sphere is approximated by an icosahedron and a regular subdivision by dividing each triangle into four sub-triangles. For the biconcave-oblate the points of the icosahedron are appropriately transformed with $F(\cdot)$, compare Figure~\ref{fig:icos_biconcave}. The results for the lowest order method can be found in Figure~\ref{fig:presc_sphere} and \ref{fig:presc_biconc} on the left, which converge to the correct values.

\begin{figure}[h]
	\centering
	\includegraphics[width=0.38\textwidth]{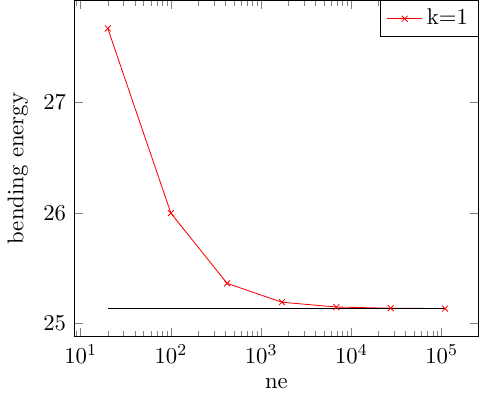}
	\includegraphics[width=0.38\textwidth]{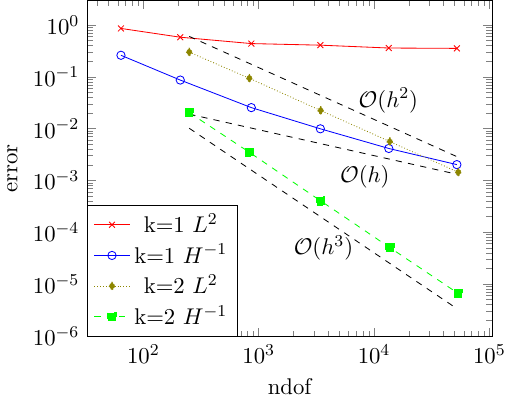}
	
	\caption{Left: Bending energy of icosahedron with lowest-order elements with respect to number of elements (ne). Exact value: $8\pi$. Right: $L^2$ and $H^{-1}$ error for unstructured meshes with linear and quadratic elements with respect to the number of degrees of freedom (ndof).}
	\label{fig:presc_sphere}
\end{figure}

\begin{figure}[h]
	\centering
	\includegraphics[width=0.38\textwidth]{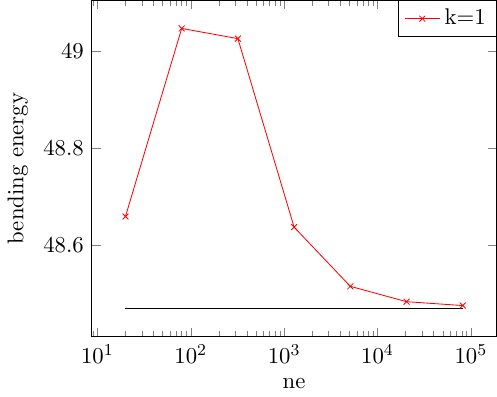}
	\includegraphics[width=0.38\textwidth]{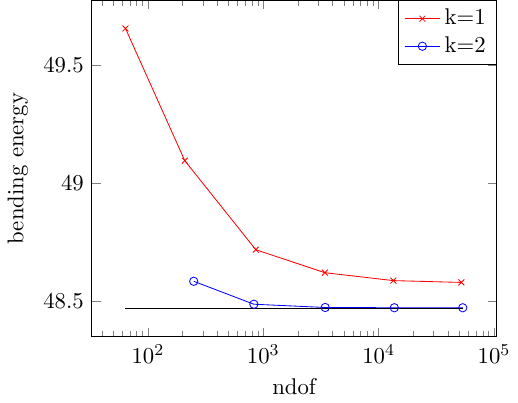}
	
	\caption{Left: Bending energy of icosahedron with lowest-order elements with respect to number of elements. Right: $L^2$-norm of bending energy for unstructured meshes with linear and quadratic elements with respect to number of degrees of freedom. Reference value is $48.47$ \cite{a_BILIKO_2020a}.}
	\label{fig:presc_biconc}
\end{figure}

Next, we consider a sequence of non-nested unstructured meshes generated by NETGEN \cite{Sch97} approximating the sphere and biconcave-oblate with linear and quadratic polynomials, where the geometry is isoparametrically curved. As depicted in Figure~\ref{fig:presc_sphere} and \ref{fig:presc_biconc} on the right the high-order method converges to the exact and reference value, respectively, in the $L^2$- and $H^{-1}$-norm, namely
\begin{align*}
&\|\kappa_h-\kappa_{\mathrm{ref}}\|^2_{L^2}:=\int_{\Surf}|\kappa_h-\kappa_{\mathrm{ref}}|^2\,ds,\quad \|\kappa_h-\kappa_{\mathrm{ref}}\|_{H^{-1}}:=\sup\limits_{\sigma\in H^1(\Surf)} \frac{\langle \kappa_h-\kappa_{\mathrm{ref}},\sigma\rangle }{\|\sigma\|_{H^1}}.
\end{align*} 
As observed in Figure~\ref{fig:presc_sphere} only the convergence rates differ, namely quadratic and cubic order, respectively. Note, that for the $H^{-1}$-norm we solve the auxiliary problem $-\Delta^S u_h=\kappa_h-\kappa_{\mathrm{ref}}$ on $\T_h$ with $u_h\in V_h^l(\T_h)$, $l>k$ as there holds $\|u\|_{H^1}=\|\kappa_h-\kappa_{\mathrm{ref}}\|_{H^{-1}}$ in the continuous case. In the lowest order case, however, the $L^2$-norm is not converging to the reference values, whereas the $H^{-1}$ error does. This is in agreement with the fact that the (discrete) mean curvature is a distribution, rather than a function, and thus, in general, we cannot expect convergence for linear elements in the $L^2$-norm.

The non-convergence of the lowest order curvature computation seems to be in contrast with the approach of using this curvature approximation for the shape optimization algorithm. Therefore, we consider the following test case: The same unstructured sequence of meshes approximating the sphere as used for the results in Figure~\ref{fig:presc_sphere} is taken and the shape optimization Algorithm~\ref{alg:gradient_alorithm} with $A_0$, $V_0$ as the initial shape, and $\kappa_b=0.01$, $c_V=10/|V_0|$, $c_A=5/|A_0|$, $c_{A_{\mathrm{loc},T}}=5/|T_0|$, and $N_{\mathrm{max}}=1000$ is used such that the initial shape is very close to a smooth sphere being the unique solution for this problem. In Figure~\ref{fig:presc_sphere_shapeopt} we can see that the shape optimization algorithm deforms the meshes only marginally in such a way that the $L^2$-norm of the curvature error now converges, even with a quadratic rate. An explanation of this phenomenon is that on the one
hand the curvature approximation enters the shape optimization step by being paired with an $H^1$-test function, compare the $\kappa\,\sigma$ term in \eqref{eq:lag_func_curv_shape_der}, enabling the convergent $H^{-1}$ property and on the other hand the algorithm tries to minimize the (local) $L^2$-norm of the mean curvature generating a sequence of optimal meshes - a quadratic convergence rate is optimal with respect to linear polynomials.

This supports and verifies the usage of linear approximations for the curvature as the shape optimization algorithm generates as a side-product meshes with beneficial curvature computation property.

\begin{figure}[h]
	\centering
	\includegraphics[width=0.3\textwidth]{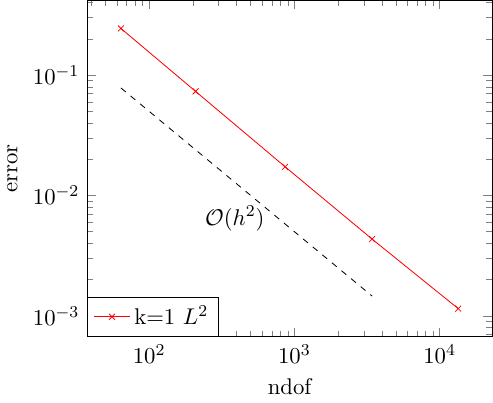}
	\includegraphics[width=0.34\textwidth]{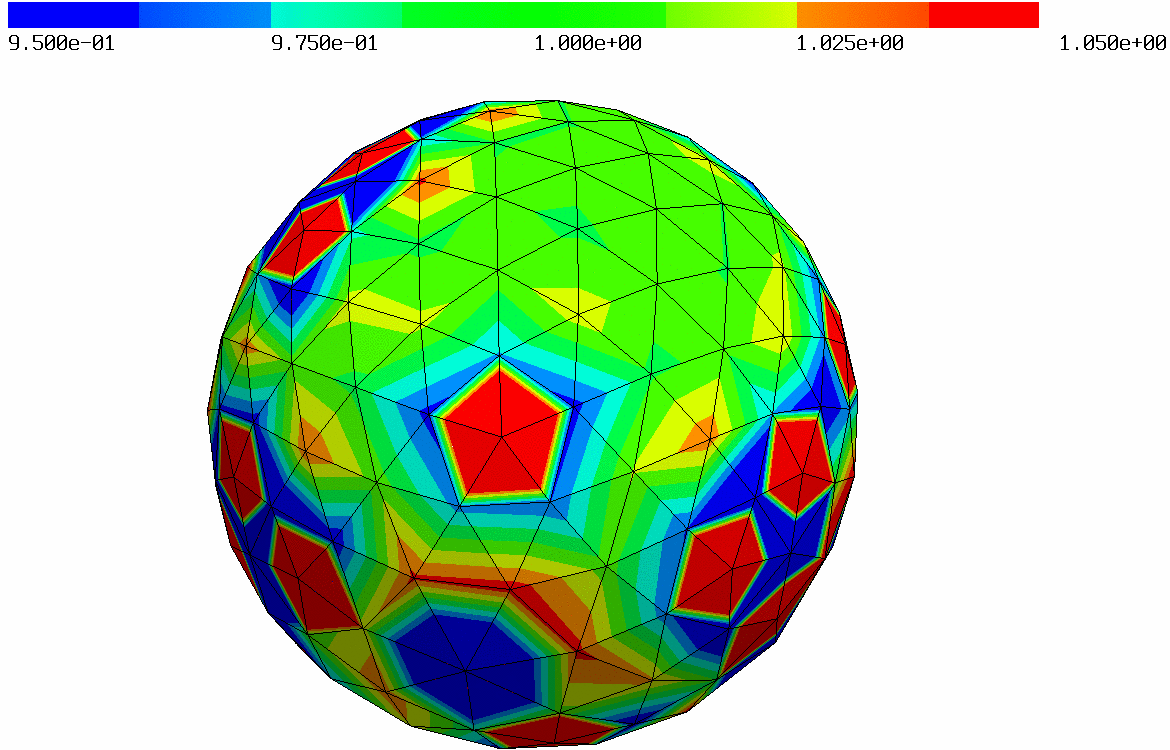}
	\includegraphics[width=0.34\textwidth]{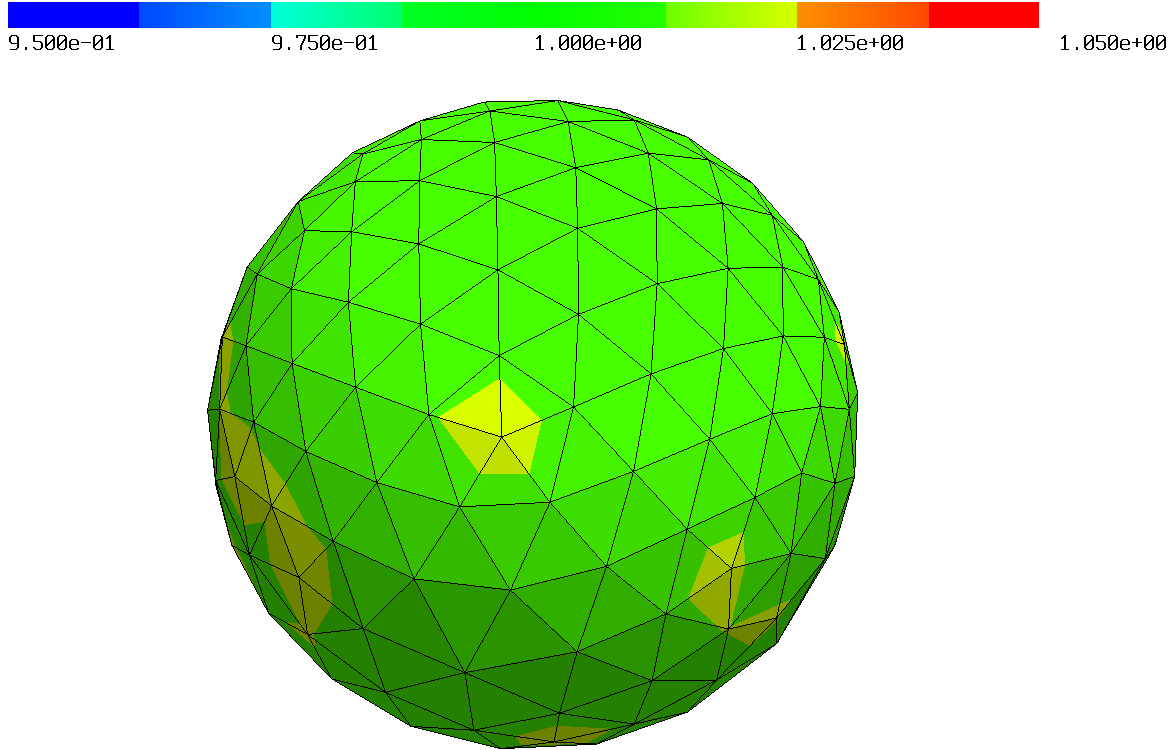}
	\caption{Left: $L^2$-error of mean curvature at unstructured meshes with linear elements measured after $1000$ shape-optimization steps with respect to number of degrees of freedom. Right: Mean curvature on initial shape and after $1000$ optimization steps.}
	\label{fig:presc_sphere_shapeopt}
\end{figure}

\subsection{Equilibrium shapes}
\label{subsec:equ_shape}
\begin{figure}[h]
	\centering
	\includegraphics[width=0.26\textwidth]{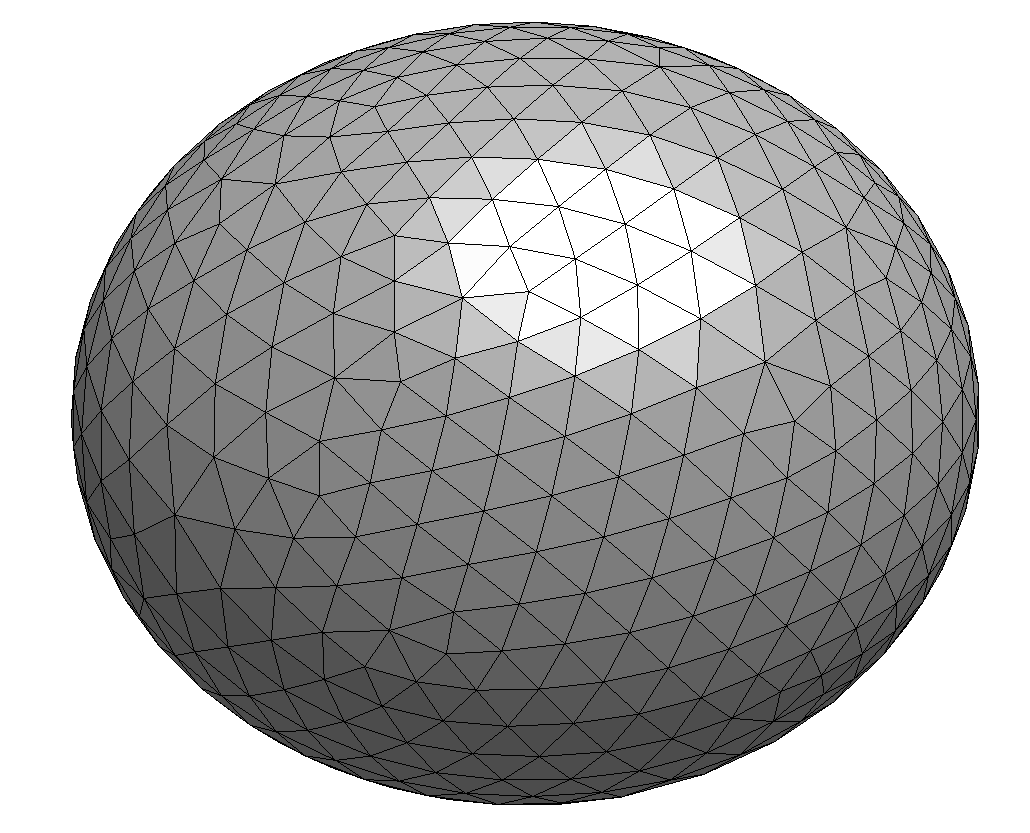}\hspace{1.5cm}
	\includegraphics[width=0.26\textwidth]{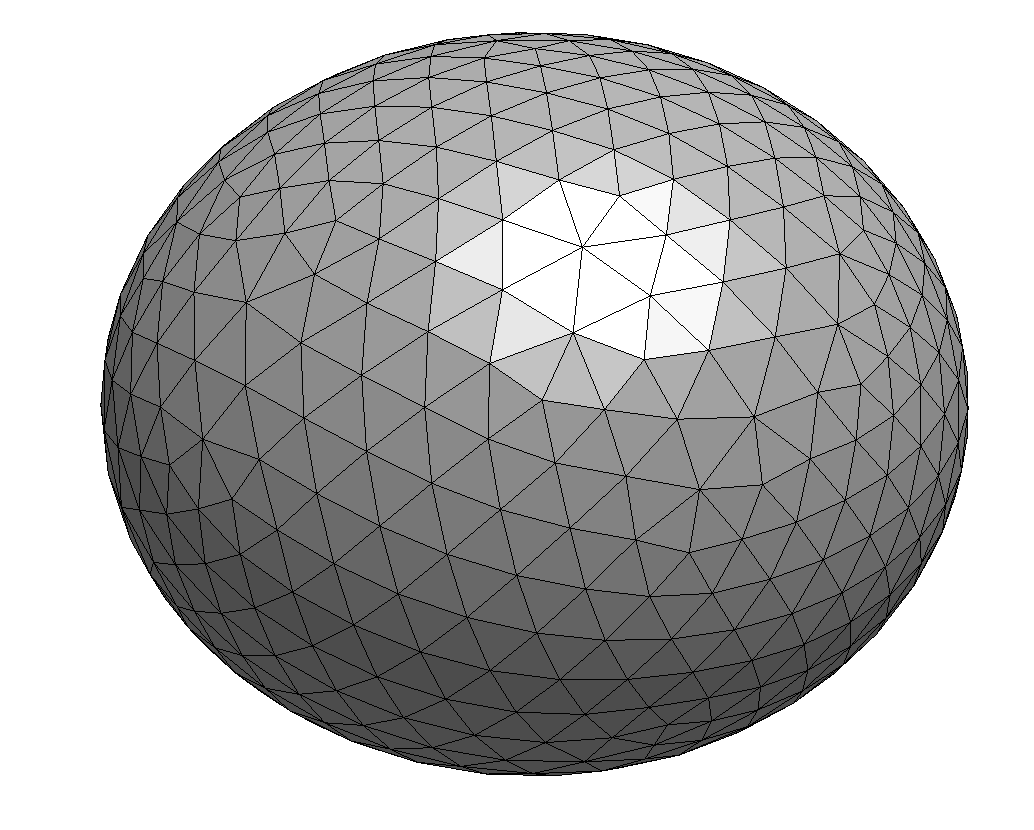}
	\caption{Initial unstructured prolate (left) and oblate (right) shapes with 1322 and 1258 triangles, respectively.}
	\label{fig:initial_shapes}
\end{figure}

In this example closed membranes with area $4\pi$ are subjected to different volume constraints leading to varying equilibrium shapes. We fix in this section the parameters $\kappa_b=0.01$, $H_0=0$, $c_V = \frac{1}{|V_0|}$, $c_A = \frac{2}{|A_0|}$, and $c_{A_{\mathrm{loc},T}}=\frac{1}{|T_0|}$ following \cite{a_BILIKO_2020a}. As we have zero spontaneous curvature, $H_0=0$, the equilibrium shapes are axisymmetric and reference computations which can be calculated analytically by solving a 2D Euler-Lagrange ODE as in \cite{DH76,SBL91}. To reproduce the phase diagram from \cite{SBL91} we start with an ellipsoid centered at the origin with semi-axes $a$, $b$, and $c$. By using a prolate ($c=1.1017$, $a=b=0.95$) and an oblate ($c=0.9$, $a=b=1.5065$) as initial shapes, see Figure~\ref{fig:initial_shapes}, we cover the prolate and oblate branches. When starting from a sphere it would be crucial to add a random noise on the initial shape to enable the shape optimization algorithm to deform the shape. As we consider unstructured grids for the oblate and prolate initial shape  non-axisymmetric deformations are induced without additionally introducing noise. Otherwise it may happen that the initial shape gets directly stuck in a  local minimum. In the phase diagram the normalized Canham-Helfrich-Evans energy \eqref{eq:canham_helfrich_energy} with respect to the bending energy of a perfect sphere
\begin{align}
E^*=E/(8\pi \kappa_b)
\end{align}
is plotted against the so-called reduced volume
\begin{align}
\bar{\nu}:= V\Big/\left(\frac{4\pi}{3}\sqrt{\frac{A}{4\pi}}^3\right).
\end{align}

Stable branches are given by prolate for $0.652 < \bar{\nu} < 1$, oblate between $0.592 < \bar{\nu} < 0.651$, and stomatocytes otherwise $0<\bar{\nu}<0.591$ \cite{SBL91}.

For the first test we consider the lowest order method \eqref{eq:lag_func_curv_shape_der_lo} on the initial shapes from Figure~\ref{fig:initial_shapes} and then apply the shape optimization Algorithm~\ref{alg:gradient_alorithm} with the conservative choice of $\alpha=0.025$, $N_{\mathrm{max}}=100000$, $\delta=1\times 10^{-7}$ without increasing $\alpha$ after an accepted step.

\begin{figure}[h]
	\centering
	\includegraphics[width=0.4\textwidth]{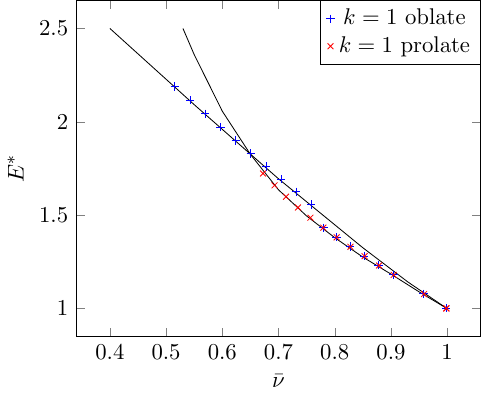}
	\includegraphics[width=0.4\textwidth]{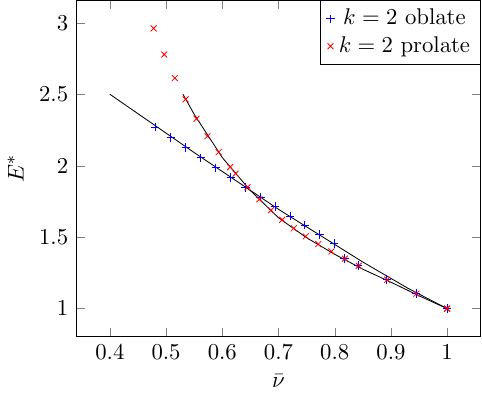}
	\caption{Results of equilibrium shapes for polynomial orders $k=1$ (left) and $k=2$ (right) for prolate and oblate shapes subjected to different volume constraints.}
	\label{fig:result_equ_shape}
\end{figure}

\begin{figure}[h]
	\centering
	\includegraphics[width=0.4\textwidth]{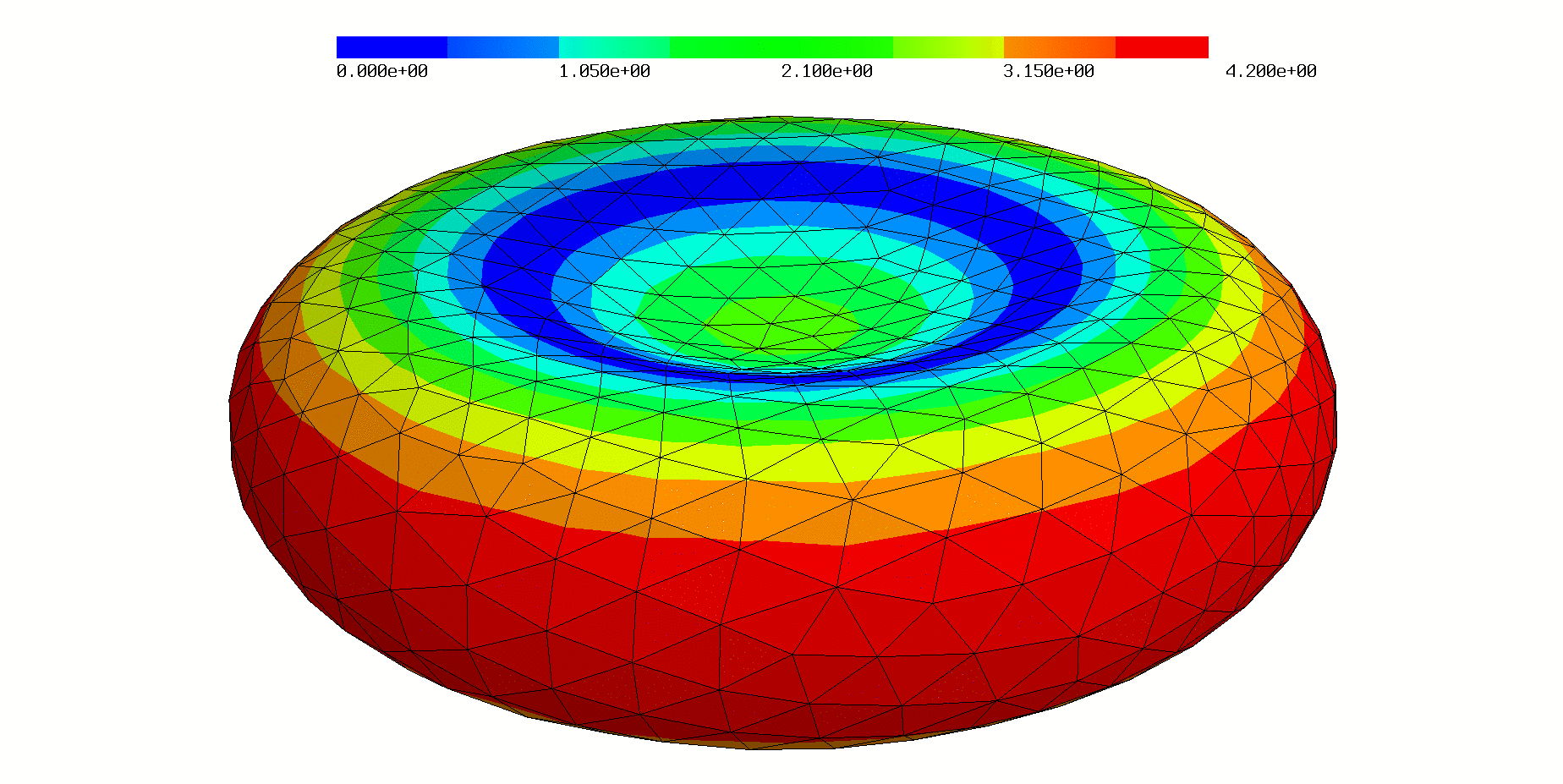}\hfill
	\includegraphics[width=0.4\textwidth]{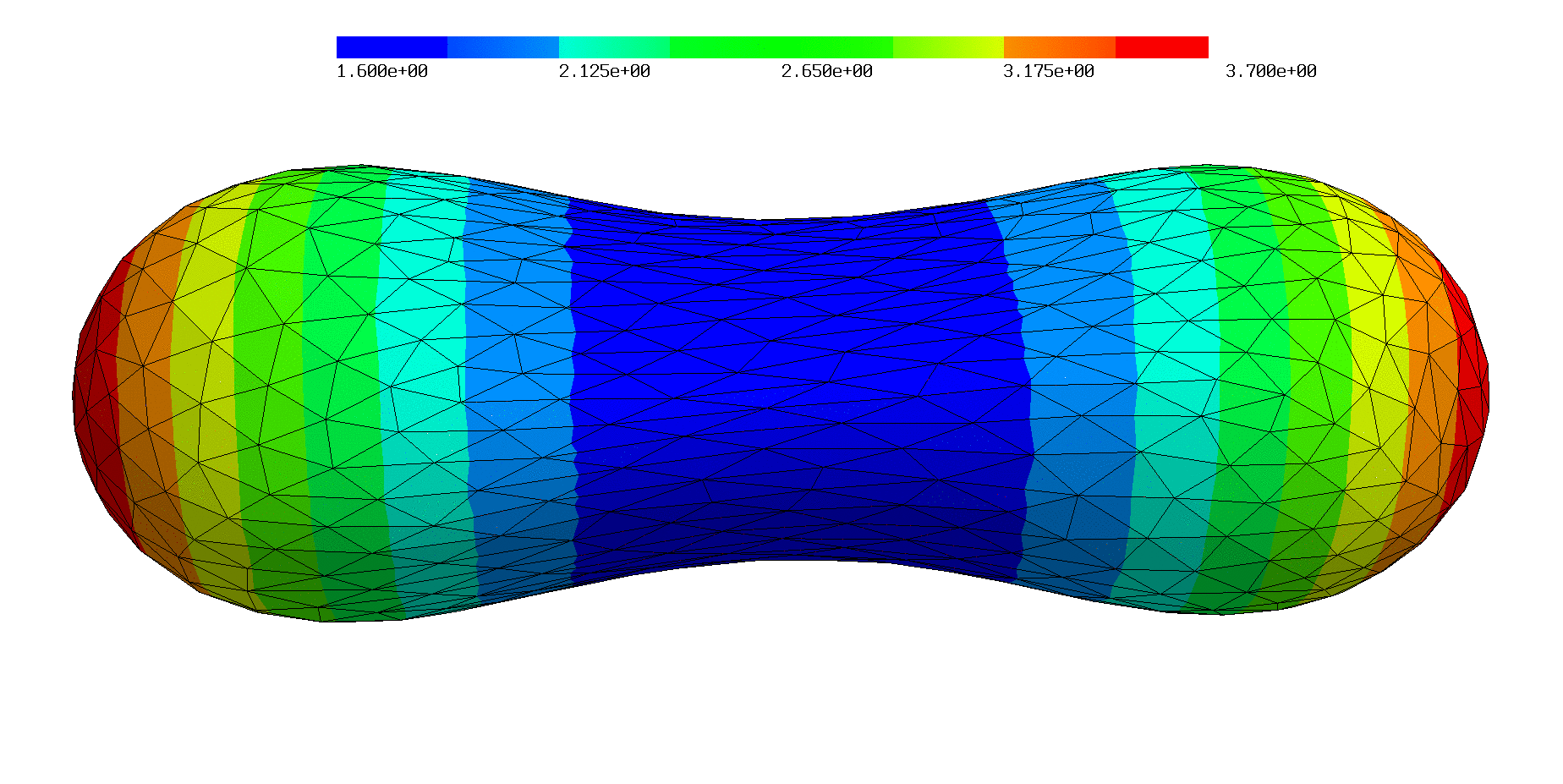}
	
	\caption{Characteristic solutions for oblate (left) and prolate (right) branch (red blood cell, dumbbell) with polynomial order $k=1$, and $\bar{\nu}=0.597$ and $\bar{\nu}=0.713$, respectively.}
	\label{fig:char_solution_prol_obl}
\end{figure}

The results depicted in Figure~\ref{fig:result_equ_shape} (left) are in good agreement with the theoretical branches, see Figure~\ref{fig:char_solution_prol_obl} for the characteristic solutions for the oblate and prolate branches.

With the oblate initial shape we were able to converge towards the unstable oblate branch for $0.652<\bar{\nu}<0.775$, however, for larger reduced volume $\bar{\nu}$ after a longer computation time the oblate shape changes significantly converging to a shape on the prolate branch.
For the prolate branch leading to the characteristic dumbbell solutions the shape gets heavily stretched for $\bar{\nu} < 0.651$ and the mesh regularity becomes ill-shaped leading to an extremely small step-size and the method does not converge anymore.

\begin{figure}[h]
	\centering
	\includegraphics[width=0.4\textwidth]{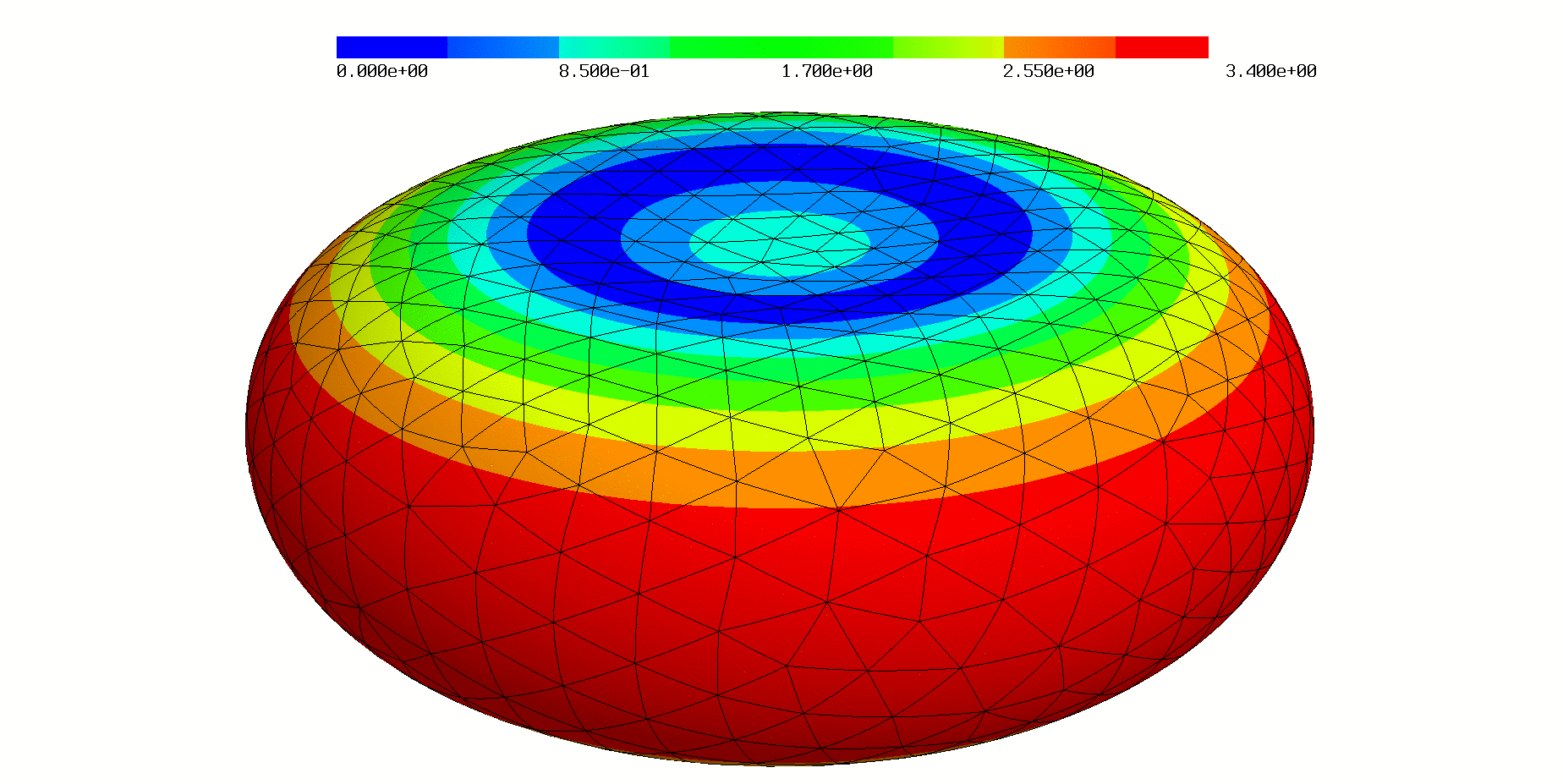}\hfill
	\includegraphics[width=0.4\textwidth]{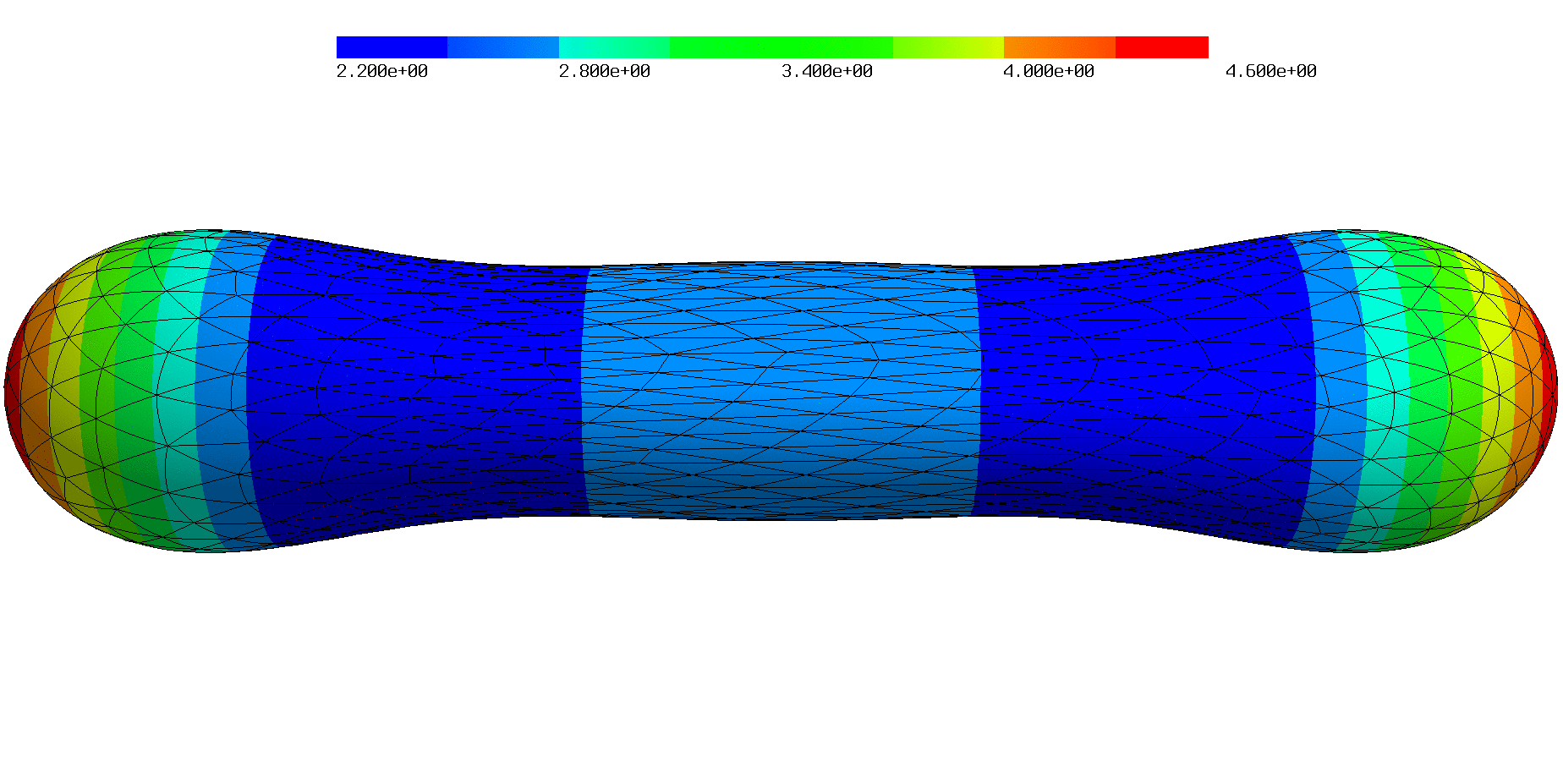}
	
	\caption{Oblate and prolate solution with polynomial order $k=2$ and $\bar{\nu}=0.773$ and $\bar{\nu}=0.594$, respectively.}
	\label{fig:sol_equ_k2}
\end{figure}

Next, we repeat the same experiments, however, with quadratic instead of linear polynomials. Further the meshes are now curved accordingly. Further, we use the improved area preservation procedure described in Section~\ref{subsec:improved_area_cons} solving a Stokes system in each iteration step. As depicted in Figure~\ref{fig:result_equ_shape} (right) the same qualitative solutions are produced, however, due to the higher polynomial degree the mesh becomes more robust with respect to the mesh quality and thus, we can follow the (unstable) prolate branch longer than for $k=1$. Further, we are also able to stay longer on the unstable oblate branch for $\bar{\nu}>0.652$, compare Figure~\ref{fig:sol_equ_k2} for two converged shapes.

\subsection{Spontaneous curvature $H_0$}
\label{subsec:spont_curv}
In this benchmark we consider the same parameters as in the previous section with the only difference of non-zero spontaneous curvature, $H_0\neq 0$. More precisely, we set $H_0=1.2$ and $H_0=1.5$ to reproduce the phase diagrams in \cite{SBL91} for these configurations. Further, as we expect strong deformed equilibrium shapes, we directly use quadratic elements, $k=2$, in combination with the improved area preservation procedure.

\begin{figure}[h]
	\centering
	\includegraphics[width=0.4\textwidth]{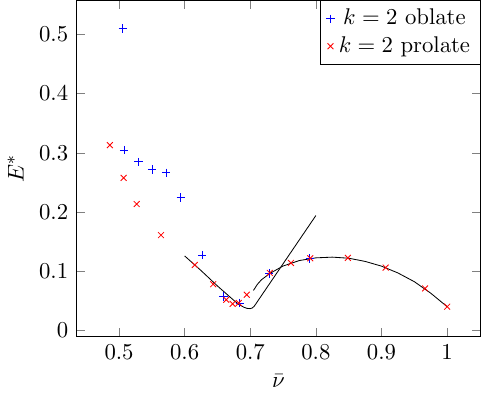}
	\includegraphics[width=0.4\textwidth]{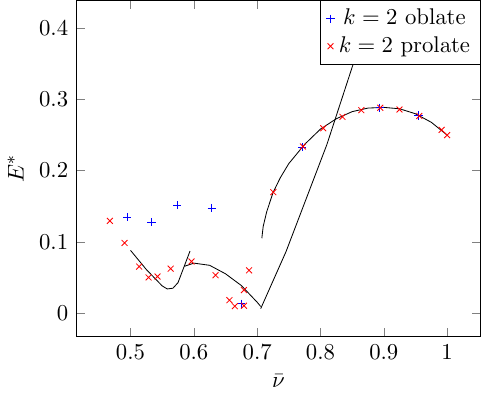}
	\caption{Results of spontaneous curvature $H_0=1.2$ (left) and $H_0=1.5$ (right) with polynomial degree $k=2$.}
	\label{fig:result_spont_p2}
\end{figure}

As depicted in Figure~\ref{fig:result_spont_p2} (left) we can reproduce the phase diagram for $H_0=1.2$ with prolate initial shapes leading to axisymmetric results, compare Figure~\ref{fig:prolate_shapes_12}. The solutions for $\bar{\nu}>0.7$ form well-shaped dumbbell shapes. At around $\bar{\nu}=0.7$ a bifurcation of branches exists leading to convergence problems. Further, the middle of the dumbbell solutions get heavily narrowed such that the mesh quality becomes critical. For decreasing reduced volume $\bar{\nu}$ the dumbbell shapes get longer and for $\bar{\nu}<0.6$ the middle radius increases not being the minimum anymore. For $\bar{\nu}<0.5$ the shape gets heavily stretched and it seems that the local maxima develop. The oblate initial shapes follow the prolate results for $\bar{\nu}>0.65$. In this benchmark, however, we observe that for $\bar{\nu}<0.65$ the unstable red blood cell type shapes converge to a not axisymmetric solution being similar to a dumbbell, but with three ends instead of two, see Figure~\ref{fig:oblate_shapes_12}. As the results are non axisymmetric this branch could not be analytically computed in \cite{SBL91}. For $\bar{\nu}\approx0.5$ the solution again seems to change. To fulfill the volume constraint the solution is nearly flattened yielding zero mean curvature around its center of gravity. We emphasize that the meshes get deformed heavily and further experiments with improved mesh regularity algorithms have to be performed in future work to investigate this branch. Further, we note that the oblate solutions for $\bar{\nu}<0.65$ did not fully converge to the equilibrate solution due to the distorted meshes, however, reflect the qualitative behavior of the exact solution.

\begin{figure}[h]
	\begin{tabular}{ccc}
		\includegraphics[width=0.33\textwidth]{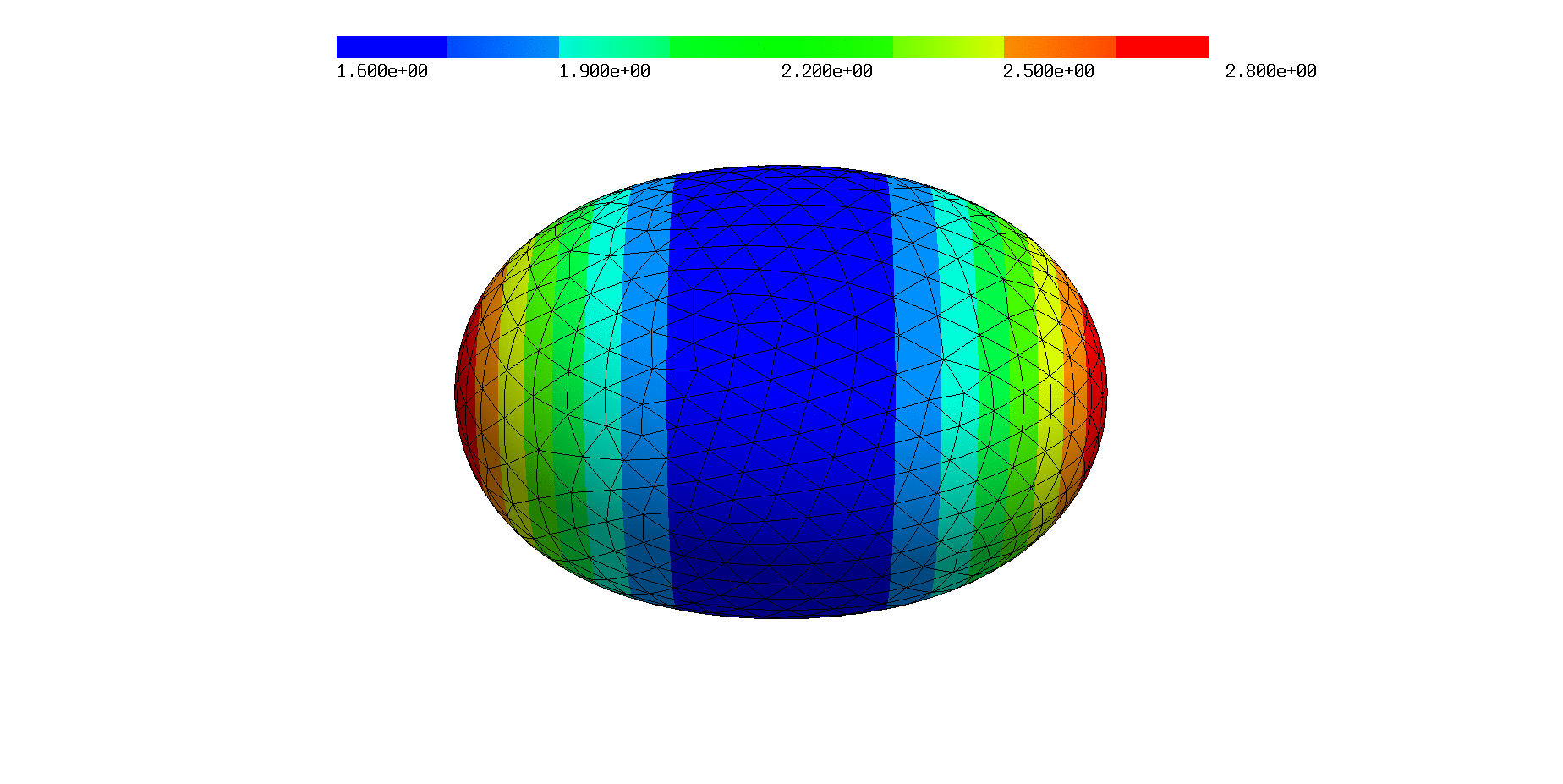}&
		\includegraphics[width=0.33\textwidth]{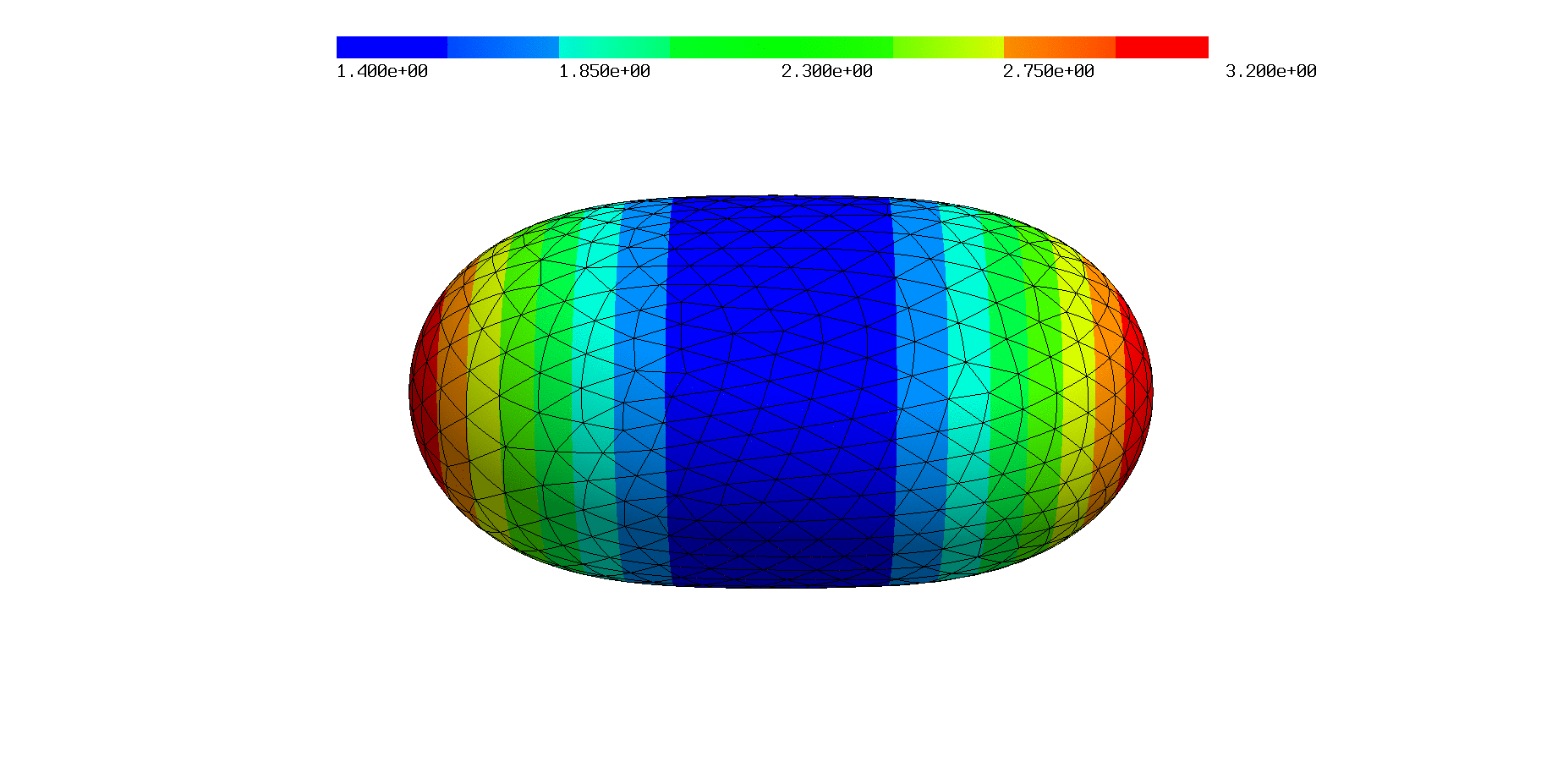}&
		\includegraphics[width=0.33\textwidth]{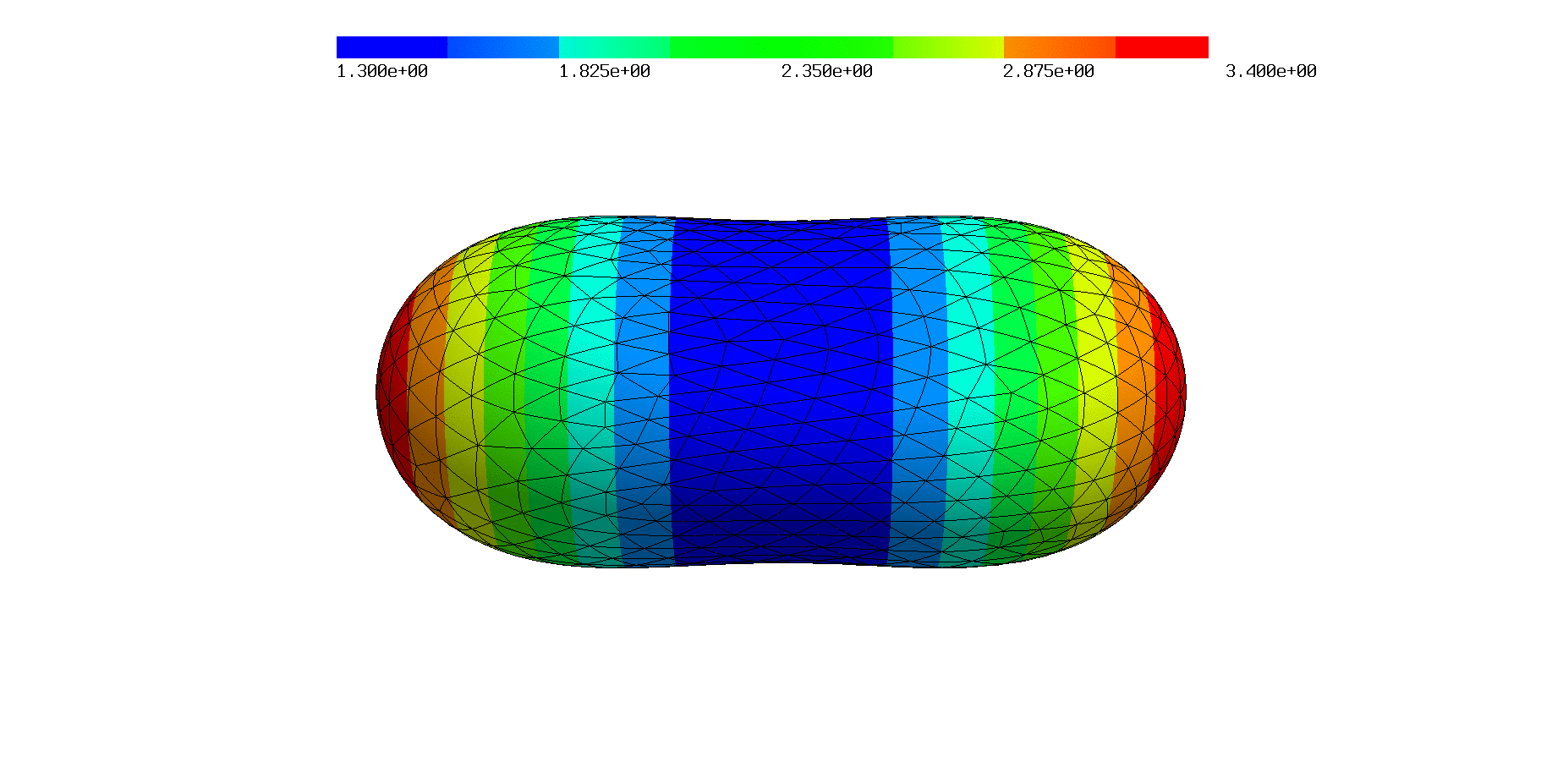}\\
		$\bar{\nu}=0.966$ & $\bar{\nu}=0.906$ & $\bar{\nu}=0.849$\\
		\includegraphics[width=0.33\textwidth]{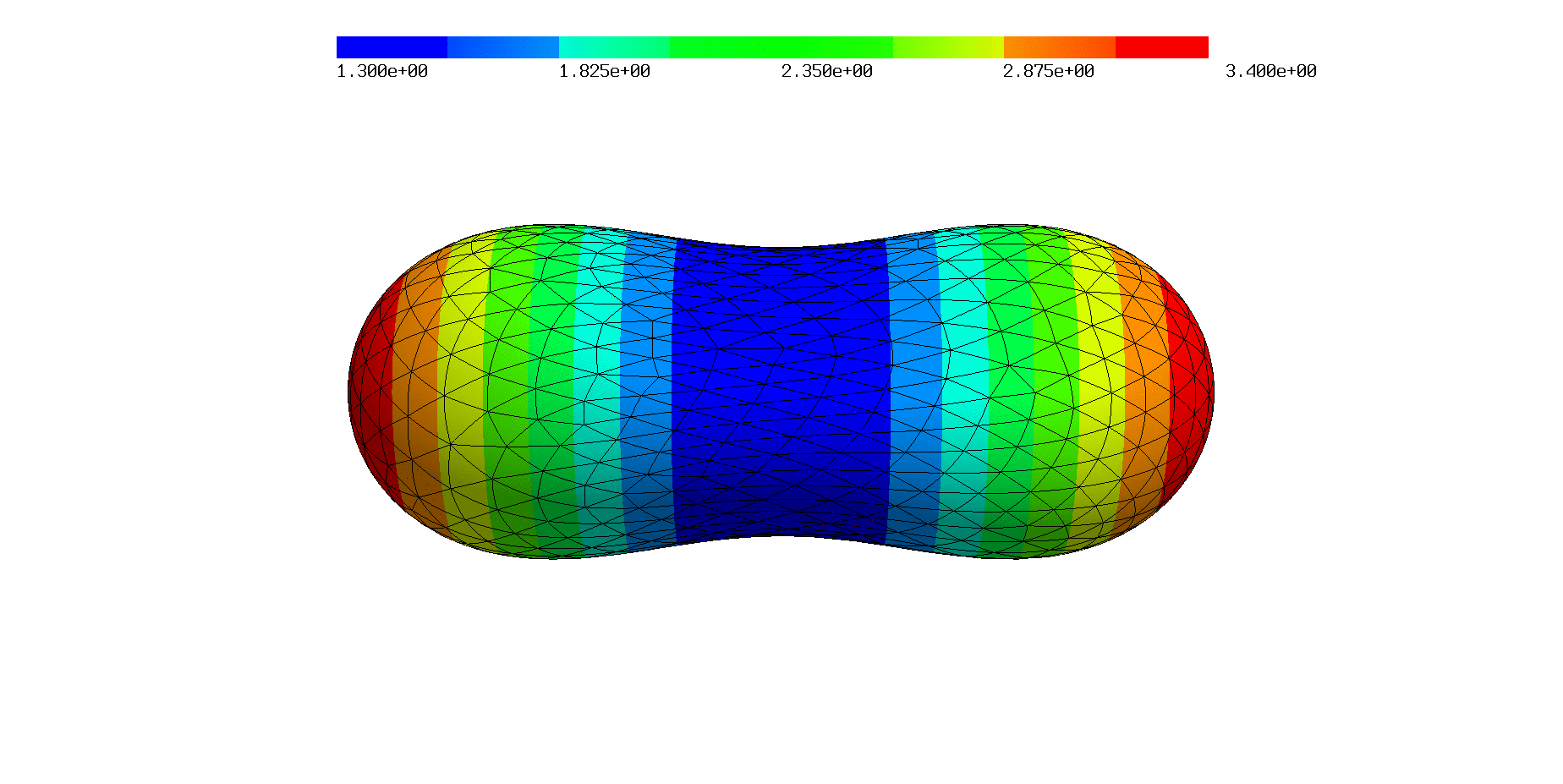}&
		\includegraphics[width=0.33\textwidth]{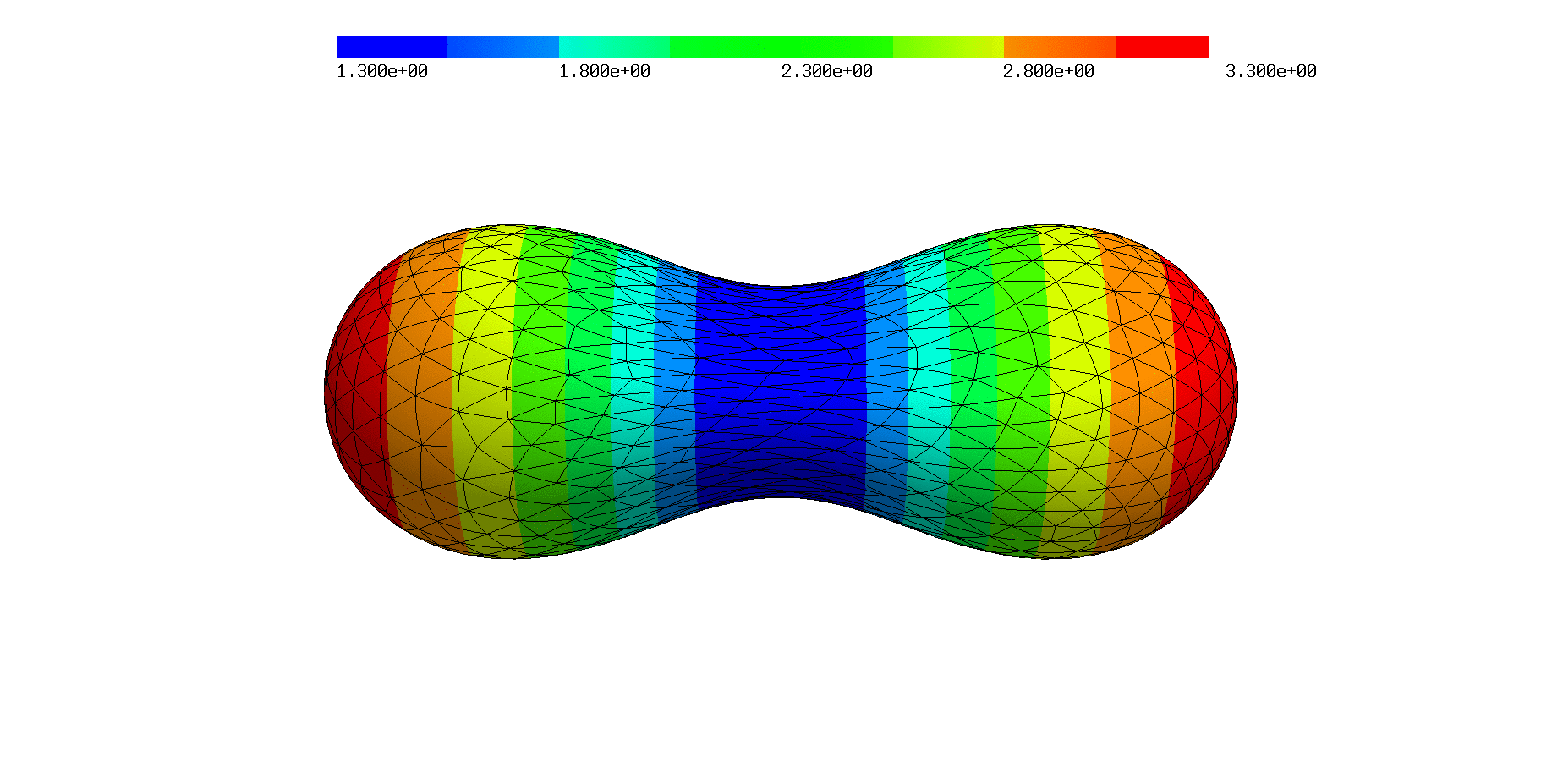}&
		\includegraphics[width=0.33\textwidth]{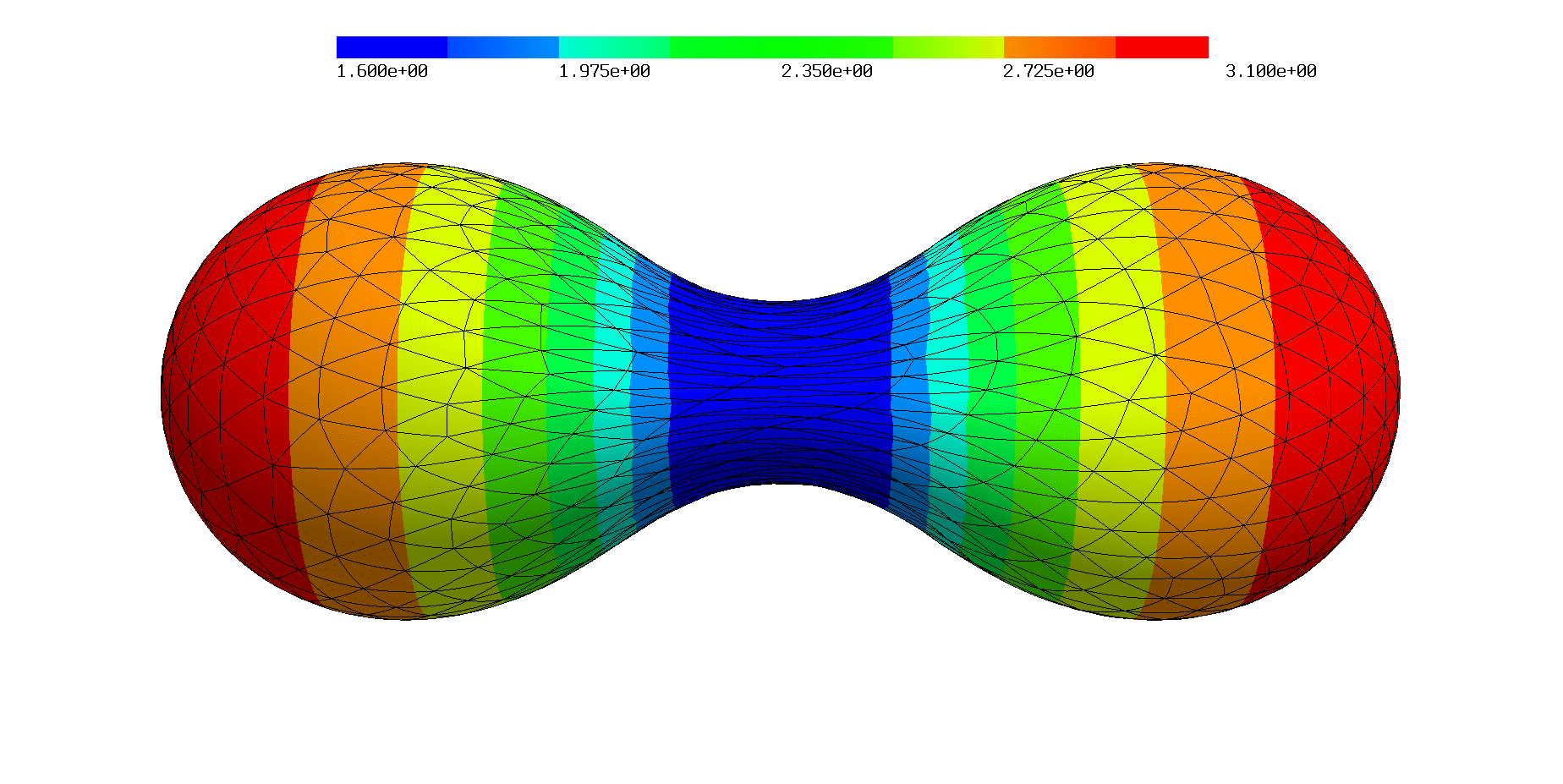}\\
		$\bar{\nu}=0.791$  & $\bar{\nu}=0.730$ & $\bar{\nu}=0.695$\\
		\includegraphics[width=0.33\textwidth]{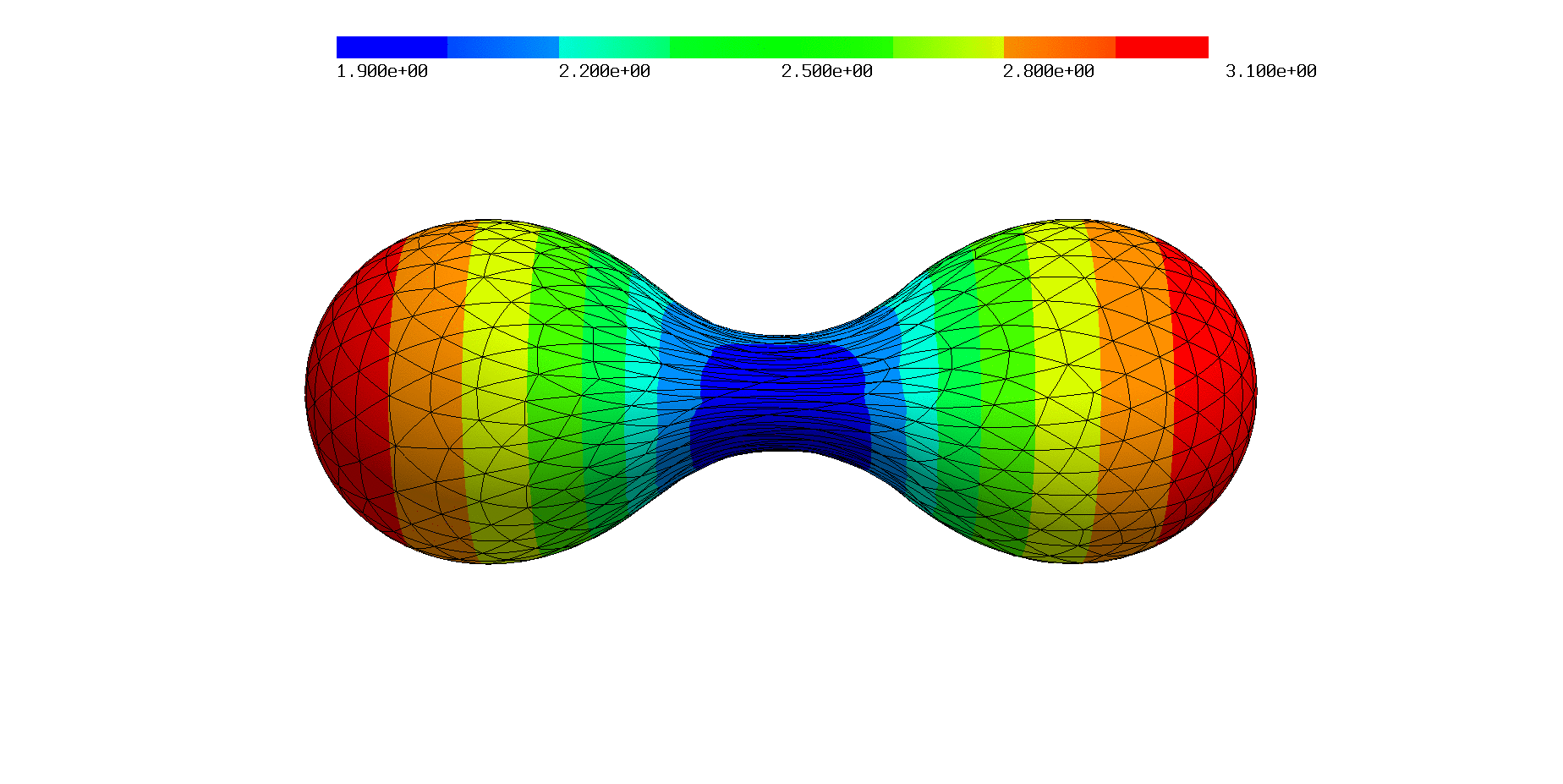}&
		\includegraphics[width=0.33\textwidth]{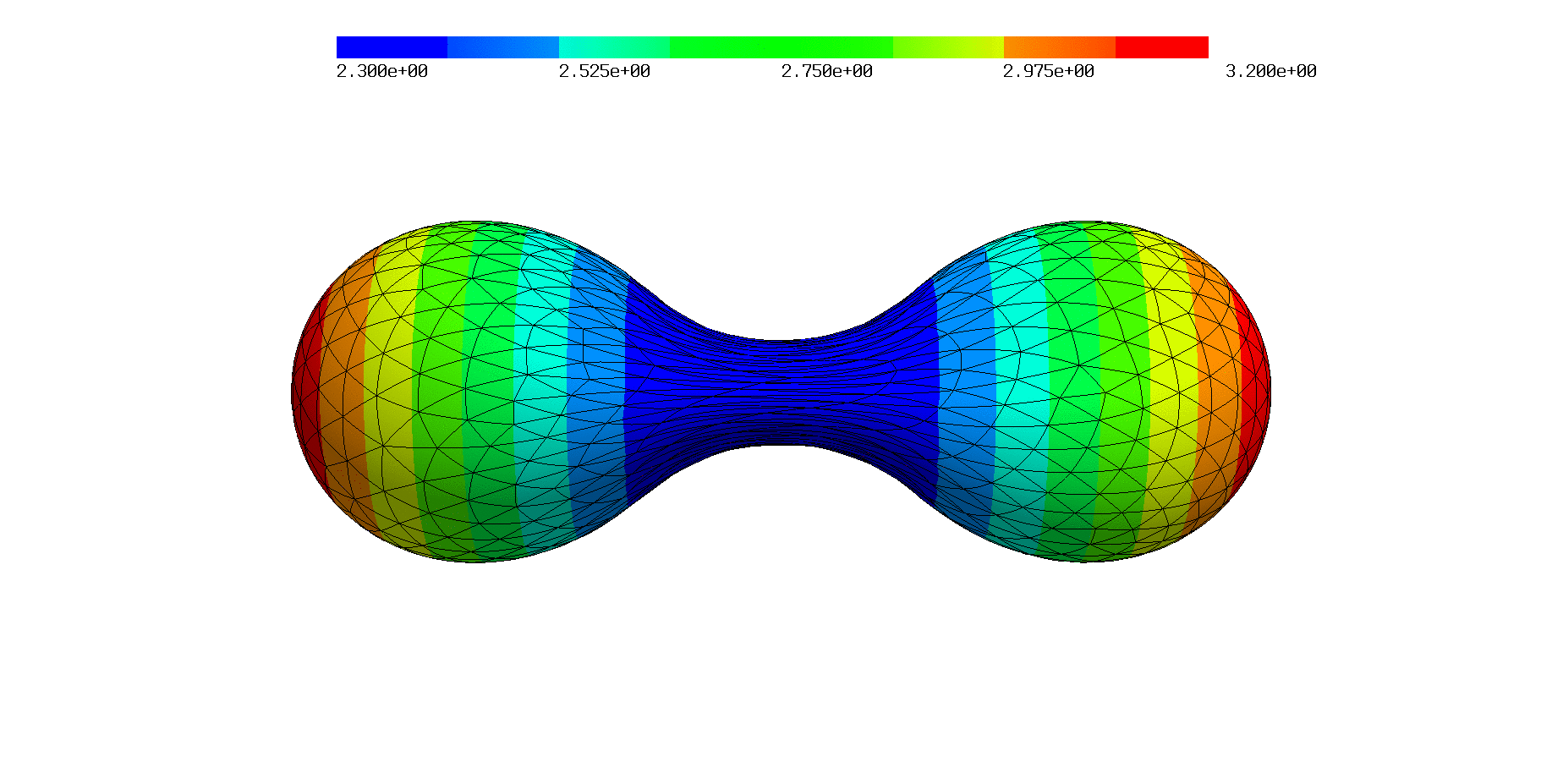}&
		\includegraphics[width=0.33\textwidth]{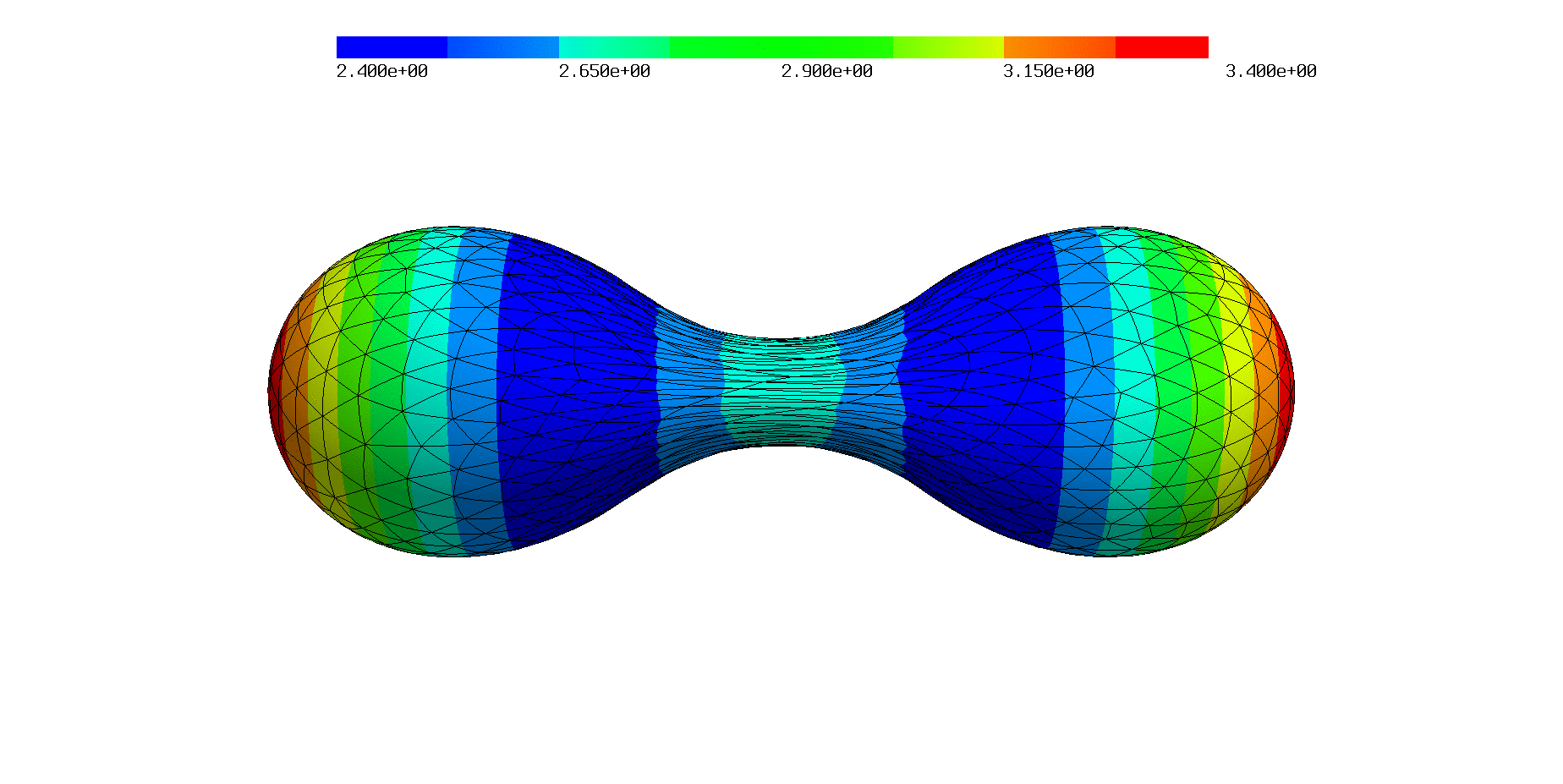}\\
		$\bar{\nu}=0.682$ & $\bar{\nu}=0.673$ & $\bar{\nu}=0.663$\\
		\includegraphics[width=0.33\textwidth]{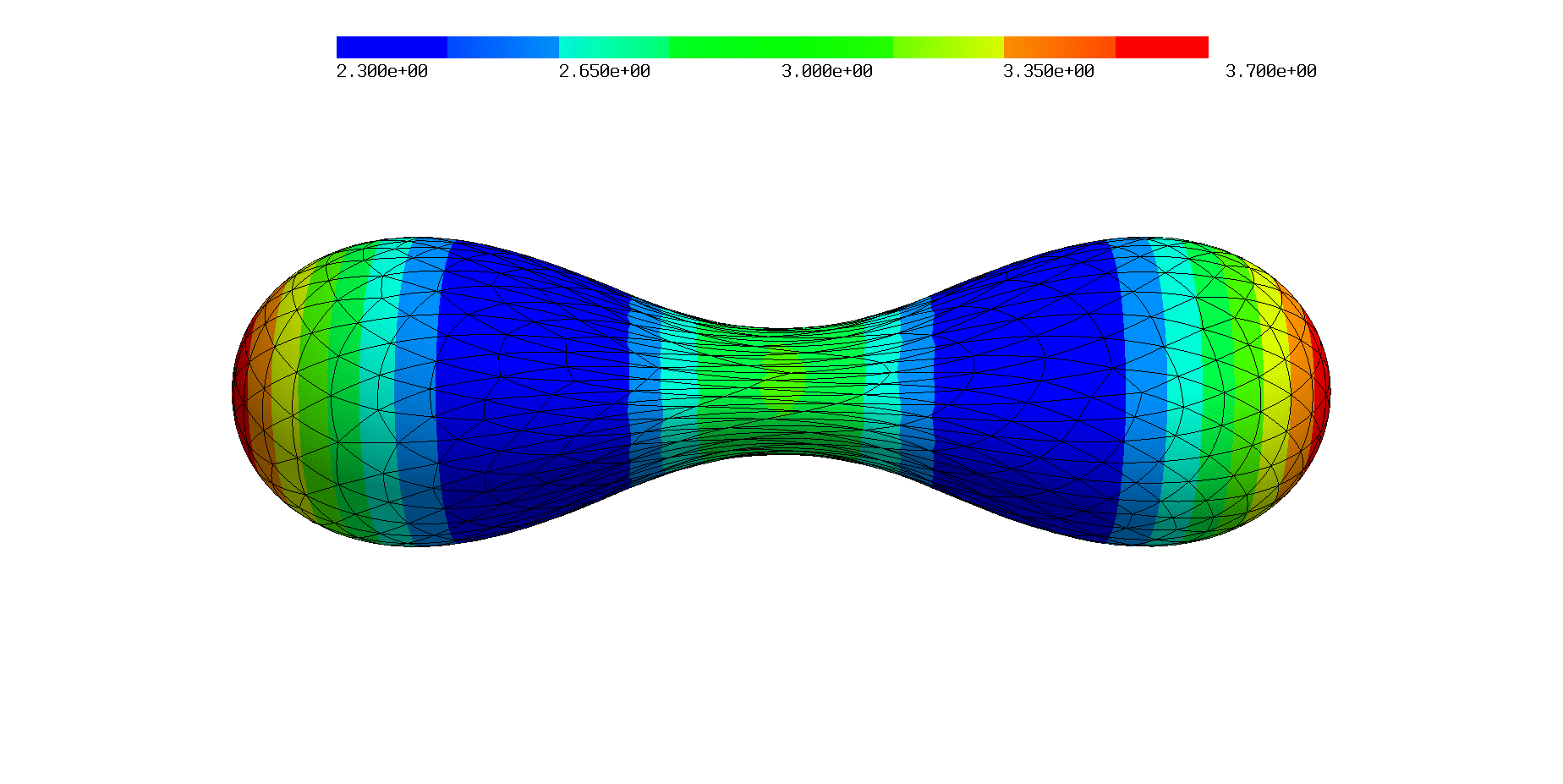}&
		\includegraphics[width=0.33\textwidth]{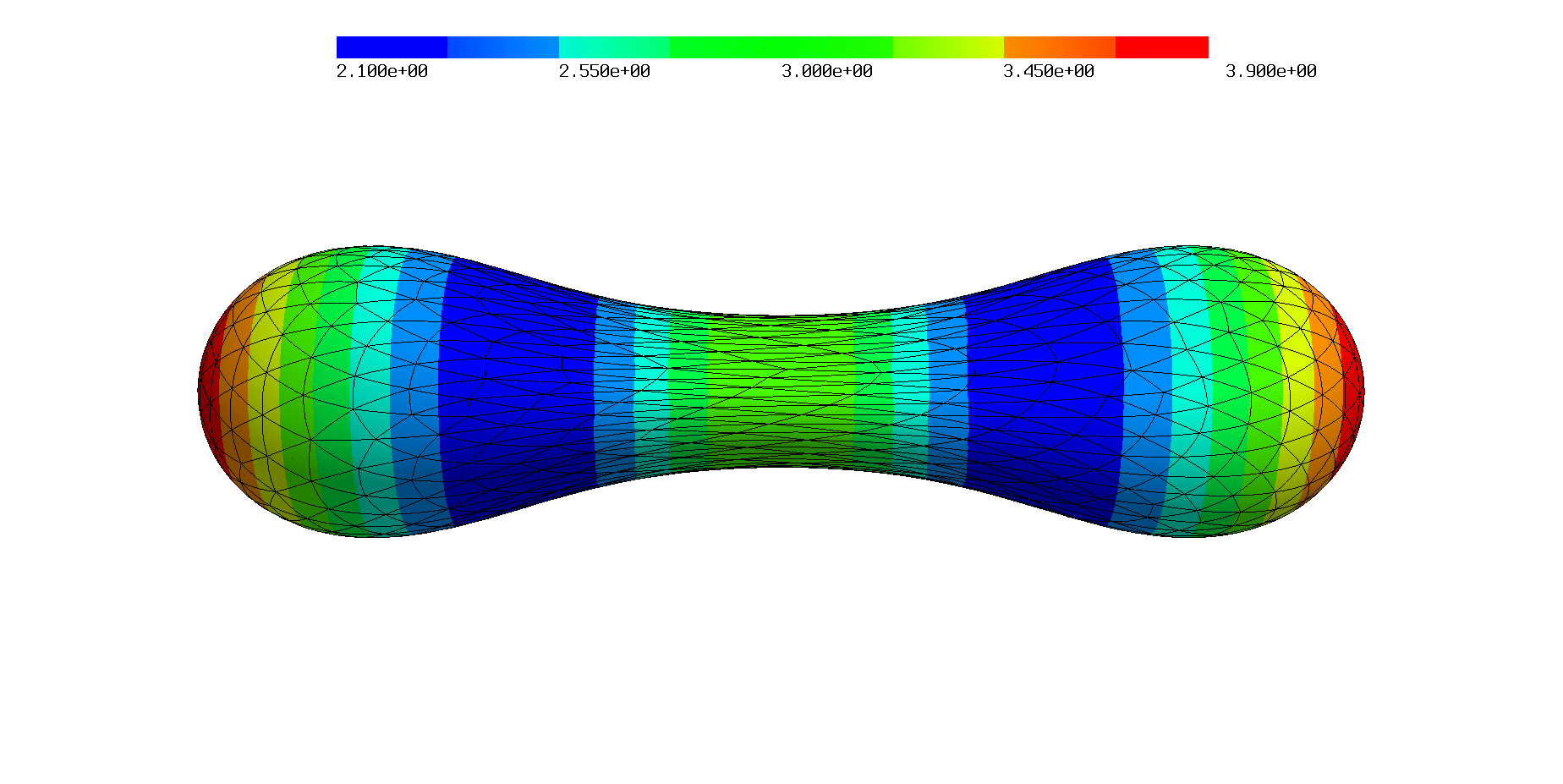}&
		\includegraphics[width=0.33\textwidth]{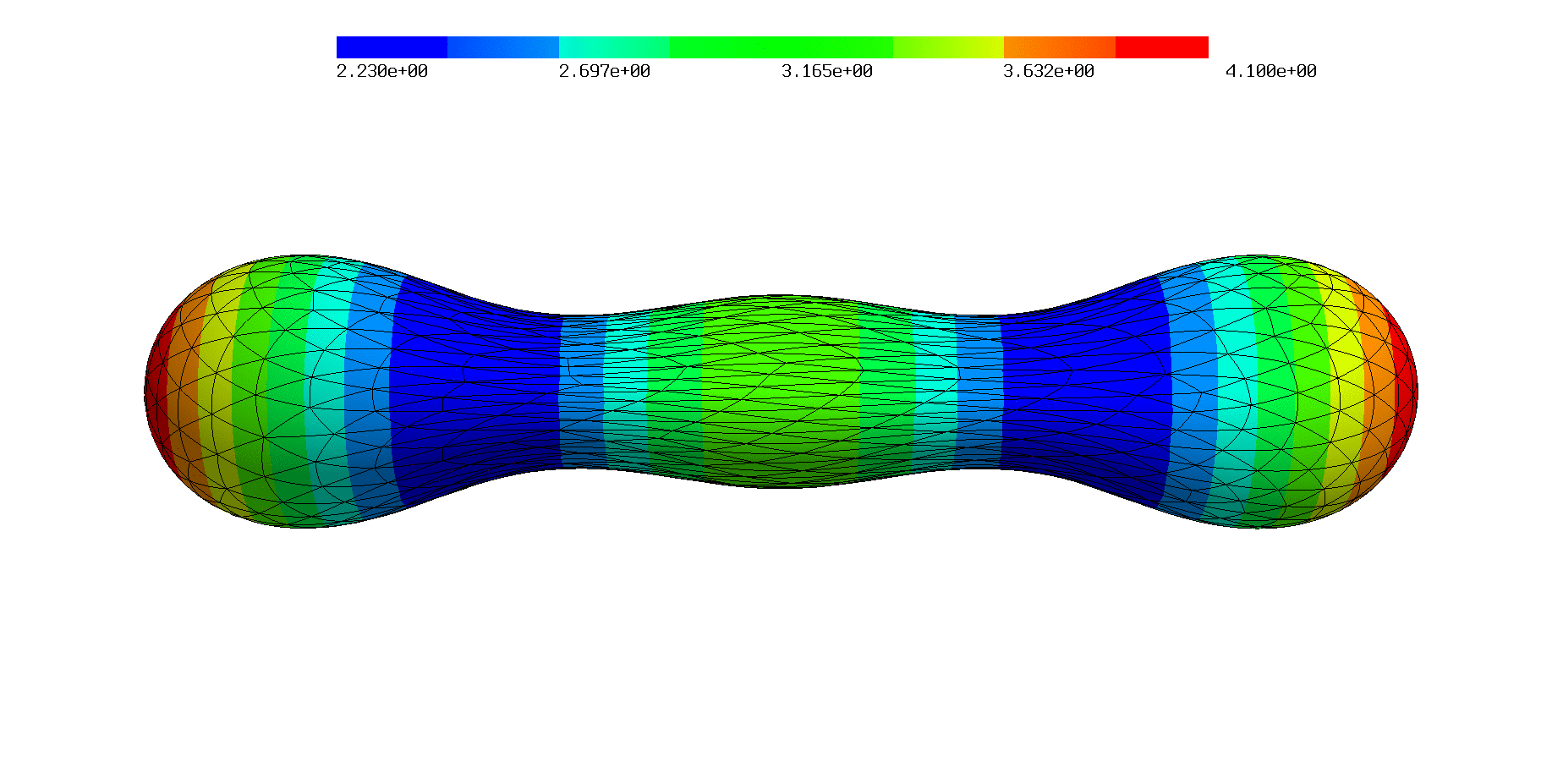}\\
		$\bar{\nu}=0.644$ & $\bar{\nu}=0.615$ & $\bar{\nu}=0.564$\\
		\includegraphics[width=0.33\textwidth]{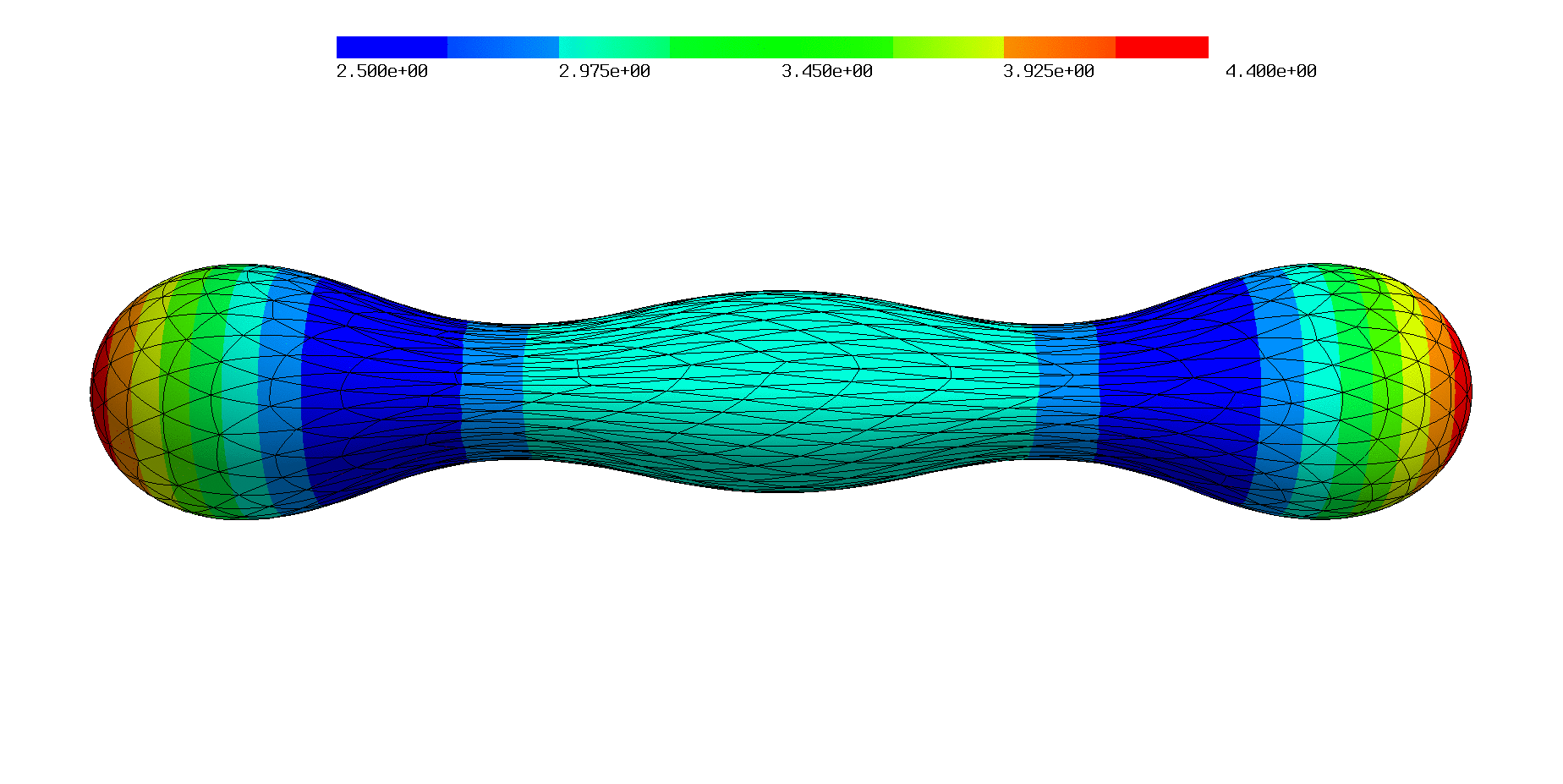}&
		\includegraphics[width=0.33\textwidth]{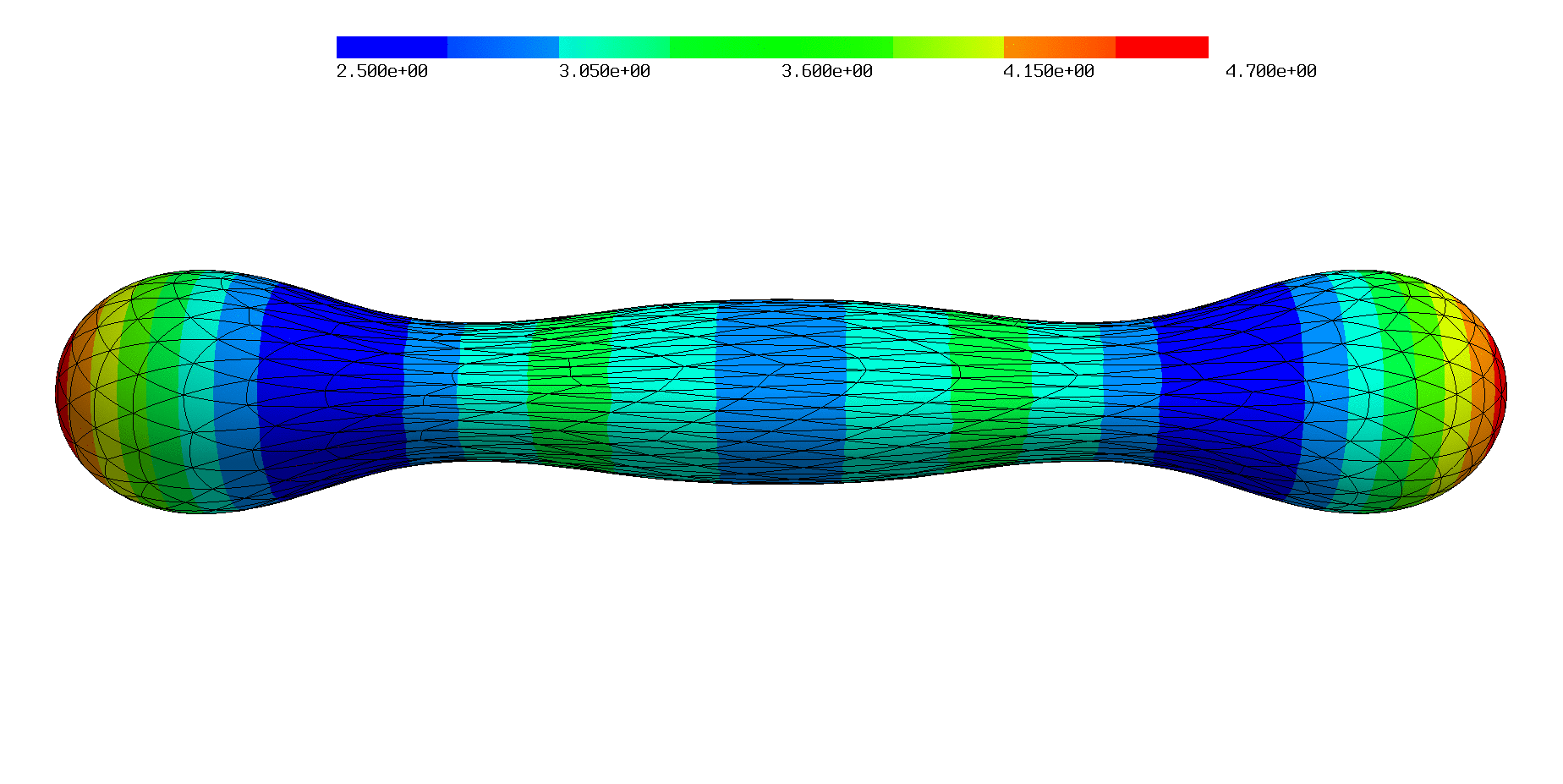}&
		\includegraphics[width=0.33\textwidth]{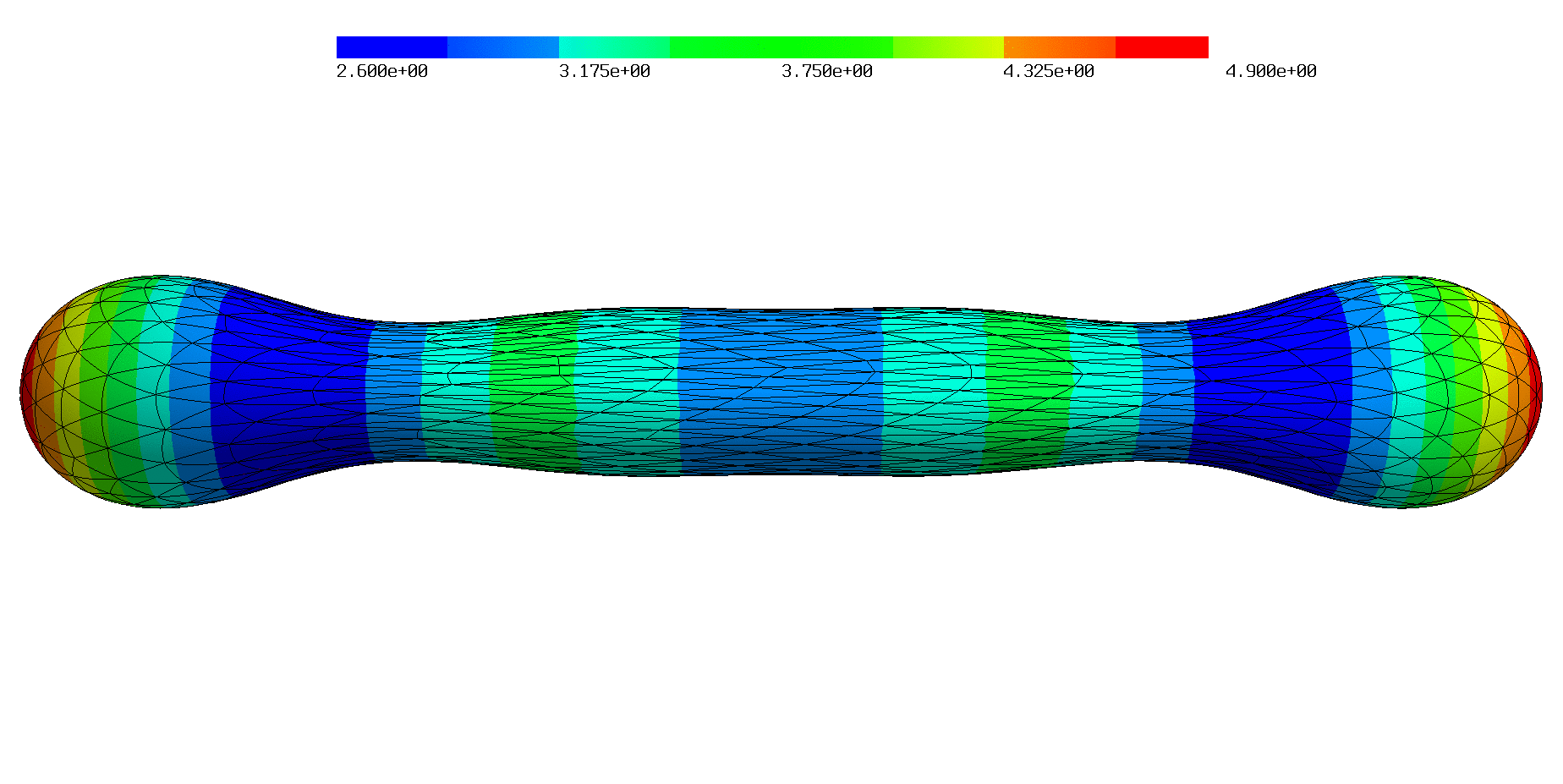}\\
		$\bar{\nu}=0.527$ & $\bar{\nu}=0.507$ & $\bar{\nu}=0.486$
	\end{tabular}
	\caption{Prolate shapes for different reduced volumes $\bar{\nu}$ with $H_0=1.2$ and polynomial order $k=2$.}
	\label{fig:prolate_shapes_12}
\end{figure}

\begin{figure}[h]
	\begin{tabular}{ccc}
		\includegraphics[width=0.33\textwidth]{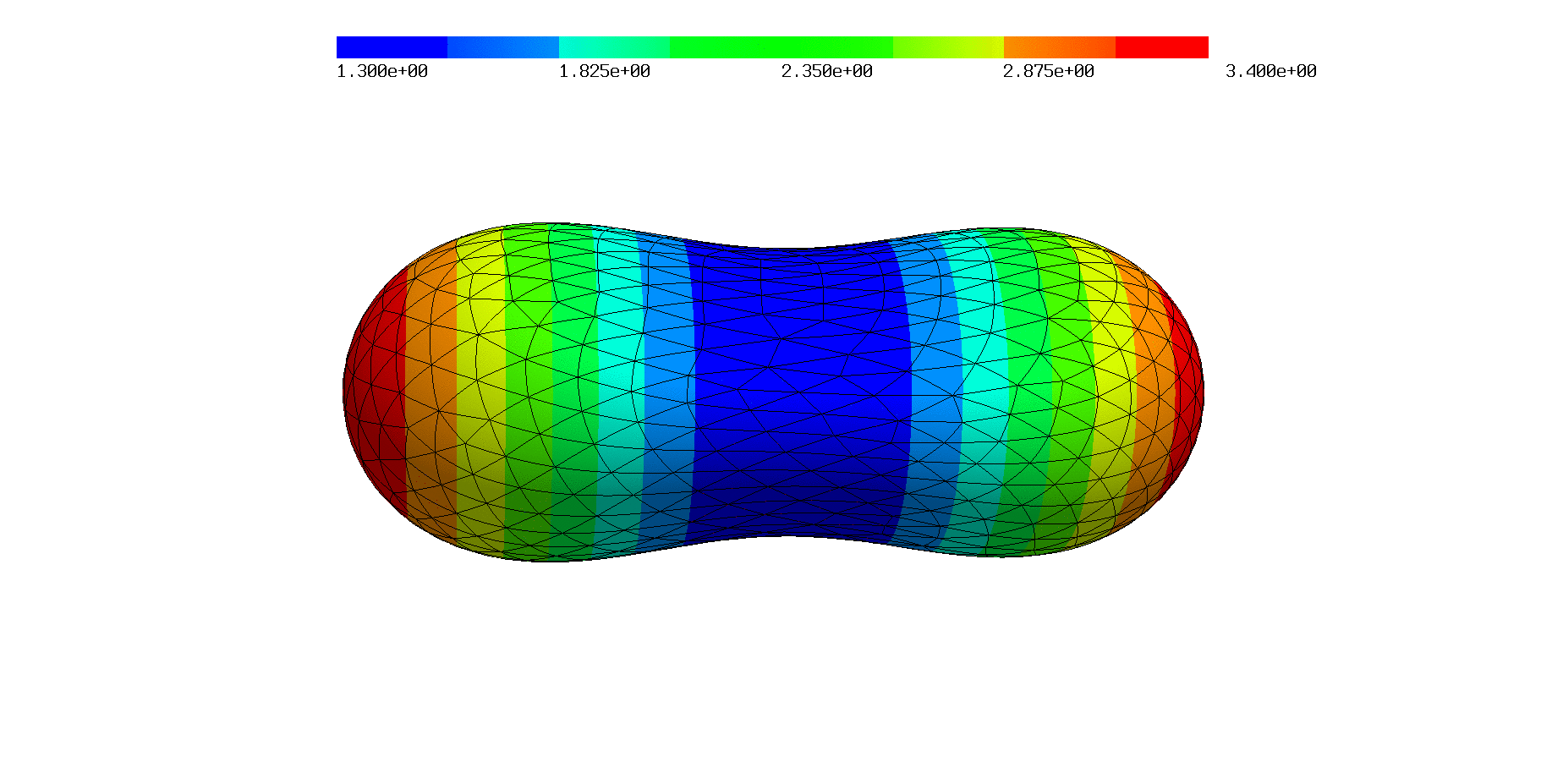}&
		\includegraphics[width=0.33\textwidth]{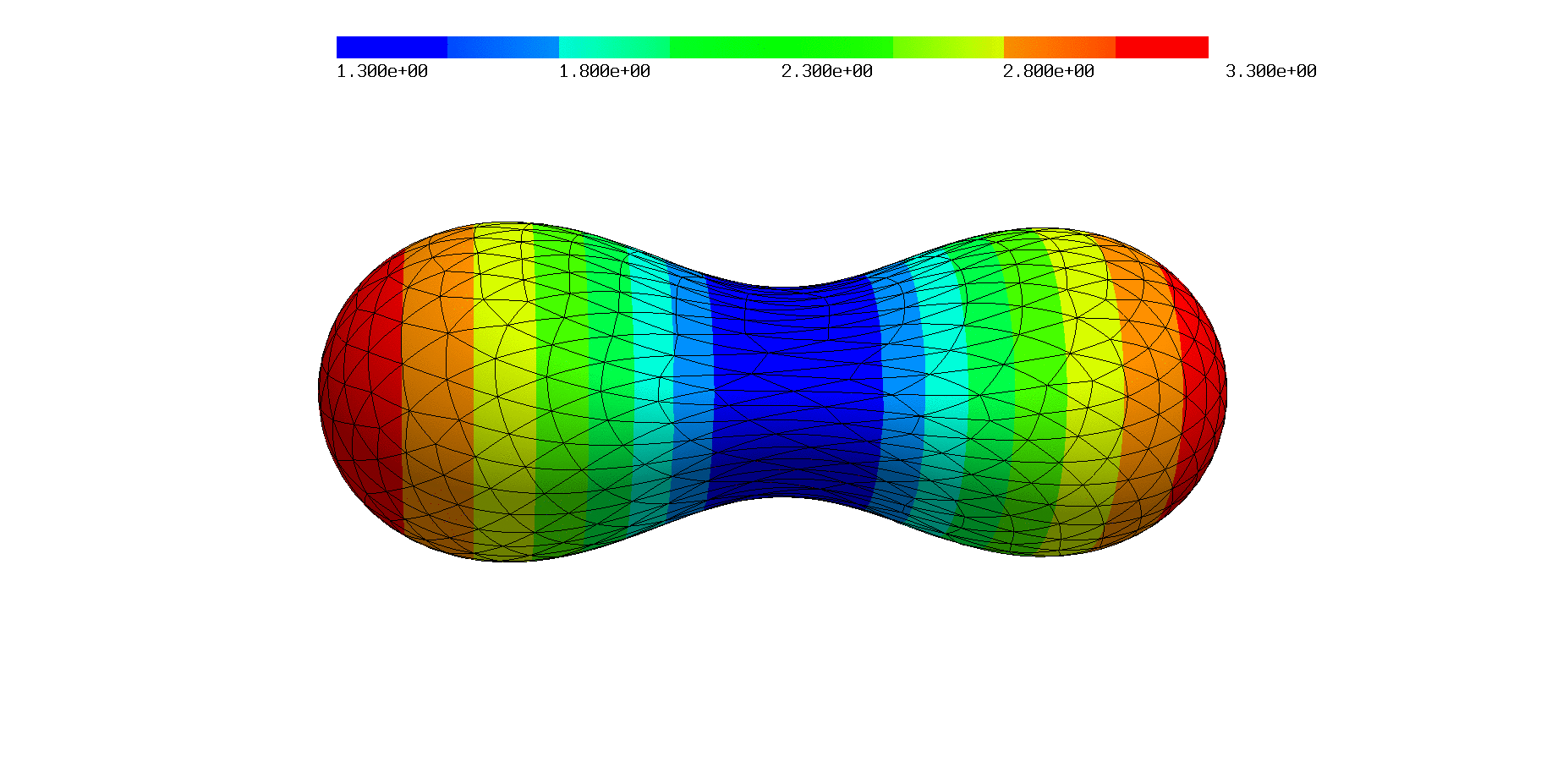}&
		\includegraphics[width=0.33\textwidth]{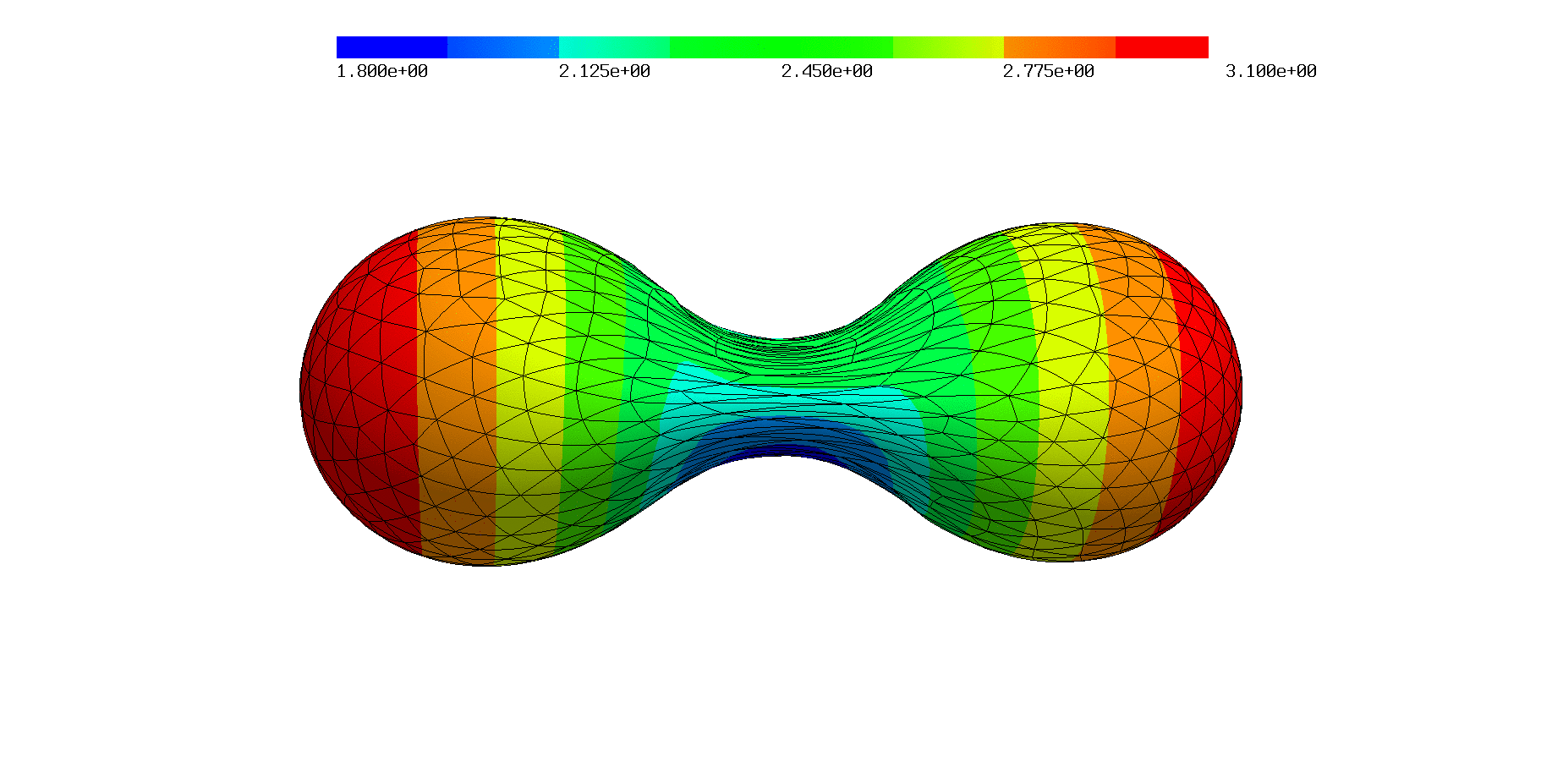}\\
		$\bar{\nu}=0.791$ & $\bar{\nu}=0.729$ & $\bar{\nu}=0.684$\\
		\includegraphics[width=0.33\textwidth]{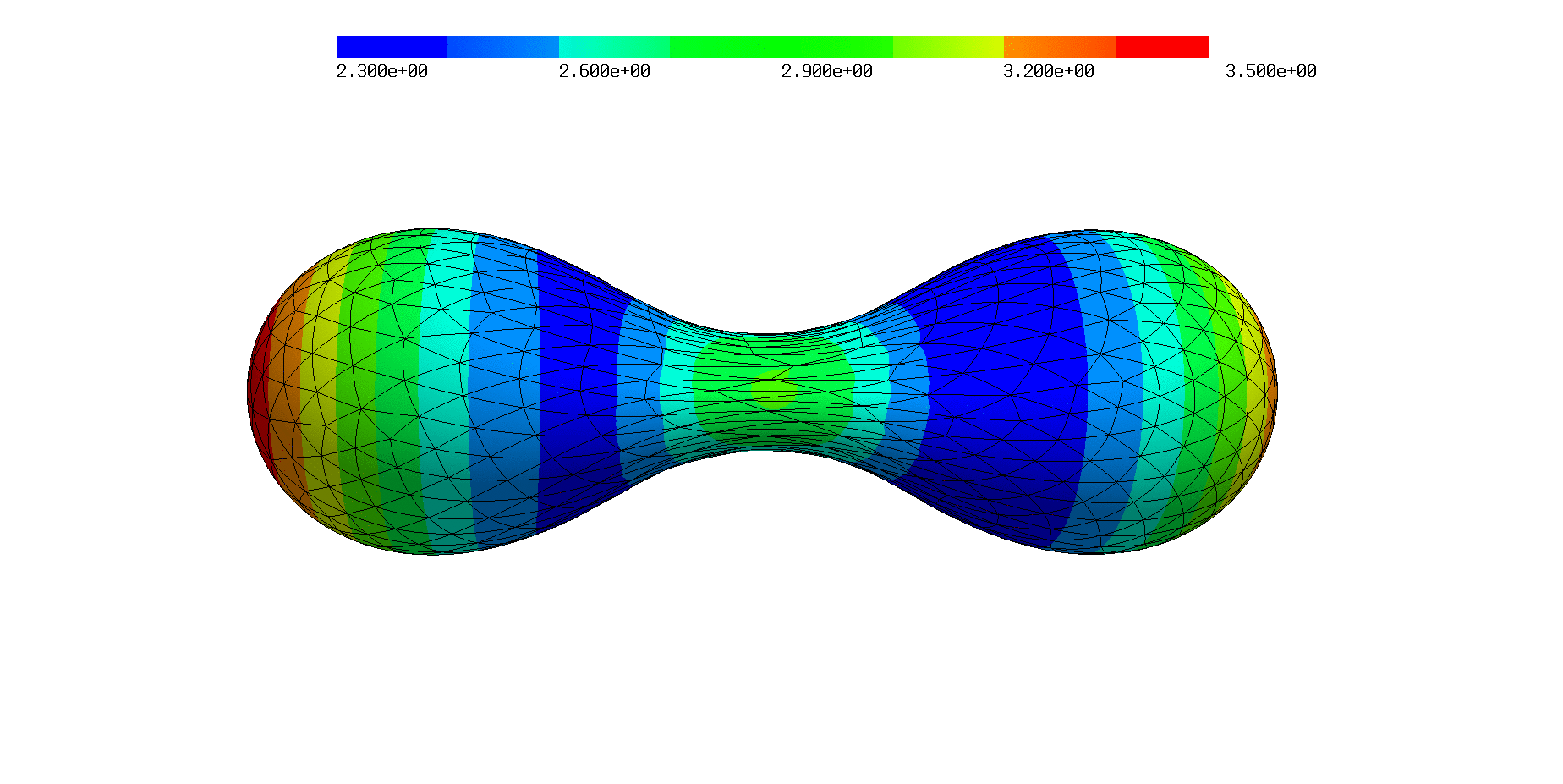}&
		\includegraphics[width=0.33\textwidth]{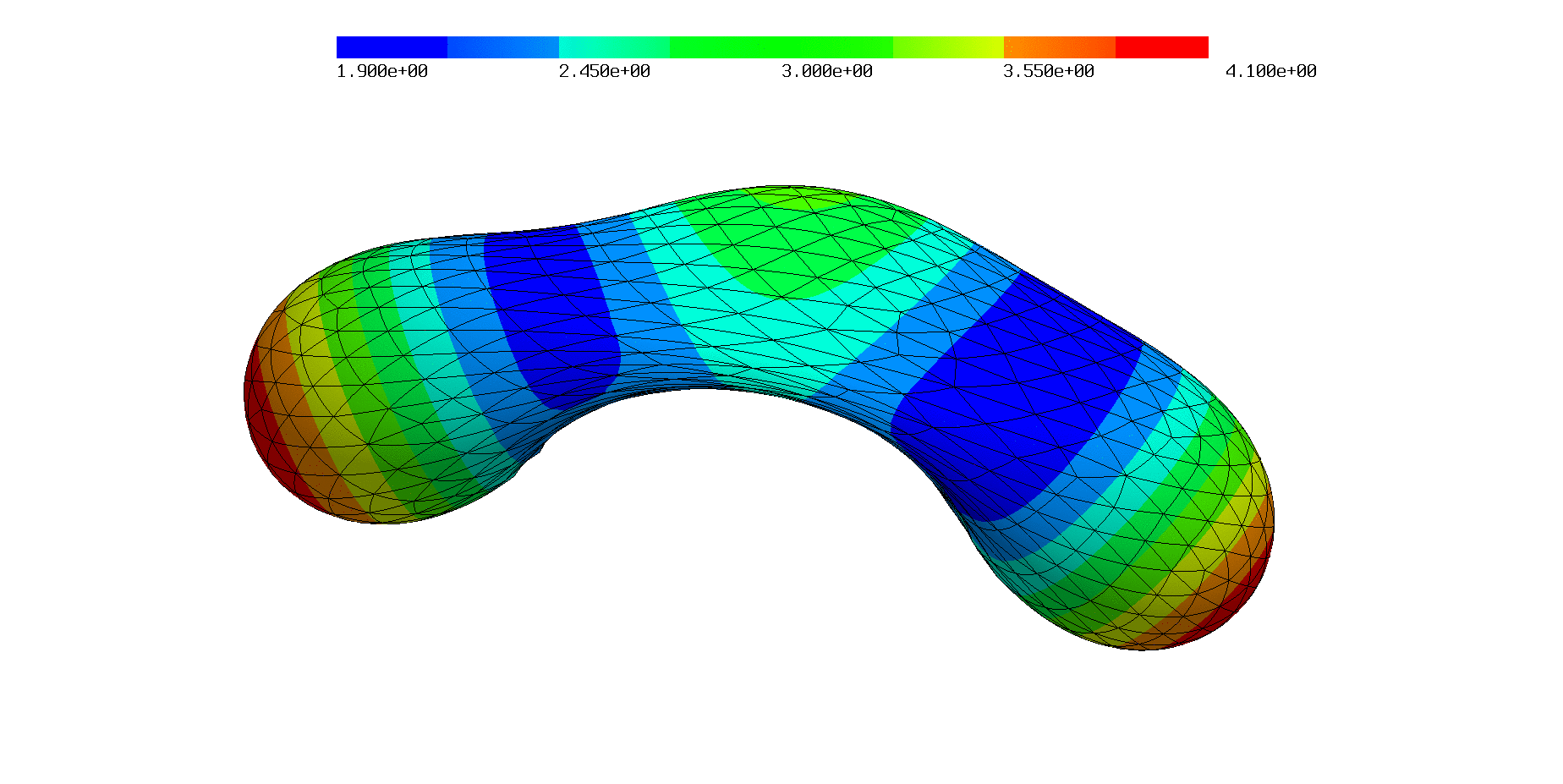}&
		\includegraphics[width=0.33\textwidth]{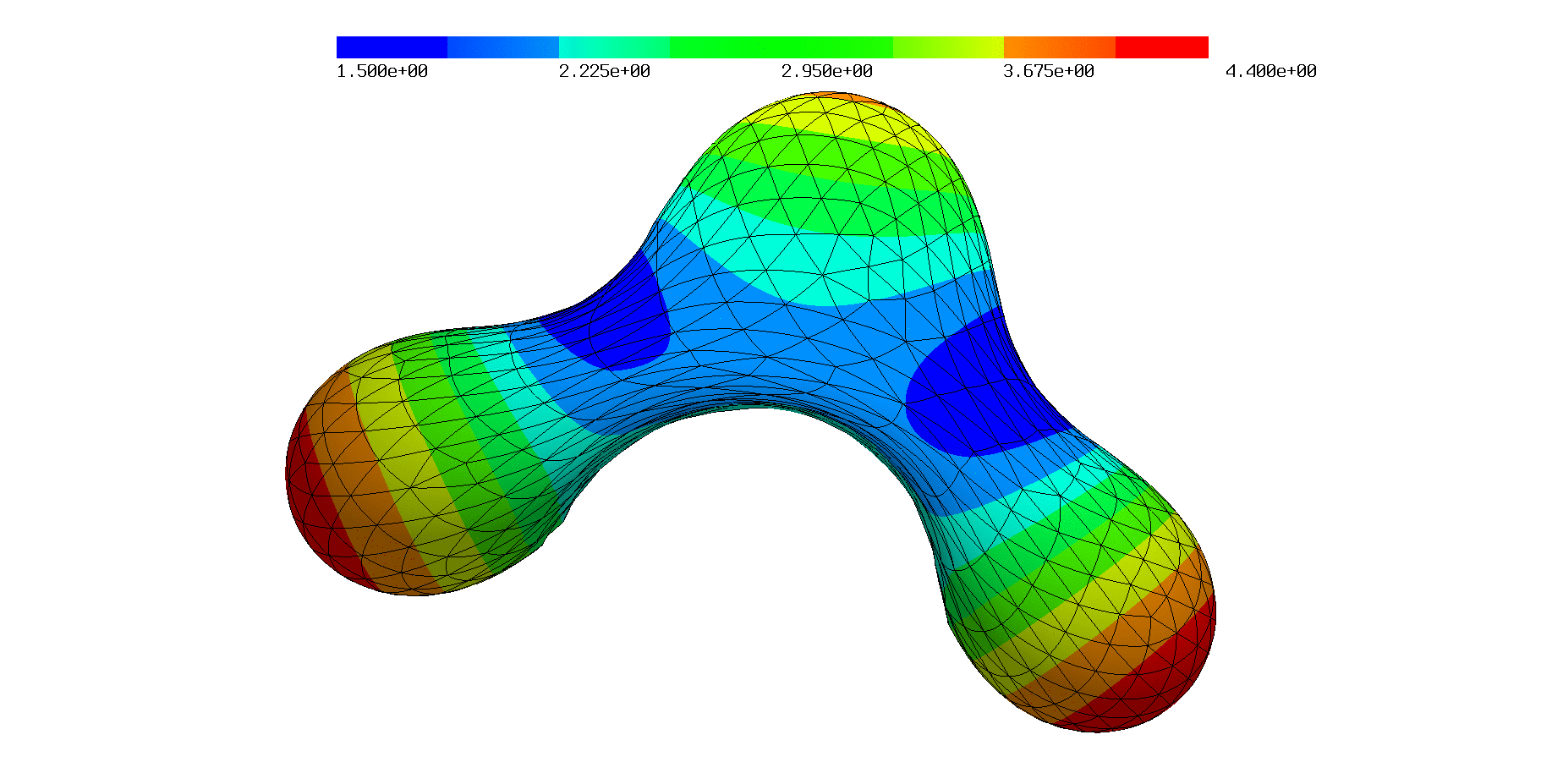}\\
		$\bar{\nu}=0.659$ & $\bar{\nu}=0.627$ & $\bar{\nu}=0.594$\\
		\includegraphics[width=0.33\textwidth]{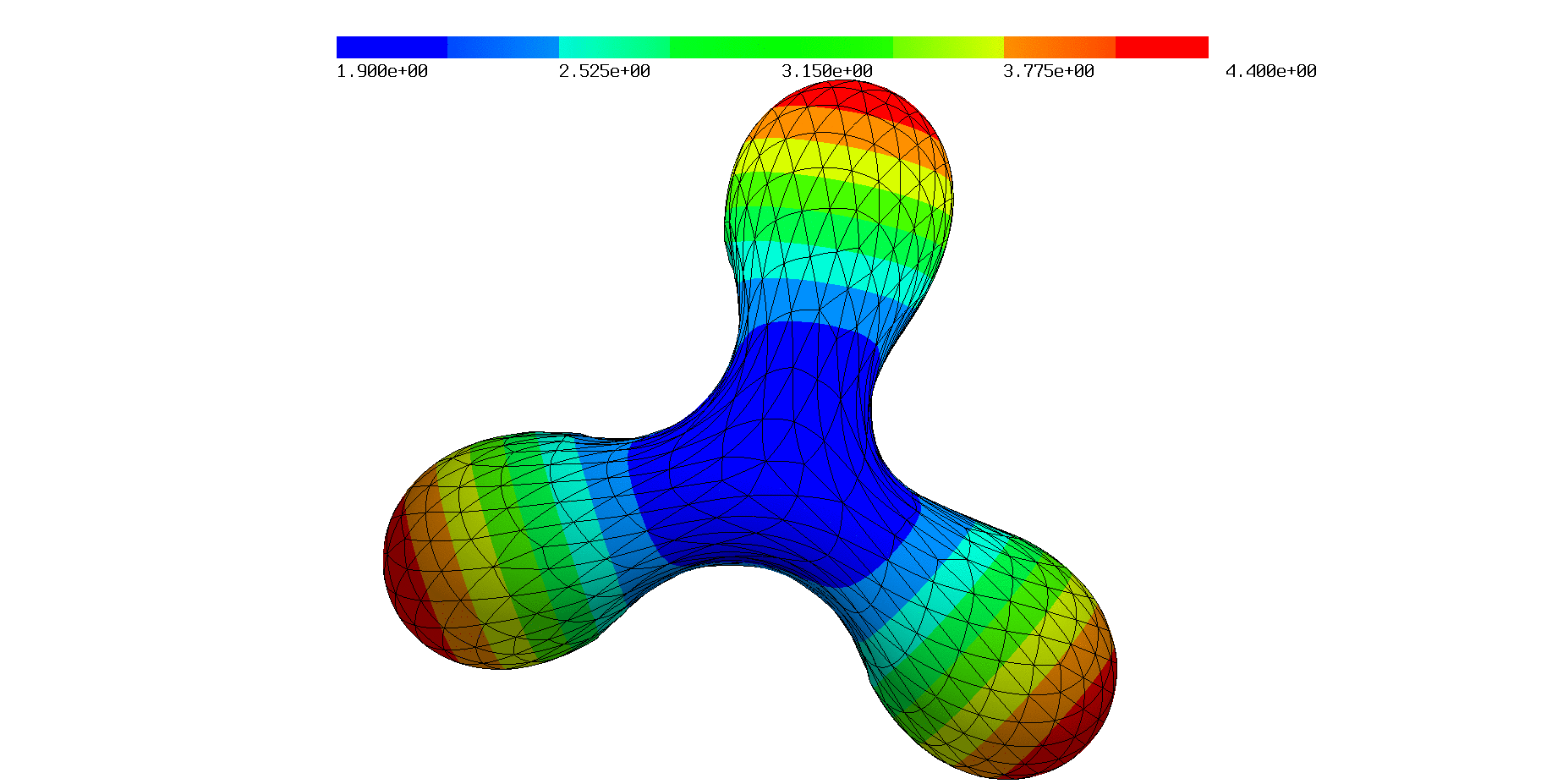}&
		\includegraphics[width=0.33\textwidth]{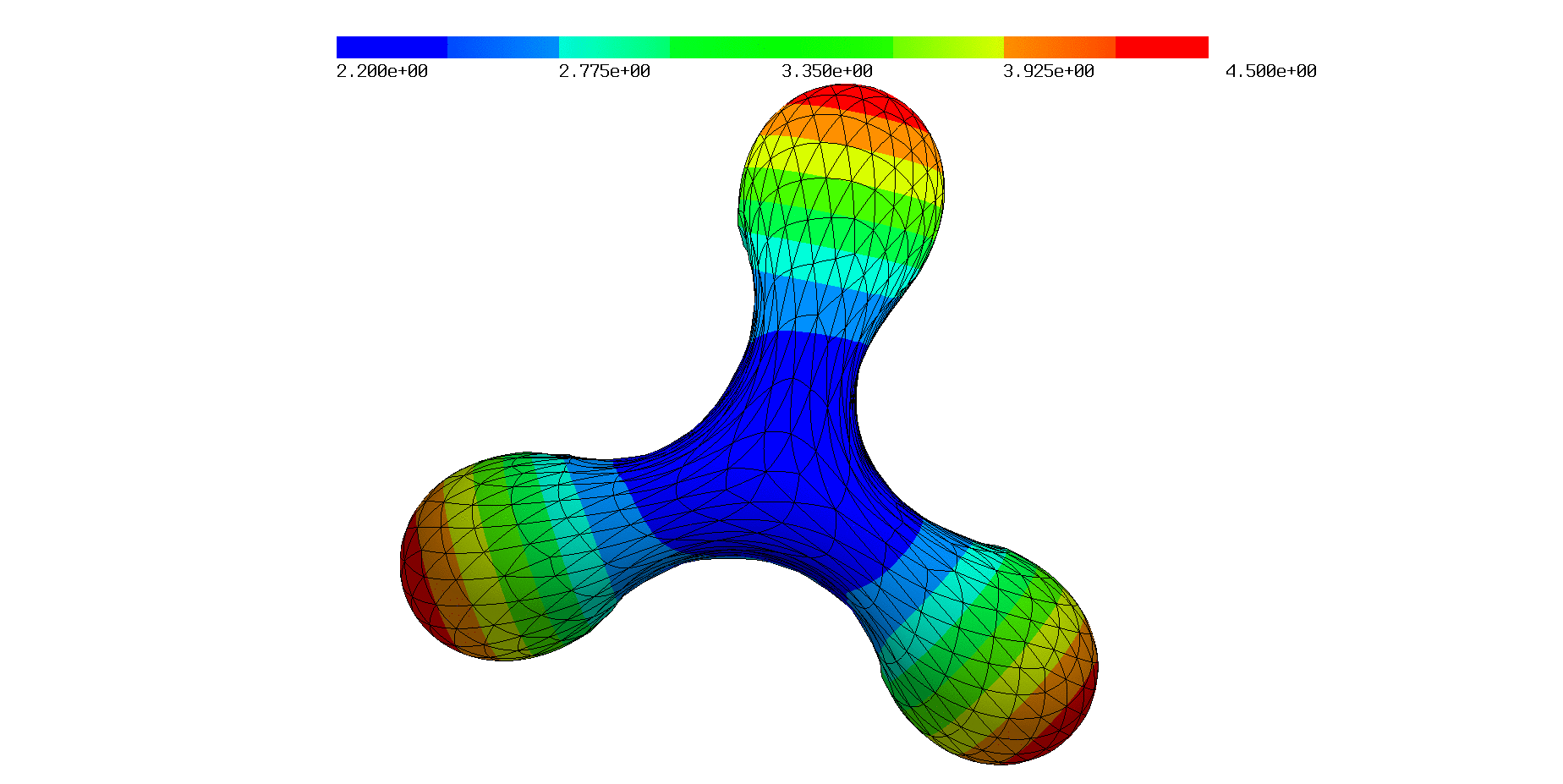}&
		\includegraphics[width=0.33\textwidth]{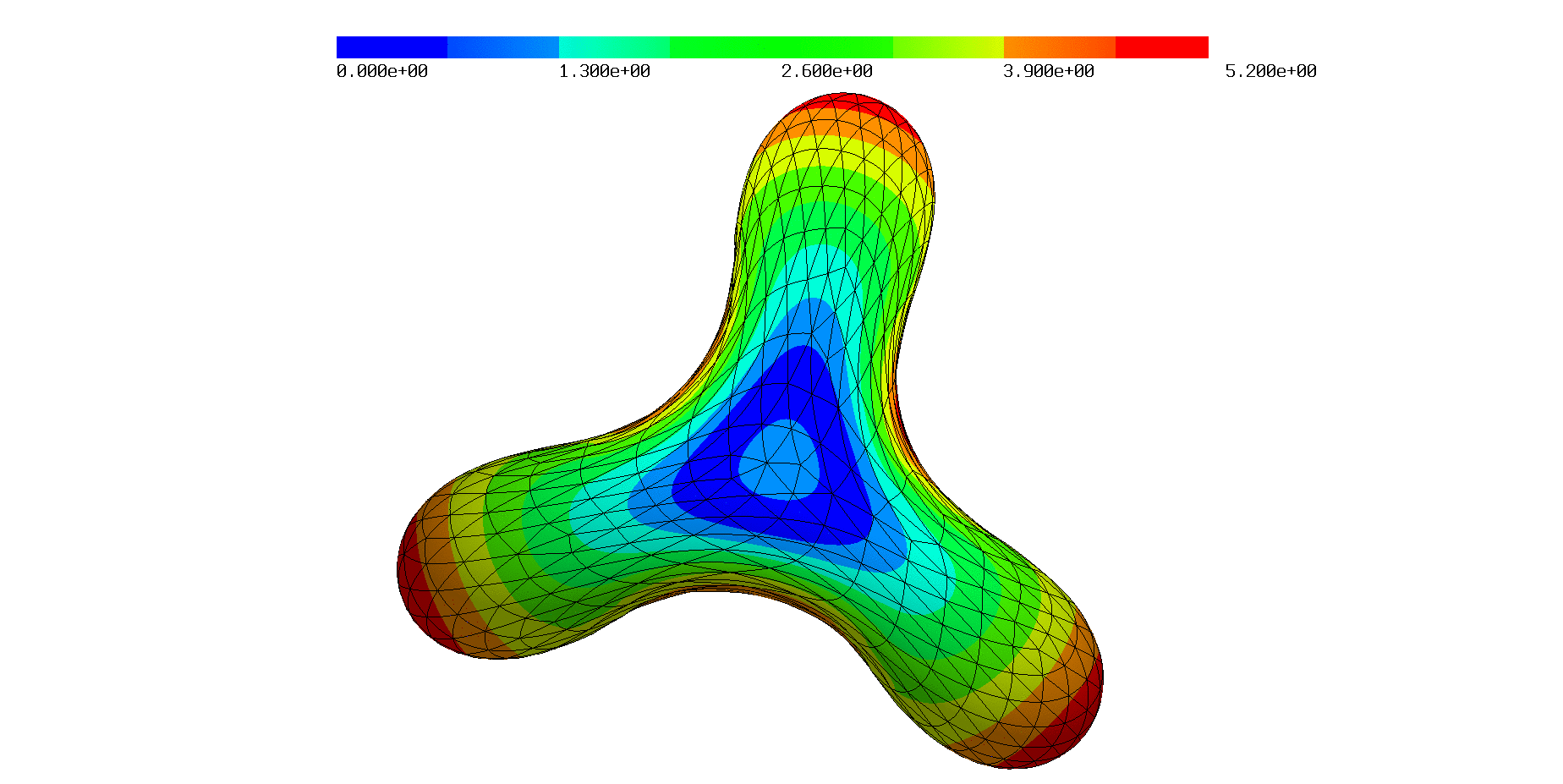}\\
		$\bar{\nu}=0.55$ & $\bar{\nu}=0.508$ & $\bar{\nu}=0.506$\\
	\end{tabular}
	\caption{Oblate shapes for different reduced volumes $\bar{\nu}$ with $H_0=1.2$ and polynomial order $k=2$.}
	\label{fig:oblate_shapes_12}
\end{figure}

The results shown in Figure~\ref{fig:result_spont_p2} (right) for $H_0=1.5$ also render the phase diagram, however, this configuration is more challenging than the previous one. On the one hand there are two bifurcation points at around $\bar{\nu}=0.7$ and $\bar{\nu}=0.58$ and on the other hand the meshes get even more deformed and especially narrowed at the middle. We conclude that remeshing techniques are essentially needed to properly resolve and converge at the bifurcation points. Despite this fact, we observe good agreement with the phase diagram from \cite{SBL91} and obtain the corresponding characteristic solutions comparable to the results in \cite{a_BILIKO_2020a}, see Figure~\ref{fig:prolate_shapes_15}. These are again all axisymmetric, whereas the oblate initial shapes again converge to dumbbell solutions with three ends as for $H_0=1.2$, compare Figure~\ref{fig:oblate_shapes_15}. As before the oblate results for $\bar{\nu}<0.65$ did not fully converge, however, definitely indicating that a non axisymmetric oblate branch exists close to them.

\begin{figure}[h]
	\begin{tabular}{ccc}
		\includegraphics[width=0.33\textwidth]{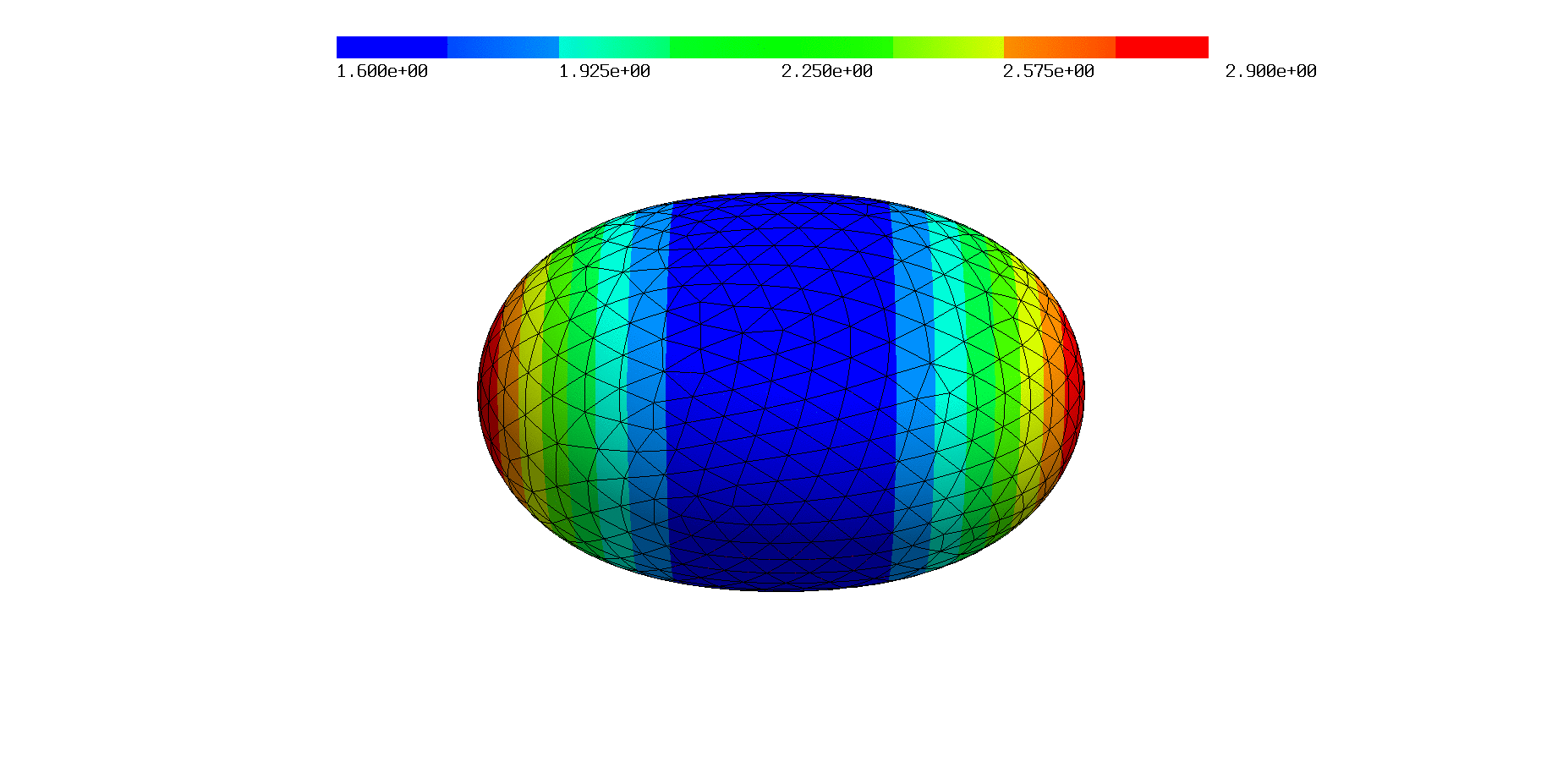}&
		\includegraphics[width=0.33\textwidth]{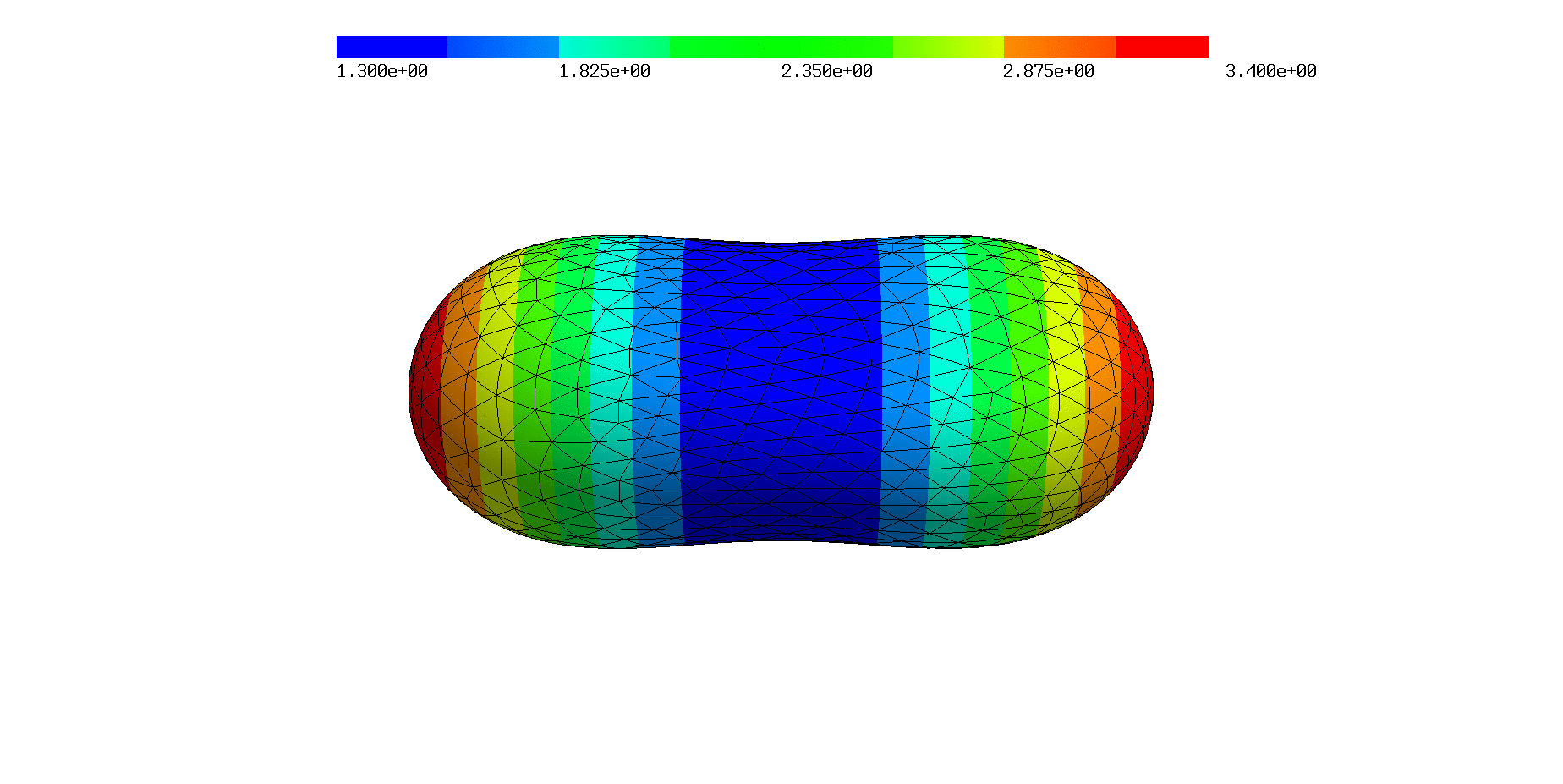}&
		\includegraphics[width=0.33\textwidth]{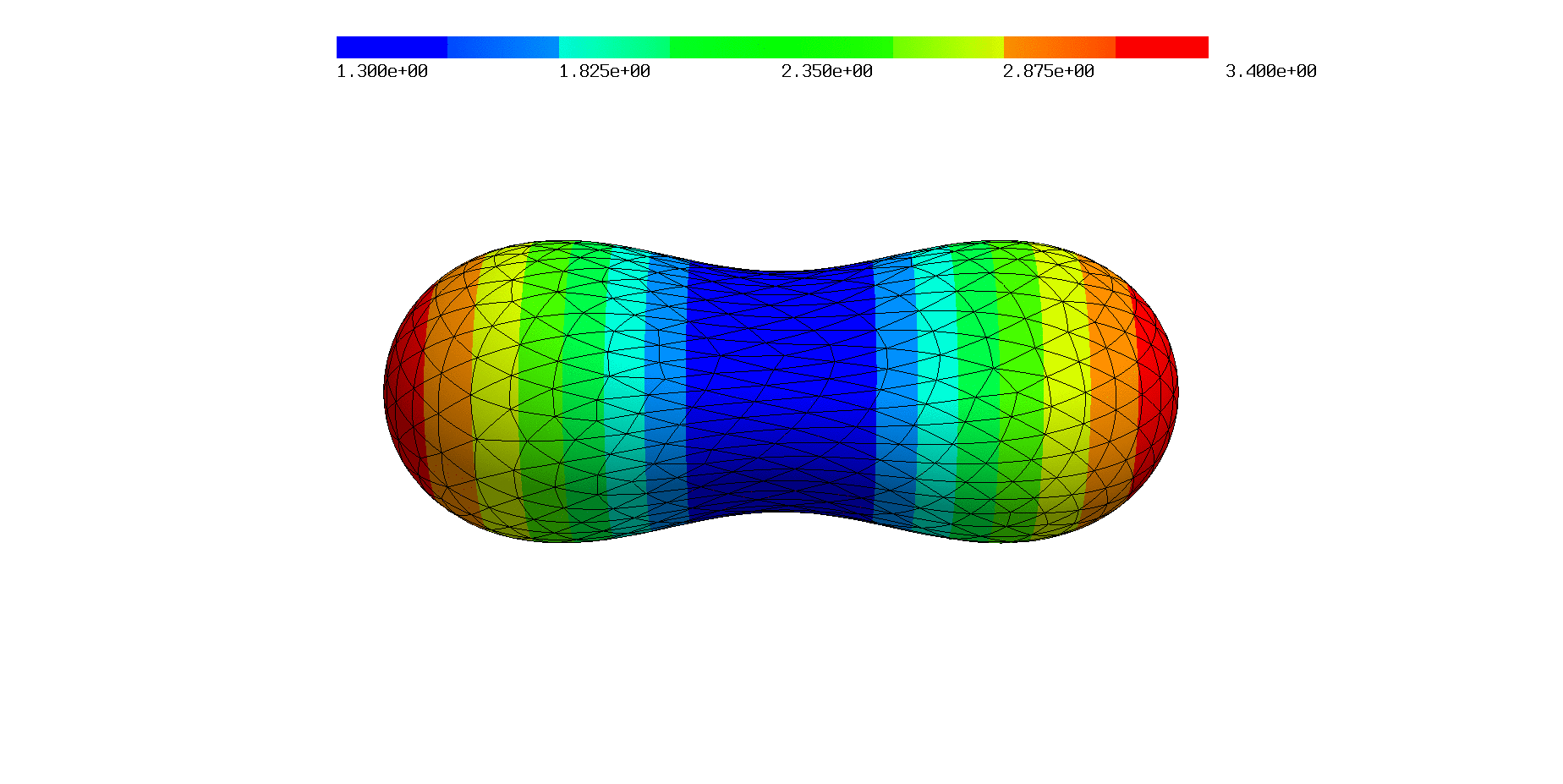}\\
		$\bar{\nu}=0.956$ & $\bar{\nu}=0.835$ & $\bar{\nu}=0.772$\\
		\includegraphics[width=0.33\textwidth]{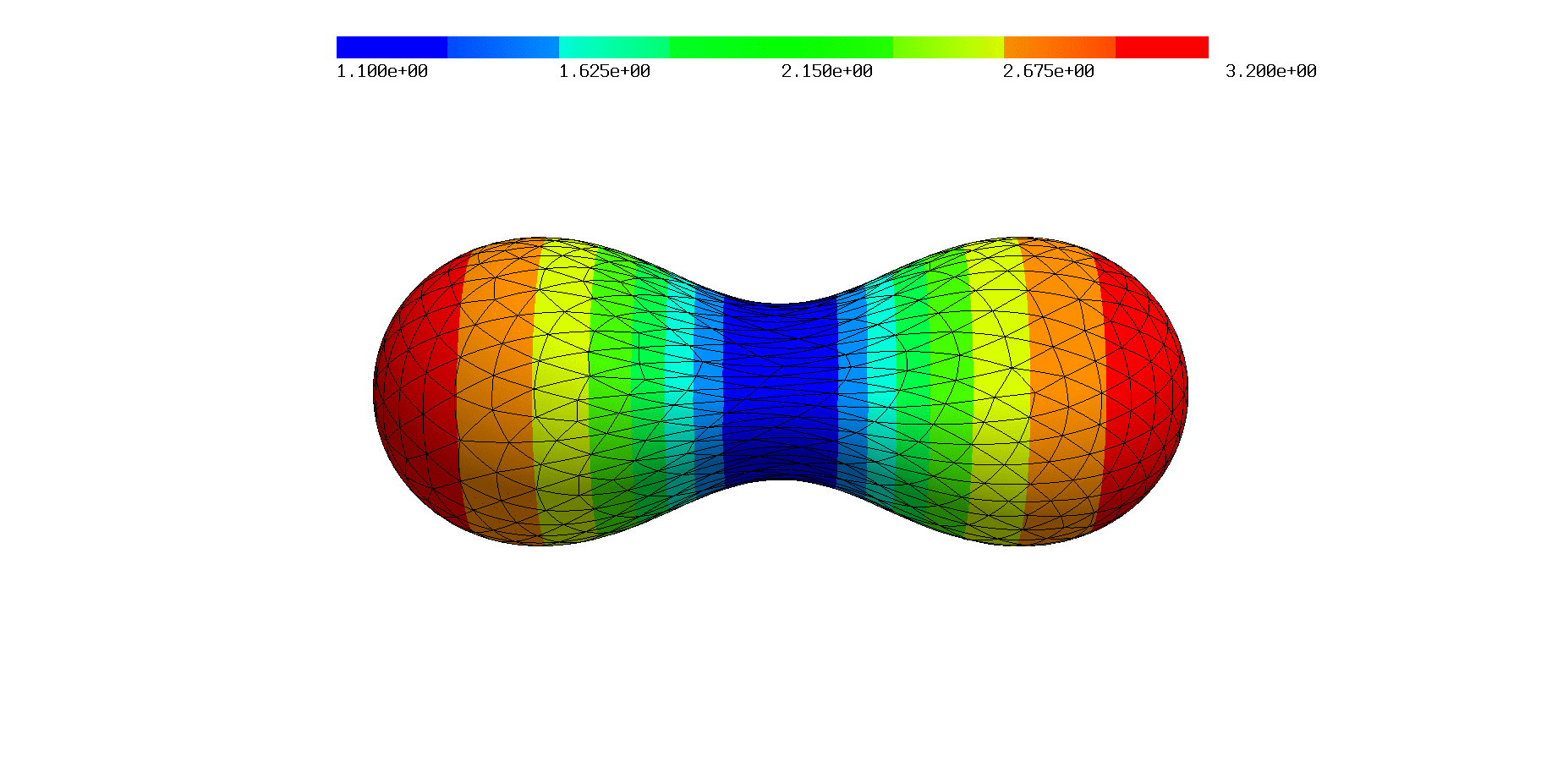}&
		\includegraphics[width=0.33\textwidth]{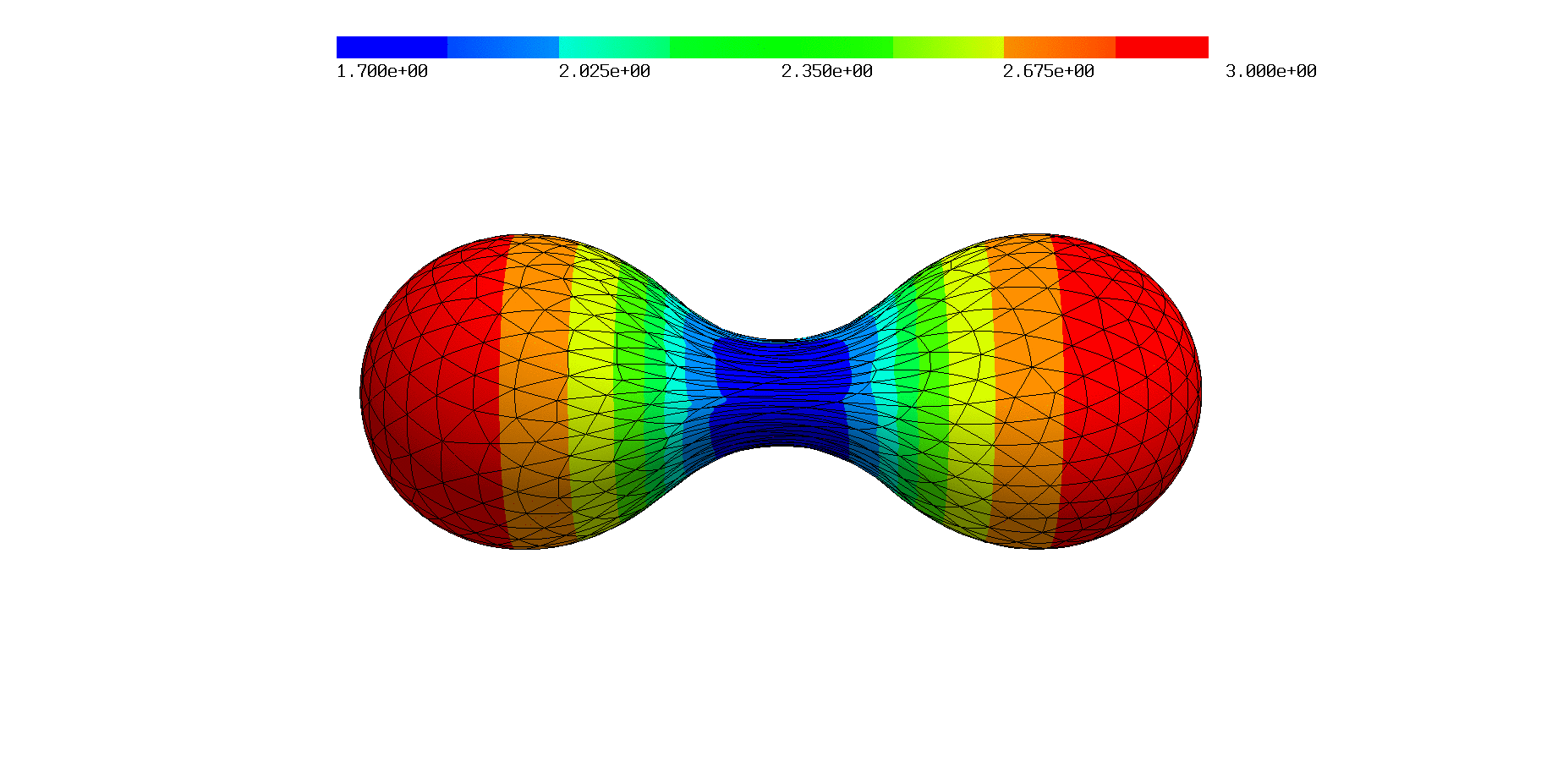}&
		\includegraphics[width=0.33\textwidth]{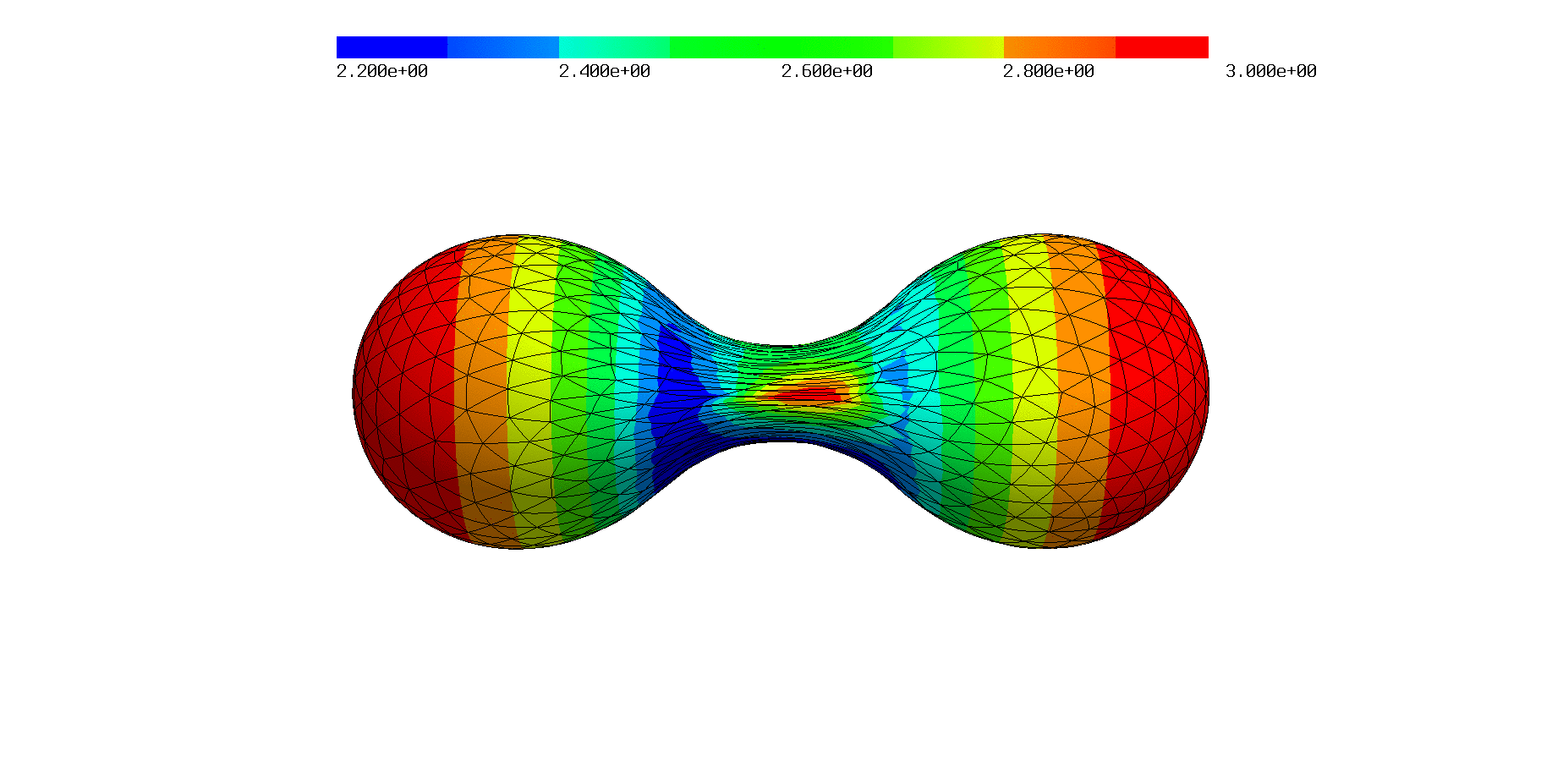}\\
		$\bar{\nu}=0.726$ & $\bar{\nu}=0.687$ & $\bar{\nu}=0.679$\\
		\includegraphics[width=0.33\textwidth]{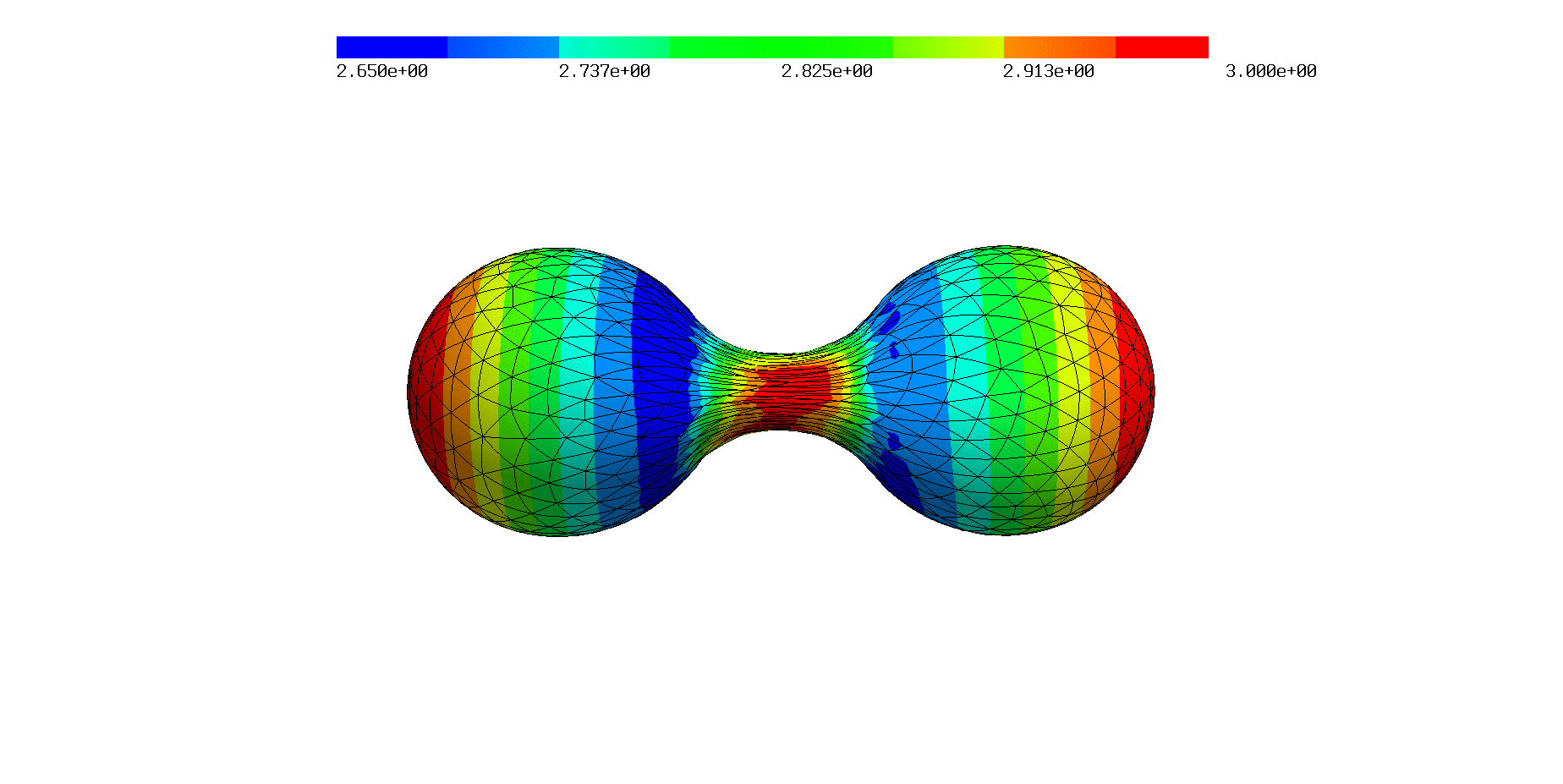}&
		\includegraphics[width=0.33\textwidth]{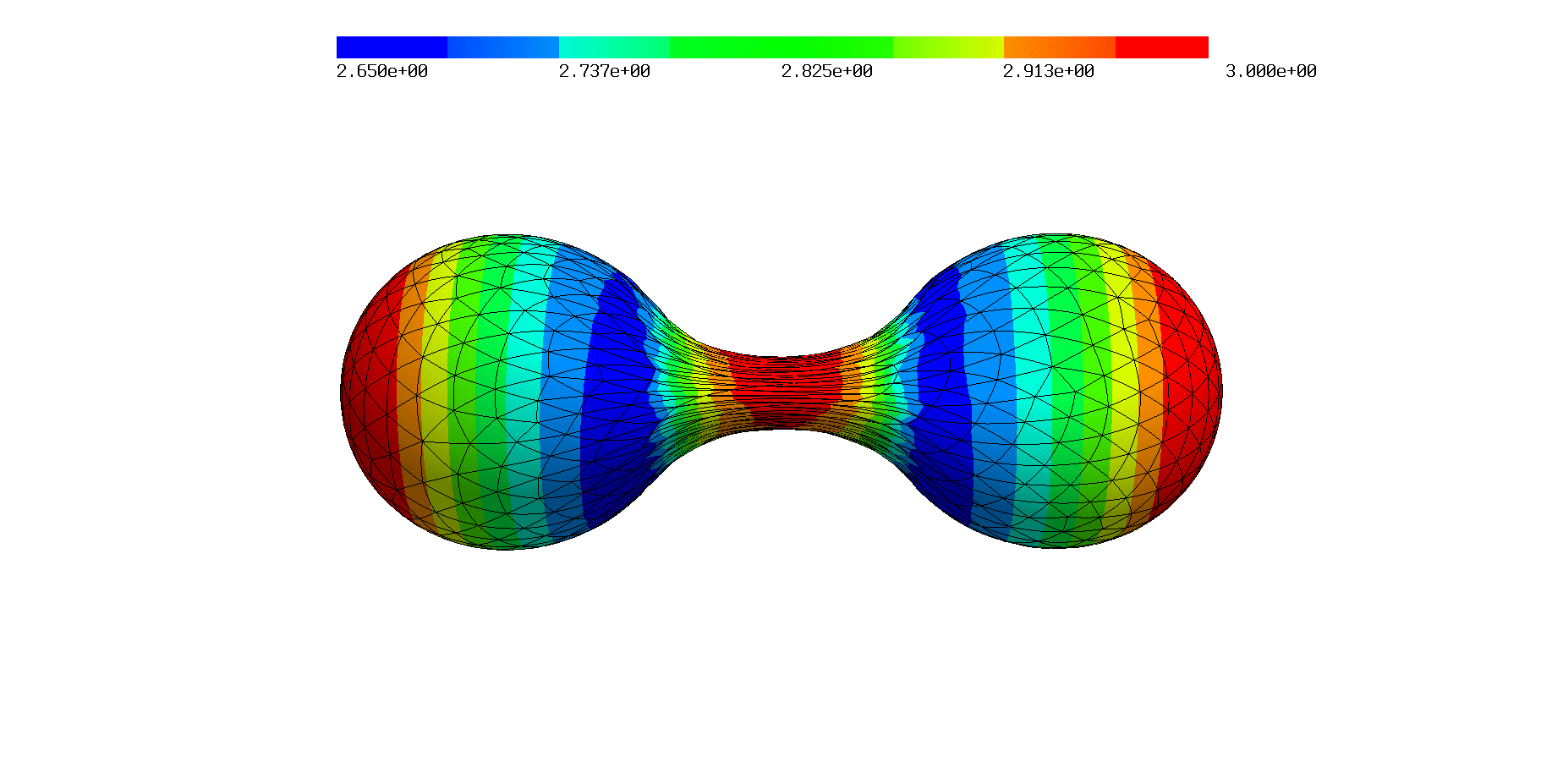}&
		\includegraphics[width=0.33\textwidth]{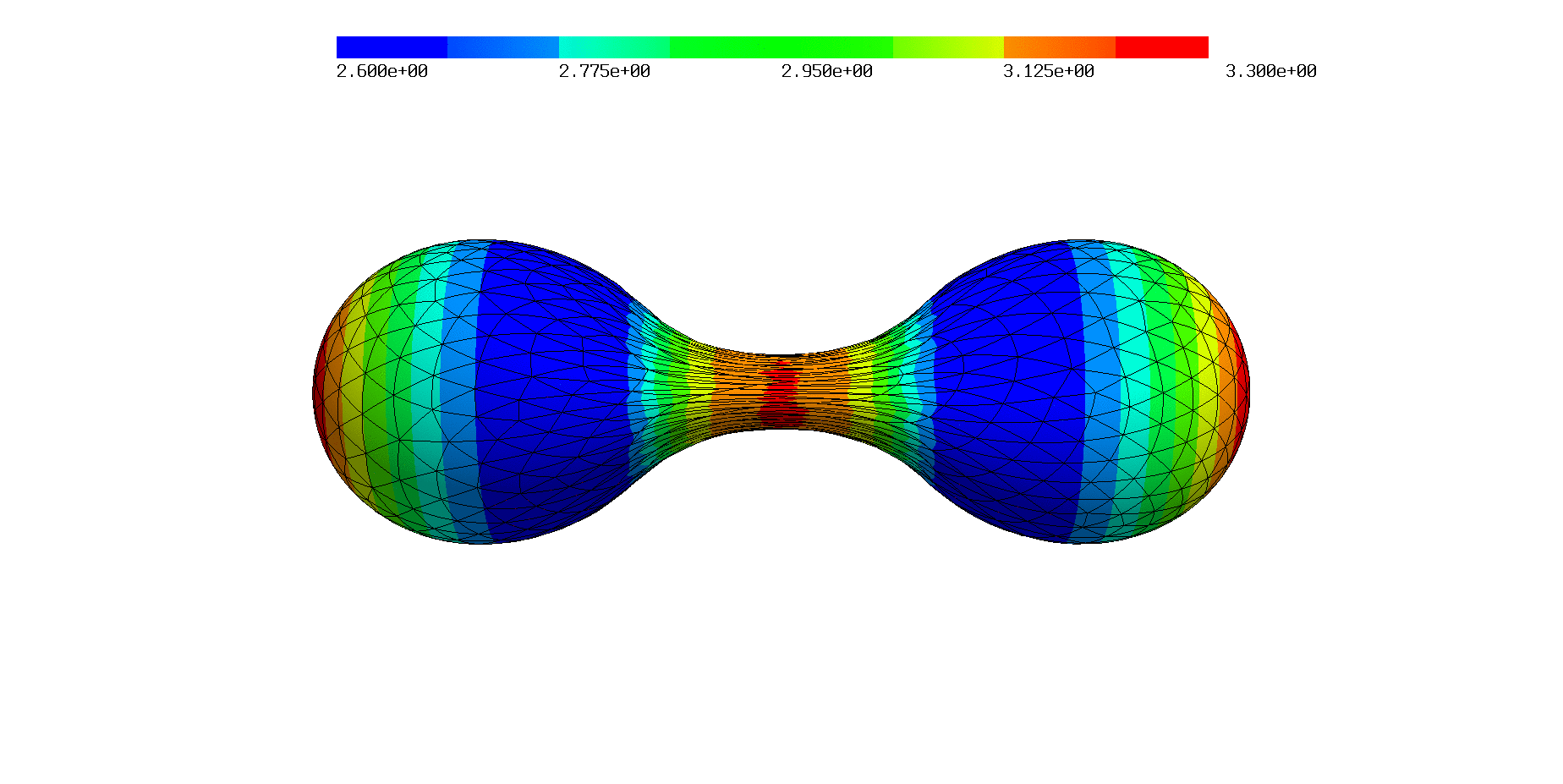}\\
		$\bar{\nu}=0.679$ & $\bar{\nu}=0.665$ & $\bar{\nu}=0.656$\\
		\includegraphics[width=0.33\textwidth]{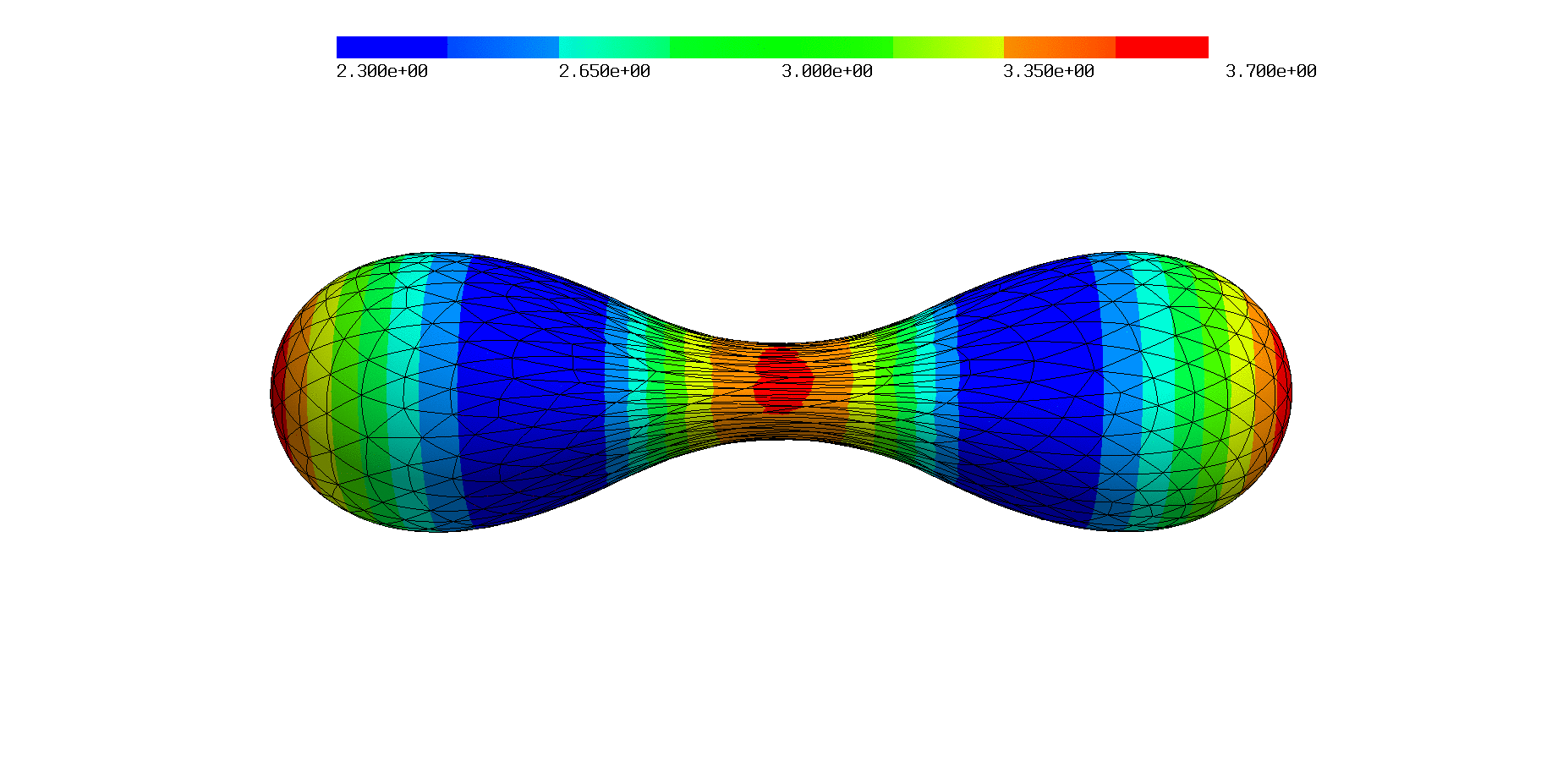}&
		\includegraphics[width=0.33\textwidth]{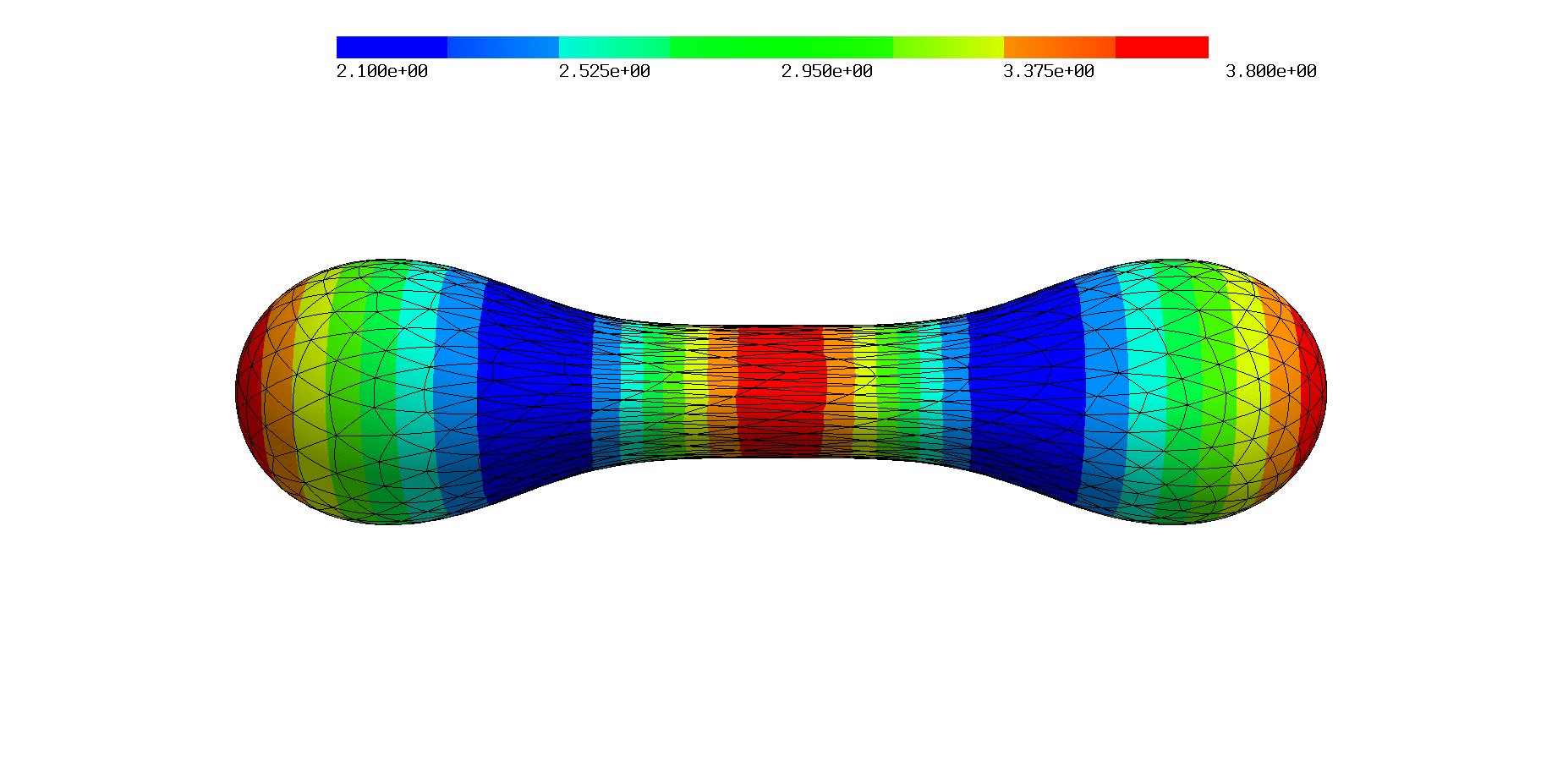}&
		\includegraphics[width=0.33\textwidth]{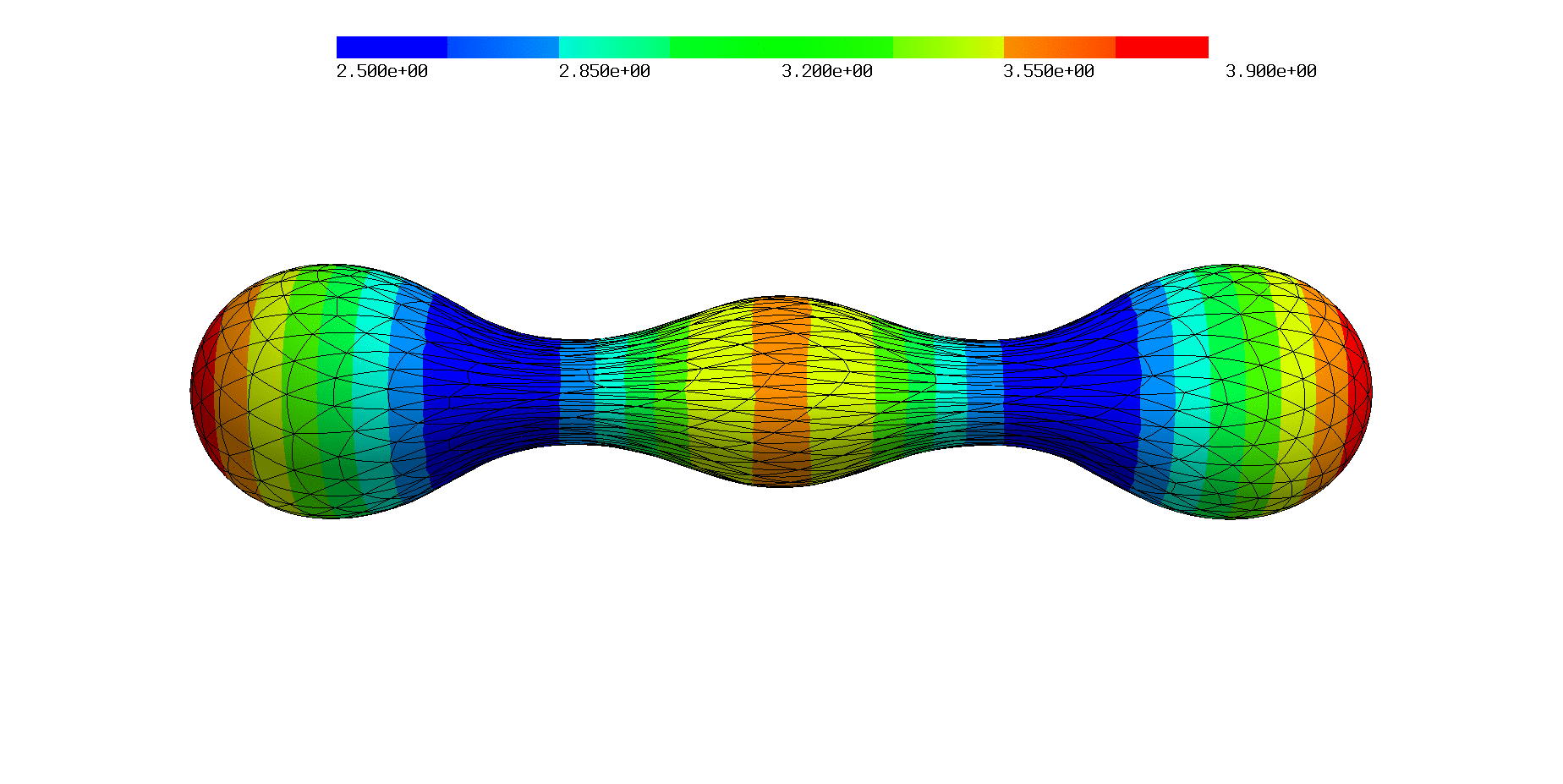}\\
		$\bar{\nu}=0.634$ & $\bar{\nu}=0.596$ & $\bar{\nu}=0.543$\\
		\includegraphics[width=0.33\textwidth]{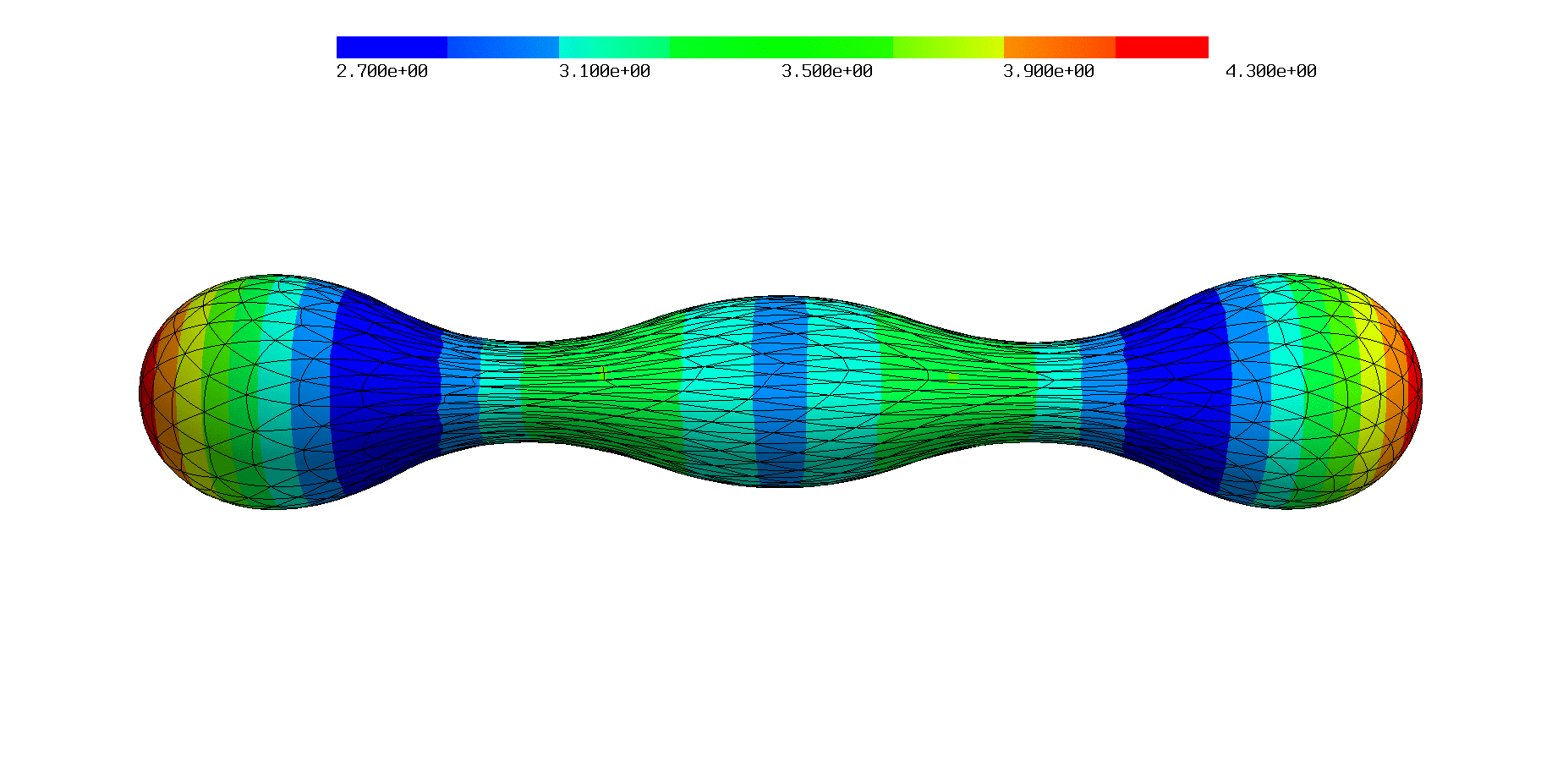}&
		\includegraphics[width=0.33\textwidth]{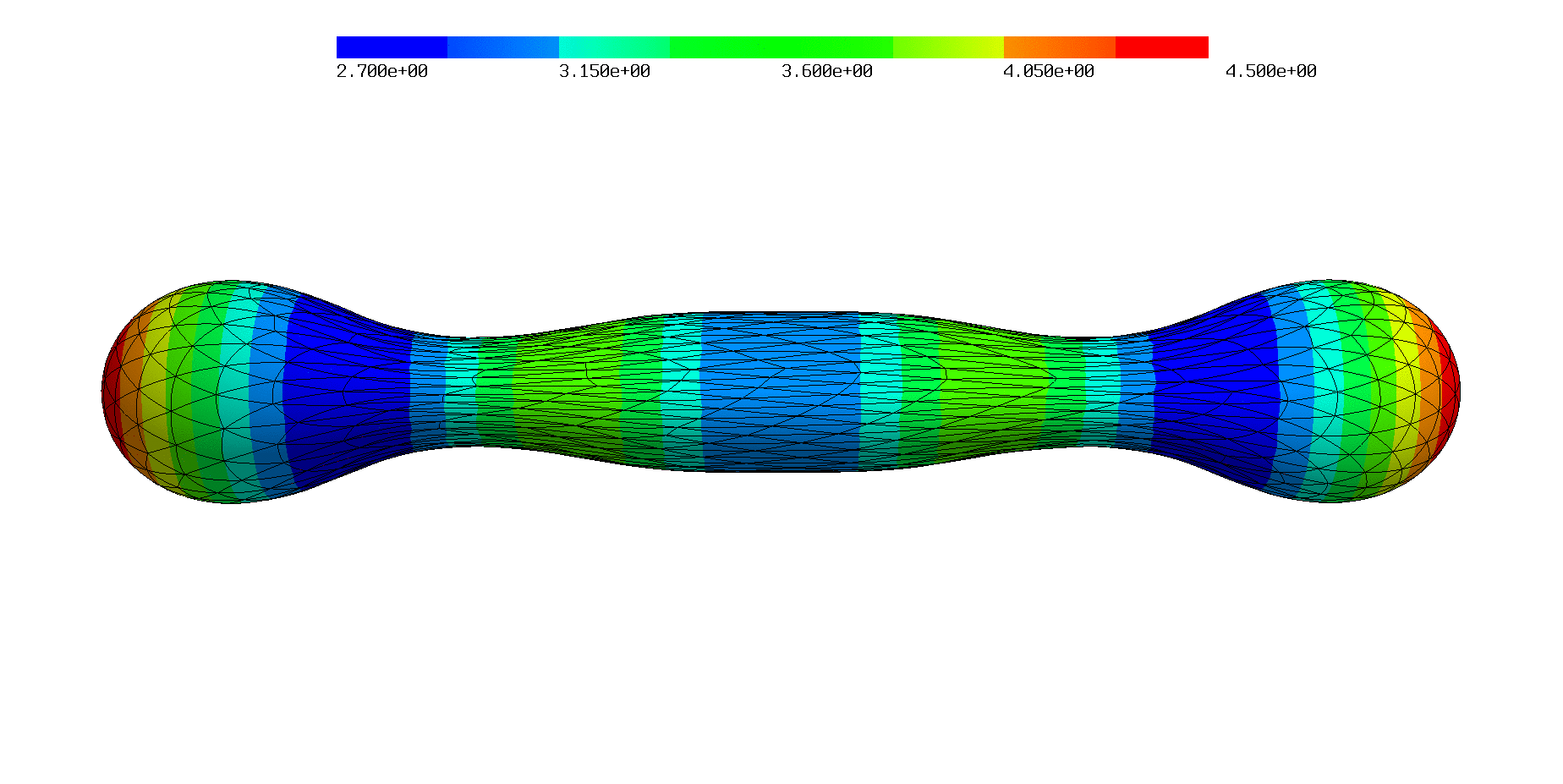}&
		\includegraphics[width=0.33\textwidth]{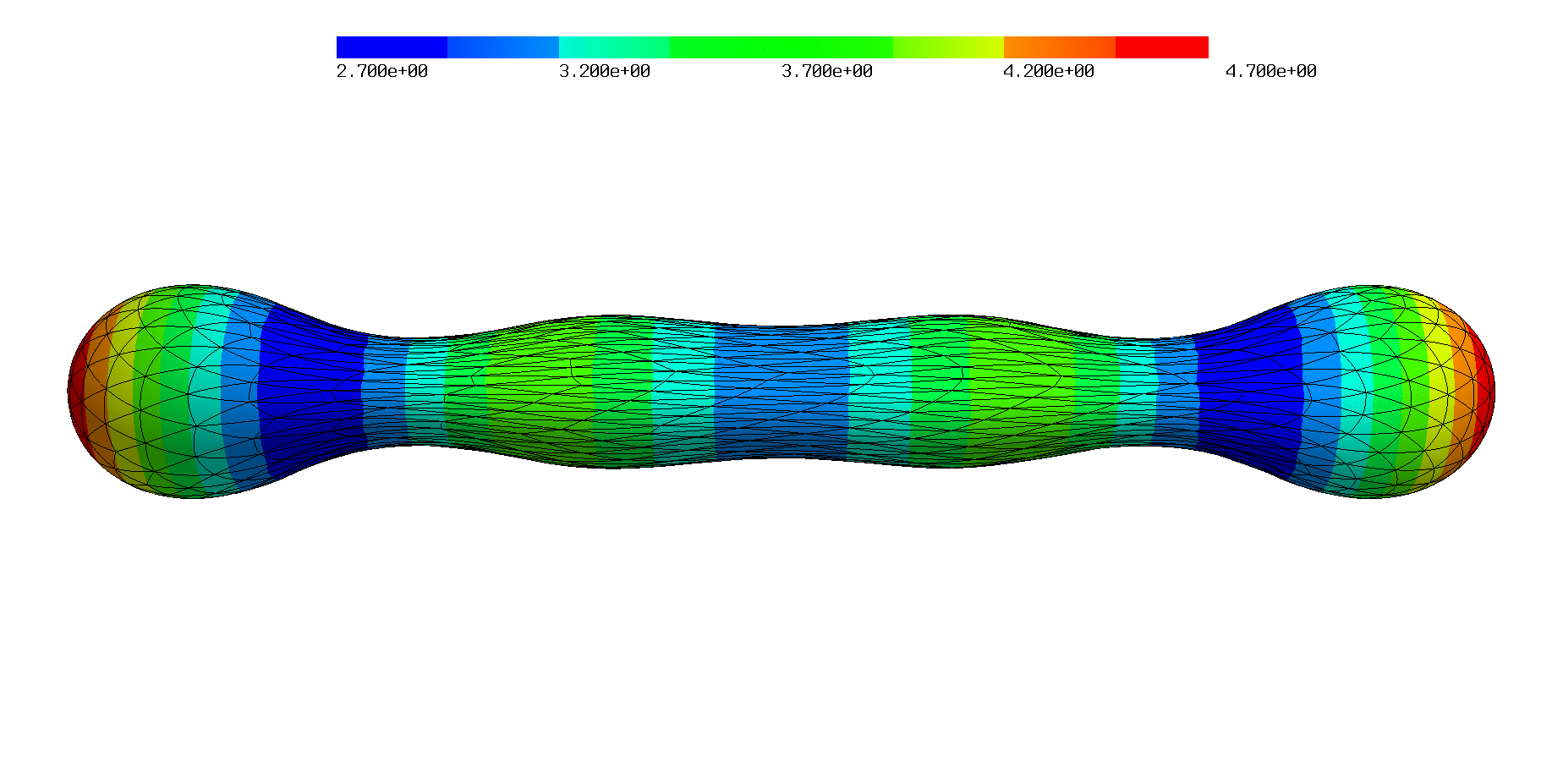}\\
		$\bar{\nu}=0.514$ & $\bar{\nu}=0.491$ & $\bar{\nu}=0.467$\\
	\end{tabular}
	\caption{Prolate shapes for different reduced volumes $\bar{\nu}$ with $H_0=1.5$ and polynomial order $k=2$.}
	\label{fig:prolate_shapes_15}
\end{figure}

\begin{figure}[h]
	\begin{tabular}{ccc}
		\includegraphics[width=0.32\textwidth]{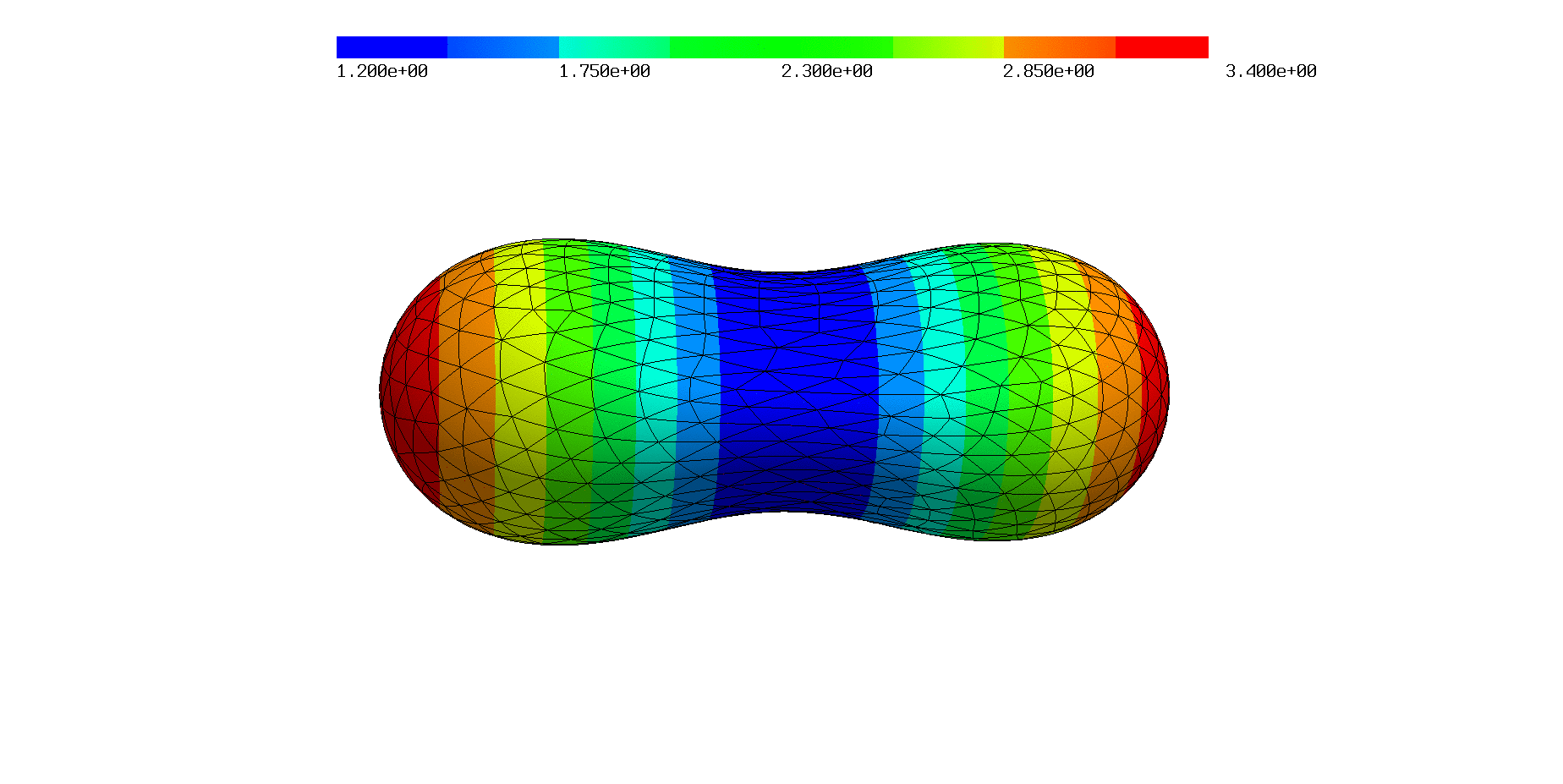}&
		\includegraphics[width=0.32\textwidth]{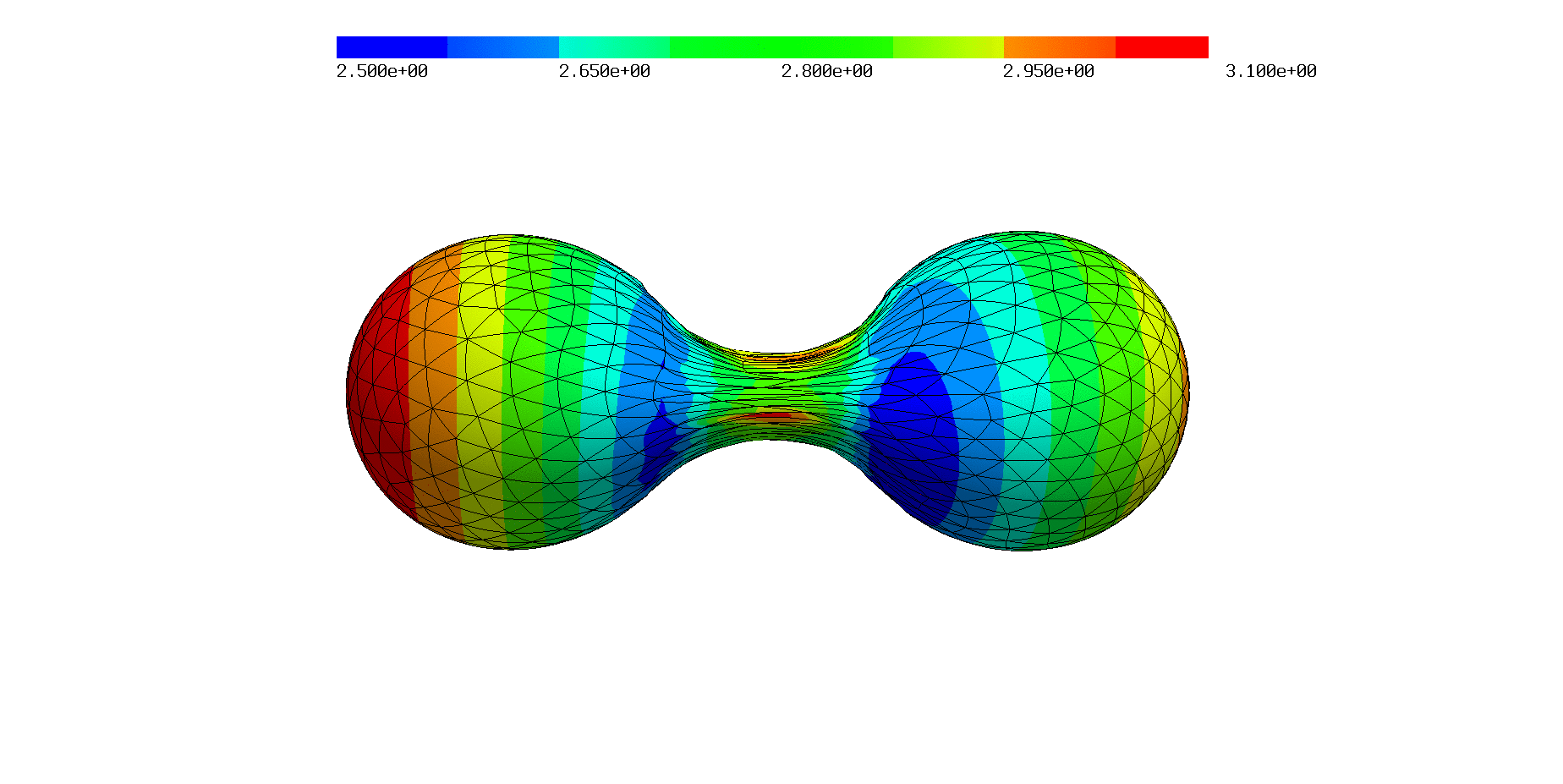}&
		\includegraphics[width=0.32\textwidth]{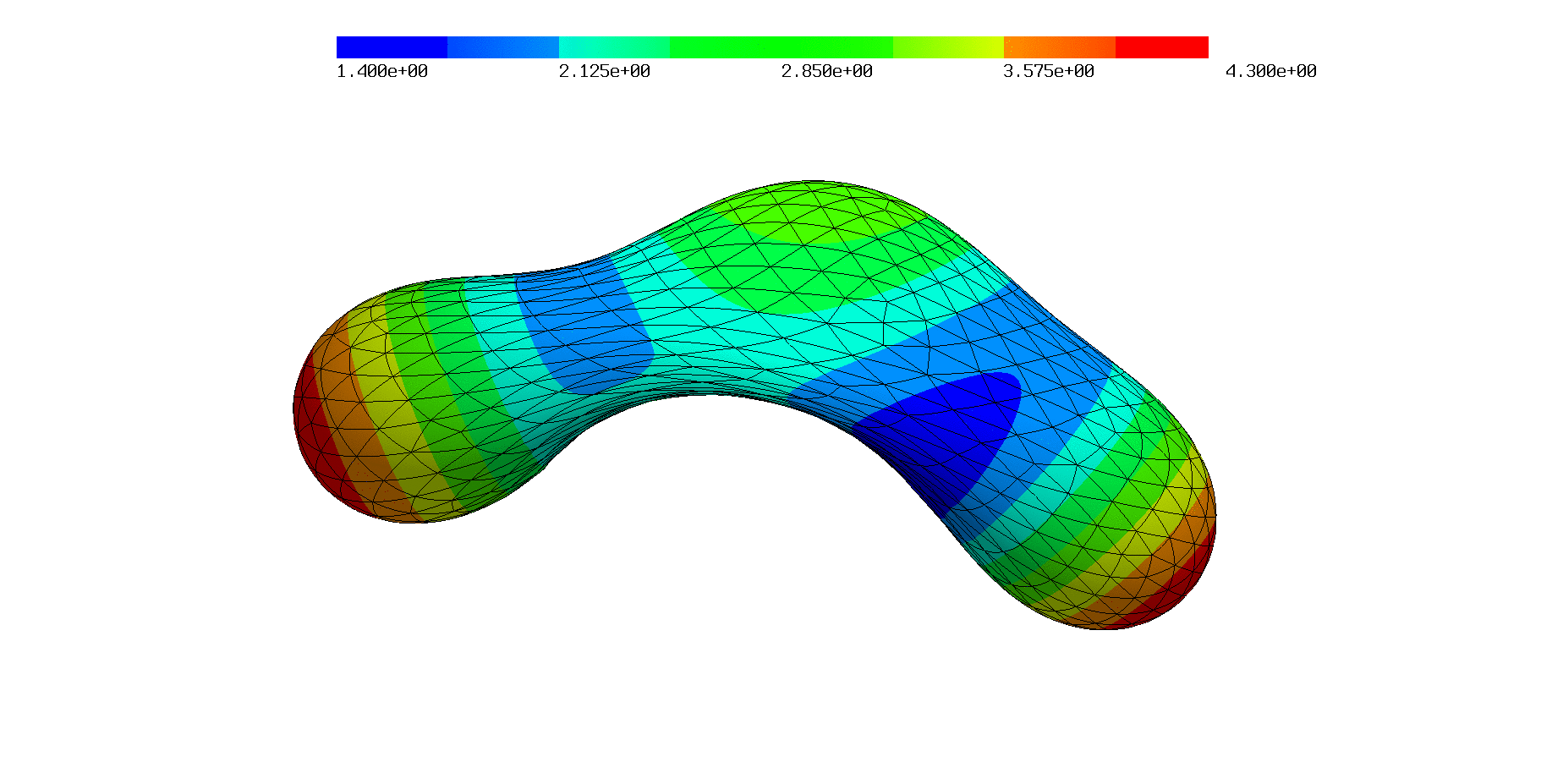}\\
		$\bar{\nu}=0.771$ & $\bar{\nu}=0.674$ & $\bar{\nu}=0.628$\\
		\includegraphics[width=0.32\textwidth]{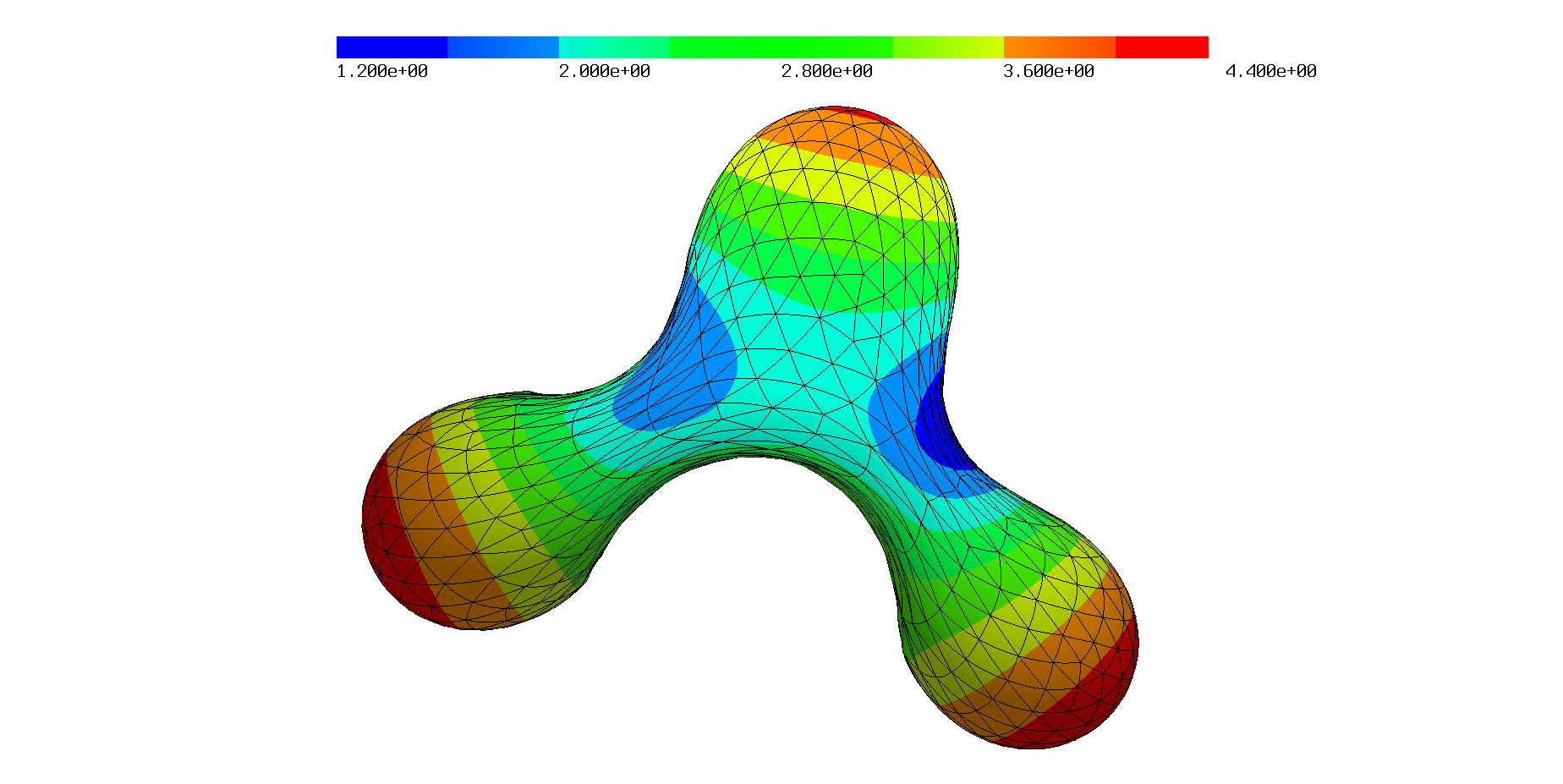}&
		\includegraphics[width=0.32\textwidth]{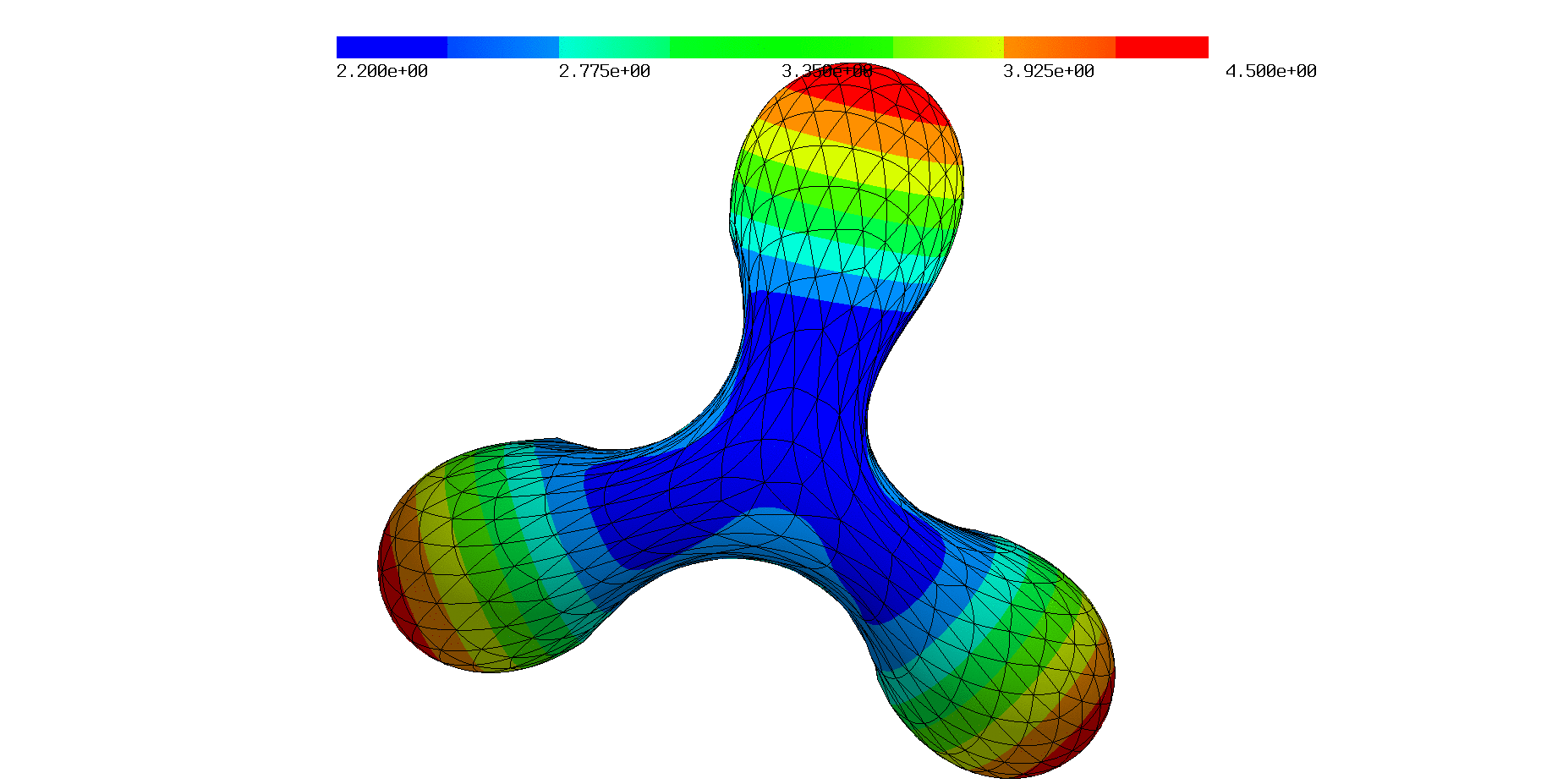}&
		\includegraphics[width=0.32\textwidth]{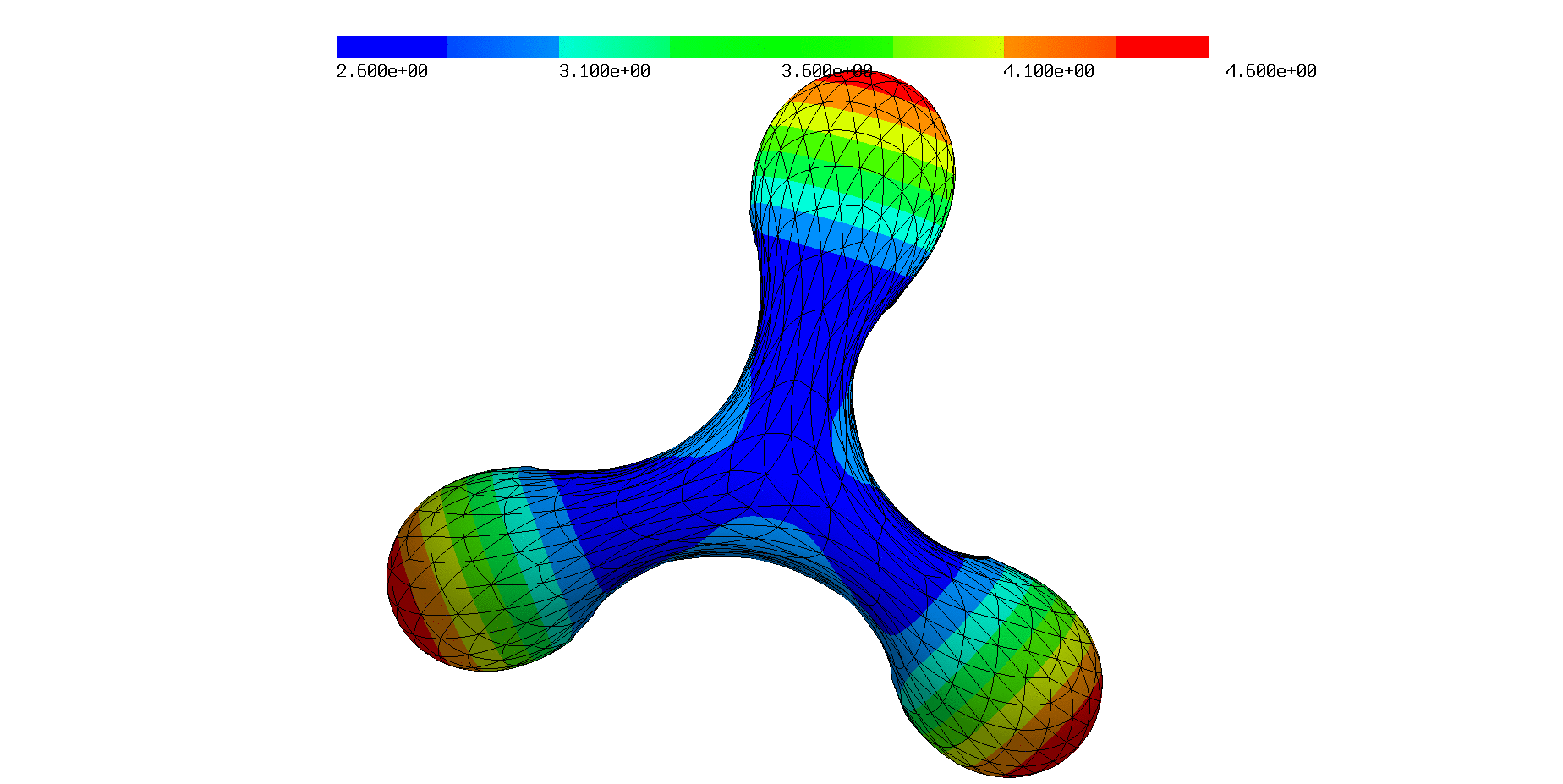}\\
		$\bar{\nu}=0.574$ & $\bar{\nu}=0.533$ & $\bar{\nu}=0.495$
	\end{tabular}
	\caption{Oblate shapes for different reduced volumes $\bar{\nu}$ with $H_0=1.5$ and polynomial order $k=2$.}
	\label{fig:oblate_shapes_15}
\end{figure}

\section*{Conclusion and future work}
In this paper we presented a novel shape optimization method for minimizing the Canham-Helfrich-Evans energy under area and volume constraints based on a lifting of the distributional discrete shape operator. This three-field approach allows for a general formula for the shape derivative independently of the used polynomial order of approximation. A shape gradient optimization procedure has been presented in NGSolve supporting automatic shape differentiation. The performance of the proposed method has been demonstrated on several benchmark examples including curvature computation and spontaneous curvature.

Due to the large deformations of the shapes in specific benchmark configurations, the mesh quality may become poor yielding worse convergence rates or even a termination of the algorithm. Therefore, re-meshing techniques (for arbitrary order of curved geometry) are topic of further research to push forward to more challenging benchmark examples.\newline

\section*{Acknowledgments}
Michael Neunteufel acknowledges support from AC2T research GmbH by project 1931712. Joachim Sch\"oberl acknowledges support by the Austrian Science Fund (FWF)
project F\,65. Kevin Sturm acknowledges support from the Austrian Science Fund (FWF) project P 32911. The authors are also indebted to Ulisse Stefanelli for interesting discussions.

\appendix

\section{Proof of Lemma \ref{lem:shape_der_wg_jump}}
\label{app:comp_shape_der}

For the angle $\sphericalangle(a,b):=\arccos(a\cdot b)$, its derivative is given by
\begin{align*}
\frac{d}{dt}\sphericalangle(a(t),b(t))|_{t=0} = -\frac{1}{\sqrt{1- (a(0)\cdot b(0))^2}}\frac{d}{dt}(a(t)\cdot b(t))|_{t=0}.
\end{align*} 
The averaged normal vector $\Av{\nubf}$ does not depend on the deformation. However, the projection $\bP_{\taubf^t}^\perp(\cdot)$ \eqref{eq:proj_edge_ortho} does. Noting that by construction $\Av{\nubf}\cdot \taubf = 0$ and thus $\bP_{\taubf}^{\perp}(\Av{\nubf})=\Av{\nubf}$ there holds
\begin{align*}
\frac{d}{dt}\sphericalangle(\mubf^t,\bP_{\taubf^t}^{\perp}(\Av{\nubf}))|_{t=0}= -\frac{\frac{d}{dt}(\mubf^t\cdot\bP_{\taubf^t}^{\perp}(\Av{\nubf}))|_{t=0}}{\sqrt{1-(\mubf\cdot\Av{\nubf})^2}}
\end{align*}
and further with the notation $\langle a,b\rangle := a\cdot b$
\begin{align*}
\frac{d}{dt}\bP_{\taubf^t}^{\perp}(\Av{\nubf})|_{t=0} &= \frac{-\frac{d}{dt}\Av{\nubf}\cdot \taubf^t \taubf^t|_{t=0}}{\|\Av{\nubf}-\underbrace{\Av{\nubf}\cdot \taubf \taubf}_{=0}\|}+\frac{\Av{\nubf}}{\|\Av{\nubf}-\underbrace{\Av{\nubf}\cdot \taubf \taubf}_{=0}\|^3}\langle \Av{\nubf}, \frac{d}{dt}\Av{\nubf}\cdot \taubf^t\taubf^t|_{t=0}\rangle\nonumber\\
&=-\frac{d}{dt}\Av{\nubf}\cdot \taubf^t \taubf^t|_{t=0}+\Av{\nubf}\langle \Av{\nubf}, \frac{d}{dt}\Av{\nubf}\cdot \taubf^t\taubf^t|_{t=0}\rangle.
\end{align*}
With
\begin{align*}
\frac{d}{dt}(\Av{\nubf}\cdot \taubf^t \taubf^t)|_{t=0} &= \Av{\nubf}\cdot\taubf(\partial \bX\taubf -(\partial \bX\taubf\cdot\taubf)\taubf)+\Av{\nubf}\cdot (\partial \bX\taubf -(\partial \bX\taubf\cdot\taubf)\taubf)\taubf\nonumber\\
&=\Av{\nubf}\cdot\taubf(\partial \bX\taubf) -2(\partial \bX\taubf\cdot\taubf)(\Av{\nubf}\cdot\taubf)\taubf+\Av{\nubf}\cdot (\partial \bX\taubf)\taubf=\Av{\nubf}\cdot (\partial \bX\taubf)\taubf
\end{align*}
we get
\begin{align*}
\frac{d}{dt}\bP_{\taubf^t}^{\perp}(\Av{\nubf})|_{t=0} &= -\Av{\nubf}\cdot (\partial \bX\taubf)\taubf+\Av{\nubf}\langle \Av{\nubf}, \Av{\nubf}\cdot (\partial \bX\taubf)\taubf\rangle=-\Av{\nubf}\cdot (\partial \bX\taubf)\taubf
\end{align*}
and thus with \eqref{eq:shape_der_mu_surface}, $\mubf\cdot\taubf=0$, and $\Av{\bnu}\cdot\taubf=0$
\begin{align*}
\frac{d}{dt}(\mubf^t\cdot\bP_{\taubf^t}^{\perp}(\Av{\nubf}))|_{t=0} &= ((\bI-\taubf\otimes\taubf)\partial\bX-\partial\bX^\top)\mubf\cdot \Av{\nubf}+\mubf\cdot (-\Av{\nubf}\cdot (\partial \bX\taubf)\taubf)\nonumber\\
&=((\bI-\taubf\otimes\taubf)\partial\bX-\partial\bX^\top)\mubf\cdot \Av{\nubf}=(\partial\bX-\partial\bX^\top)\mubf\cdot \Av{\nubf}
\end{align*}
giving the desired result
\begin{align*}
\frac{d}{dt}\sphericalangle(\mubf^t,\bP_{\taubf^t}^{\perp}(\Av{\nubf}))|_{t=0}= -\frac{(\partial\bX-\partial\bX^\top)\mubf\cdot \Av{\nubf}}{\sqrt{1-(\mubf\cdot\Av{\nubf})^2}}.
\end{align*}

\section{Angle equivalence}
\label{sec:angle_equivalence}
\begin{lemma}
	\label{lem:additive_angle}
	Let $\Va,\Vb\in \VR^3$ with $\|\Va\|=\|\Vb\|=1$. Further let $\Vc\in \VR^3$ with $\|\Vc\|=1$ and $\Vc$ ``lies between'' $\Va$ and $\Vb$, i.e., there exists $t\in [0,1]$ such that $\Vc \in \mathrm{span}\{ t\,\Va+(1-t)\Vb \}$. Then 
	\begin{align}
	\arccos(\Va\cdot \Vb) = \arccos(\Va\cdot \Vc) + \arccos(\Vc\cdot \Vb).
	\end{align}
\end{lemma}
\begin{proof}
	As $\Vc\in \text{span}\{ \Va,\Vb \}$ we rotate the coordinate system such that all vectors lie w.l.o.g. in the x-y-plane, i.e., $\Va$, $\Vb$, $\Vc\in \VR^2$. There holds $\Va\cdot \Vb = \mathcal{R}(a\bar{b})$, where we identified $\Va,\Vb$ with complex numbers, $\bar{(\cdot)}$ denotes the complex conjugation, and $\mathcal{R}(\cdot)$ the real part. As the vectors are normalized we have
	\begin{align*}
	a = e^{i\alpha},\qquad b = e^{i\beta}, \qquad c = e^{i\gamma},\qquad\qquad i^2=-1,\quad \alpha,\beta,\gamma \in [0,2\pi).
	\end{align*}
	W.l.o.g. assume that $\alpha>\beta$. The condition that $\Vc$ lies between $\Va$ and $\Vb$ is then equivalent to $\alpha \geq \gamma \geq \beta$. Thus, there holds with $\mathcal{R}(a\bar{b})=\mathcal{R}(e^{i(\alpha-\beta)})=\cos(\alpha-\beta)$
	\begin{align*}
	\arccos(\Va\cdot \Vb) = \arccos(\cos(\alpha-\beta))\overset{\alpha\geq\beta}{=}\alpha-\beta
	\end{align*}
	and the right-hand side is
	\begin{align*}
	\arccos(\Va\cdot \Vc) + \arccos(\Vc\cdot \Vb) &\overset{\alpha\geq\gamma\geq\beta}{=}\alpha-\gamma+\gamma-\beta= \alpha -\beta.
	\end{align*}
\end{proof}

\section{Supplementary material}
\label{sec:suppl_material}
In this supplementary material we present and describe the basic algorithm of our proposed method to solve the Canham--Helfrich--Evans minimization problem including area and volume constraints. For a better presentation we split the code into several snippets. Summing them up an executable file running in NGSolve\footnote{www.ngsolve.org} \cite{Sch14} is obtained. Note that NumPy\footnote{www.numpy.org} is required to execute the file. For further details concerning shape optimization in NGSolve we refer to \cite{KSNS20}.

\begin{lstlisting}[caption=Include packages and define parameters.,label=lst:parameters]
from ngsolve import *
from netgen.csg import *
from netgen.meshing import MeshingStep
from math import pi
import numpy as np

autodiff = False # use automatic shape derivative?

order    = 1     # polynomial order
maxh     = 0.2   # mesh-size
v        = 0.7   # goal reduced volume
H0       = 0     # spontaneous curvature

kb  = 0.01         # bending constant
kv  = Parameter(1) # penalty for volume
kag = Parameter(2) # penalty for global area
kal = Parameter(1) # penalty for local area

xvec  = CF( (x,y,z) )  # identity CoefficientFunction

nsurf = specialcf.normal(3)      # outer normal vector
tang  = specialcf.tangential(3)  # edge tangential vector
nel   = Cross(nsurf,tang)        # co-normal vector
\end{lstlisting}

First, we include necessary packages and define several parameters as the used polynomial order and mesh-size. Additionally the identity function as well as the outer normal, edge tangential, and co-normal vector used later are declared.

\begin{lstlisting}[caption=Define mesh.,label=lst:mesh]
geo = CSGeometry()
a = 1.1017
b = 0.95
geo.Add(Ellipsoid(Pnt(0,0,0), Vec(a,0,0), Vec(0,b,0), Vec(0,0,b)))
mesh = Mesh(geo.GenerateMesh(maxh=maxh, perfstepsend=MeshingStep.MESHSURFACE))
mesh.Curve(order)
Draw(mesh)
\end{lstlisting}

In Listing~\ref{lst:mesh} we generate a prolate initial surface shape with given mesh-size and curve it appropriately.

\begin{lstlisting}[caption=Preparation for are and volume constraint.,label=lst:area_vol]
A0  = Integrate(1, mesh, BND)                # inital area
V0  = Integrate(1/3*xvec*nsurf, mesh, BND)   # initial volume
V   = v*4/3*pi*(A0/(4*pi))**(3/2)            # goal volume

At0 = GridFunction(SurfaceL2(mesh, order=0)) # initial local areas
At0.vec.FV().NumPy()[:] = Integrate(1, mesh, BND, element_wise=True)

At  = GridFunction(SurfaceL2(mesh, order=0)) 
At.vec.data = At0.vec    # initialize current local areas
A_cur = Parameter(A0)    # initialize current area
V_cur = Parameter(V0)    # initialize current volume
\end{lstlisting}

Next, we compute the initial area and enclosed volume of the mesh and define the goal volume by means of the reduced volume parameter. Further, the area of each triangle is stored for the local area stabilization constraint. The involved \texttt{SurfaceL2} space of order= $0$ consists of a constant value per surface element, which is stored in the \texttt{GridFunction} object. If required the \texttt{GridFunction} can be drawn visualizing the local areas.

\begin{lstlisting}[caption={Energy, Cost, and CostDiff auxiliary function.},label=lst:energy_cost_costdiff]
# compute normalized Canham-Helfrich-Evans energy
def Energy(kappa, gfset):
	mesh.SetDeformation(gfset)
	energy = Integrate(2*kb*(1/2*kappa-H0)**2, mesh, BND)
	mesh.UnsetDeformation()
	return energy/(8*pi*kb)

# compute costs w.r.t. bending energy and area/volume constraints
def Cost(kappa):
	# compute current areas, volume, and bending energy
	A_cur.Set(Integrate(1, mesh, BND))
	At.vec.FV().NumPy()[:] = Integrate(1, mesh, BND, element_wise=True)
	V_cur.Set(1/3*Integrate(xvec*nsurf, mesh, BND))
	bending = Integrate(2*kb*(1/2*kappa-H0)**2, mesh, BND)

	constr_At =  kal.Get()*(np.square(At.vec.FV().NumPy()-At0.vec.FV().NumPy())/At0.vec.FV().NumPy()).sum()
	constraint = kag.Get()*(A_cur.Get()-A0)**2/A0 + kv.Get()*(V_cur.Get()-V)**2/V + constr_At

	return  bending + constraint

# return shape derivative of Cost functional
def CostDiff(kappa, PSI):
	if autodiff: # automatic shape derivative
		bending = (2*kb*(1/2*kappa-H0)**2*ds).DiffShape(PSI)
		constr  = (2*(kag*(A_cur-A0)/A0 + kal*(At-At0)/At0 + kv*(V_cur-V)/V/3*xvec*nsurf)*ds).DiffShape(PSI)
	else: # manual shape derivative
		tangdet = div(PSI).Trace() # surface divergence
		bending = 2*kb*tangdet*(1/2*kappa-H0)**2*ds
		constr  = 2*(kag*(A_cur-A0)/A0*tangdet + kv*(V_cur-V)/V*PSI*nsurf + kal*(At-At0)/At0*tangdet)*ds
	return bending + constr
\end{lstlisting}

The function \texttt{Energy} computes the normalized Canham--Helfrich--Evans bending energy. Additionally to the curvature $\kappa$ also a \texttt{GridFunction} object storing the current displacement information of the mesh is handed over. With the method \texttt{SetDeformation} of the mesh all integration procedures are performed as if we would consider the deformed mesh according to the displacement until the \texttt{UnsetDeformation} command is used.\newline
The \texttt{Cost} function computes the current Canham--Helfrich--Evans bending energy together with the area and volume constraints. Note, that we use NumPy to avoid slow Python for-loops.\newline
Next, we need the shape derivative of the cost functional in direction $\Psi$ defined later. We can either use the build-in automatic shape differentiation procedure denoted by \texttt{DiffShape(PSI)} \cite{KSNS20} or by manually computing the derivatives. Note that the \texttt{ds} object indicates that the integration will be performed on the surface.

\begin{lstlisting}[caption={Equation and EquationDiff auxiliary function.},label=lst:equ_equdiff]
# return equation
def Equation(kappa, sigma, nav):
	return (kappa*sigma + Trace(Grad(nsurf))*sigma)*ds + (pi/2-acos(nel*nav))*sigma*ds(element_boundary=True)

# return shape derivative of equation
def EquationDiff(kappa, sigma, nav, PSI):
	if autodiff: # automatic shape derivative
		return Equation(kappa, sigma, nav).DiffShape(PSI)

	# manual shape derivative
	tangdet = div(PSI).Trace()
	return (tangdet*(kappa*sigma + Trace(Grad(nsurf))*sigma) + ((Grad(PSI).Trace()*Grad(sigma))*nsurf - Trace(Grad(PSI).Trace().trans*Grad(nsurf))*sigma))*ds + (((Grad(PSI).Trace()*tang)*tang*(pi/2-acos(nel*nav)) + 1/sqrt(1-InnerProduct(nav,nel)**2)*((Grad(PSI).Trace() - Grad(PSI).Trace().trans)*nel)*nav)*sigma -(Grad(PSI).Trace()*nel)*nsurf*sigma)*ds(element_boundary=True)
\end{lstlisting}

In Listing~\ref{lst:equ_equdiff} we define the equation of the state problem, where the distributional curvature gets lifted to the auxiliary curvature field $\kappa$. For the needed shape derivative we again can directly differentiate it. The analytical computations presented in the paper are quite involved, however, manageable.

\begin{lstlisting}[caption={(Bi-)Linear forms for (adjoint) state and shape gradient problem.},label=lst:bilinearforms]
VEC = VectorH1(mesh, order=order)    # vector-Lagrange finite element space
PHI,PSI = VEC.TnT()                  # trial- and testfunction

gfX   = GridFunction(VEC)            # store shape gradient
gfset = GridFunction(VEC)            # store current displacement

fes = H1(mesh, order=order)          # Lagrange finite element space
kappa, sigma = fes.TnT()             # trial- and testfunction

gfkappa = GridFunction(fes)          # store curvature (state)
gfsigma = GridFunction(fes)          # store moments (adjoint state)

# space for averaging normal vector
fesfacet = VectorFacetSurface(mesh, order=order-1)
gfh = GridFunction(fesfacet)           # store the averaged normal vector
nav = Normalize(CF( gfh.components ))  # normalize averaged normal vector

# left-hand side for solving (adjoint) state problem
a = BilinearForm(fes, symmetric=True)
a += kappa*sigma*ds
a.Assemble()  # assemble and invert as preparation
inva = a.mat.Inverse(freedofs=fes.FreeDofs(), inverse="sparsecholesky")

# right-hand side for state problem (curvature)
fa = LinearForm(fes)
fa += -Trace(Grad(nsurf))*sigma*ds - (pi/2-acos(nel*nav))*sigma*ds(element_boundary=True)

# right-hand side for adjoint state problem
dCostdu = LinearForm(fes)
dCostdu += 2*kb*(1/2*gfkappa-H0)*sigma*ds

# left-hand side for shape optimization gradient method
aX = BilinearForm(VEC, symmetric=True)
aX += (InnerProduct(Grad(PHI).Trace(), Grad(PSI).Trace()) + 1e-10*PHI*PSI)*ds
aX.Assemble()  # assemble and invert as preparation
invaX = aX.mat.Inverse(VEC.FreeDofs(), inverse="sparsecholesky")

# right-hand side for shape optimization gradient method
fX = LinearForm(VEC)
fX += CostDiff(gfkappa, PSI).Compile()
fX += EquationDiff(gfkappa, gfsigma, nav, PSI).Compile()
\end{lstlisting}

After having defined the important functions we can focus on the solving algorithm. For the displacement field and the shape gradient we use a vector-valued Lagrange finite element space and define the symbolic trial- and testfunction objects $\Phi$ and $\Psi$. The shape gradient and displacement field themselves get stored as a \texttt{GridFunction} object of the $H^1$-conforming finite element space.\newline
The independent curvature field $\kappa$ as well as the Lagrange multiplier (adjoint state) $\sigma$ are discretized by scalar Lagrange elements and get stored in the corresponding \texttt{GridFunctions}. To compute the used averaged normal vector a \texttt{SurfaceVectorFacet} finite element space is used living only on the edges (the skeleton) of the triangulation. The corresponding \texttt{GridFunction} needs to be normalized to measure the correct angle. As discussed in the paper we directly neglected the projection operator $\bP_{\taubf}^\perp(\cdot)$ to reduce the expressions gaining performance, which, however, could be implemented with the following line.
\begin{lstlisting}[caption={Projected averaged normal vector.},label=lst:proj_av_nv]
nav = Normalize( CF( gfh.components ) - (tang*CF( gfh.components ))*tang )
\end{lstlisting}
To solve the state and adjoint state problem a mass matrix is assembled and inverted with the build-in ``sparsecholesky'' solver. The right-hand side for the state problem gets represented by the \texttt{LinearForm} \texttt{fa}, which is written completely symbolically. We note that \texttt{$\dots$*ds(element\_boundary=True)} corresponds to the integral $\sum_{T\in\T_h}\int_{\partial T}\dots\,d\gamma$. the right-hand side of the adjoint state problem is given by the variation of the Canham--Helfrich--Evans energy with respect to the curvature $\kappa$.\newline
For the shape optimization gradient method we define, assemble, and invert the (regularized) $H^1$-scalar product. The right-hand side is given by the shape derivative of the state equation as well as the cost functional. Note that the \texttt{Compile()} statement optimizes the internally generated expression tree of symbolic expressions to gain evaluation performance.

\begin{lstlisting}[caption={Solve current (adjoint) state problem.},label=lst:solvepde]
def solvePDE():
	# average current normal vector
	gfh.Set(nsurf, dual=True, definedon=mesh.Boundaries(".*")) 
	# solve adjoint and state equation
	a.Assemble()
	fa.Assemble()
	inva.Update()
	gfkappa.vec.data = inva*fa.vec
	dCostdu.Assemble()
	gfsigma.vec.data = -inva*dCostdu.vec
	return
\end{lstlisting}

As we need to solve the (adjoint) state problem in every optimization step, we summarize them in Listing~\ref{lst:solvepde}. First, the new normal vector is averaged and then the problems for $\kappa$ and $\sigma$ are solved.

\begin{lstlisting}[caption={Draw fields and optimization parameters.},label=lst:draw_opt_par]
solvePDE()
Draw(gfX, mesh, "gfX")
Draw(gfsigma, mesh, "adjoint")
Draw(gfkappa, mesh, "state")
Draw(gfset, mesh, "displacement")
Draw(Norm(0.5*gfkappa), mesh, "mean")
SetVisualization(deformation=True)

iter_max            = 1000   # maximal number of optimization steps
scale_init          = 0.025  # initial step-size
scale_max           = 0.1    # maximal step-size
scaleIncreaseFactor = 1.00   # increasing factor after accepted step
tol_scale           = 1e-11  # tolerance for minimal step-size 
tol_gfX             = 1e-12  # tolerance for shape gradient
tol_J               = 1e-10  # tolerance for costs
normGFX_start       = None   # store initial shape gradient

isConverged = False
Jold        = 0                 # store previous costs
gfsettmp    = GridFunction(VEC) # store temporary mesh displacement
\end{lstlisting}

In Listing~\ref{lst:draw_opt_par} all quantities are drawn for visualization and the mesh will be visually deformed to the current shape by the \texttt{SetVisualization(deformation=True)} command.\newline
Then, several self-explaining optimization parameters are defined.

\begin{lstlisting}[caption={Shape optimization gradient method loop.},label=lst:shape_opt_loop]
with TaskManager():
	solvePDE()
	Jnew = Cost(gfkappa)
	Jold = Jnew
	
	print("it init", 'cost', Jnew )
	scale = scale_init 

	for k in range(iter_max):
		# solve (adjoint) state problem and prepare shape derivative
		mesh.SetDeformation(gfset)
		solvePDE()
		aX.Assemble()
		fX.Assemble()
		mesh.UnsetDeformation()

		invaX.Update()
		gfX.vec.data = invaX * fX.vec  # next shape gradient

		currentNormGFX = Norm(gfX.vec)
		if k == 0: normGFX_start = currentNormGFX

		while True: # line-search
			if scale < tol_scale or currentNormGFX < normGFX_start*tol_gfX or Jnew < tol_J: # converged?
				isConverged = True
				break

			# guess for next mesh displacement
			gfsettmp.vec.data = gfset.vec - scale * gfX.vec

			mesh.SetDeformation(gfsettmp)
			solvePDE()
			Jnew = Cost(gfkappa)
			mesh.UnsetDeformation()

			if Jnew <= Jold: # accept step?
				Jold = Jnew
				print("----------it", k, 'scale', scale, 'cost', Jnew )
				gfset.vec.data = gfsettmp.vec
				scale = min(scale_max, scale*scaleIncreaseFactor)
				break
			else: # if not, reduce step-size
				scale = scale / 2

		print("--------------------------------------||gfX||", currentNormGFX)

		Redraw() # redraw solutions

		if isConverged:
			print("converged with J = ", Jnew, ", ||gfX||=", currentNormGFX, ", scale = ", scale)
			break
\end{lstlisting}

Finally, the shape gradient Algorithm~5.1 in the paper is presented in Listing~\ref{lst:shape_opt_loop}. First, we activate the build-in \texttt{TaskManager} to perform the following assembling and inversion processes in (thread-)parallel, solve the (adjoint) state problem on the initial shape and evaluate the cost function. In every optimization step we re-compute the (adjoint) state problem as well as the next shape gradient on the current configuration of the mesh. Therefore, analogously as in Listing~\ref{lst:energy_cost_costdiff} in the \texttt{Energy} function, we use the \texttt{SetDeformation} method, where the \texttt{GridFunction} object \texttt{gfset} is used as input having stored the displacement information how the initial mesh has to be deformed to obtain the current shape. Next, a line-search is performed to guarantee that the final gradient step non-increases the cost functional. Therefore, the temporary object \texttt{gfsettmp} saves the previous displacement plus a scaled gradient step with the step-size parameter \texttt{size}. After computing the cost functional on the temporary configuration the decrease of the cost functional is checked. If accepted, the step-size is increased by an factor, otherwise the step-size gets halved and a gradient step with the reduced size is tried until the cost functional decreases or the step-size becomes to small yielding a break down of the algorithm.

\begin{lstlisting}[caption={Postprocessing.},label=lst:postprocess]
# compute final area, volume, reduced volume, cost, and energy
mesh.SetDeformation(gfset)
Vnew = Integrate(1/3*xvec*nsurf, mesh, BND)
Anew = Integrate(1, mesh, BND)
vnew = Vnew/(4/3*pi*(Anew/(4*pi))**(3/2))
cost = Cost(gfkappa)
mesh.UnsetDeformation()

print("cost   = ", cost)
print("energy = ", Energy(gfkappa,gfset))
print("Vnew   = ", Vnew)
print("Anew   = ", Anew)
print("vnew   = ", vnew)
\end{lstlisting}

After the algorithm determinated, the quantities as area and volume are updated. Note that the final reduced volume $v_{\mathrm{new}}$ does not necessarily need to coincide with the goal reduced volume $v$ from Listing~\ref{lst:parameters} as the penalty method for the area and volume constraints is used.
\bibliographystyle{acm}
\bibliography{cites}

\end{document}